\documentclass[11pt]{amsart}   	
\pdfoutput=1

\usepackage[margin=1in]{geometry}

\usepackage{amsmath,amssymb,amsfonts,amsthm}
\usepackage{hyperref}
\usepackage{mathtools}
\usepackage[capitalise]{cleveref}
\usepackage{tikz-cd}
\usepackage{tikz}
\usepackage{enumitem}
\usepackage{makecell}
\setcellgapes{4pt}
\usepackage{pdfpages} 
\usepackage{enumitem}
\usepackage{subcaption}
\usepackage{multicol}

\usepackage{enumitem}
\usepackage{tabstackengine}

\usepackage{comment}
\usepackage{pdflscape}

\newcommand{\add}{\mathsf{add}}
\newcommand{\Filt}{\mathsf{Filt}}
\newcommand{\Gen}{\mathsf{Gen}}

\newcommand{\Cogen}{\mathsf{Cogen}}

\newcommand{\ind}{\mathsf{ind}}

\newcommand{\wide}{\mathsf{wide}}
\newcommand{\tors}{\mathsf{tors}}
\newcommand{\ftors}{\mathsf{f}\text{-}\mathsf{tors}}
\newcommand{\torf}{\mathsf{torf}}
\newcommand{\str}{\mathsf{s}\tau\text{-}\mathsf{rigid}}
\newcommand{\stt}{\mathsf{s}\tau\text{-}\mathsf{tilt}}
\newcommand{\taut}{\tau\text{-}\mathsf{tilt}}
\newcommand{\mods}{\mathsf{mod}}
\newcommand{\proj}{\mathsf{proj}}

\newcommand{\cone}{\mathsf{Cone}}
\newcommand{\Hasse}{\mathsf{Hasse}}

\newcommand{\A}{\mathbb{A}}
\renewcommand{\L}{\Lambda}

\newcommand{\bcomp}{\mathrm{B}}
\newcommand{\ccomp}{\mathrm{C}}
\newcommand{\eproj}{\mathrm{P}}

\newcommand{\bcomplement}[1]{\overline{\mathrm{B}}(#1)}
\newcommand{\ccomplement}[1]{\overline{\mathrm{C}}(#1)}

\newcommand{\Ccal}{\mathcal{C}}

\newcommand{\Dcal}{\mathcal{D}}

\newcommand{\Tcal}{\mathcal{T}}

\newcommand{\Wcal}{\mathcal{W}}
\newcommand{\Xcal}{\mathcal{X}}

\DeclareMathOperator{\Hom}{\mathrm{Hom}}
\DeclareMathOperator{\Ext}{\mathrm{Ext}}
\DeclareMathOperator{\im}{\mathrm{im}}

\DeclareMathOperator{\coker}{\mathrm{coker}}
\DeclareMathOperator{\End}{\mathrm{End}}
\DeclareMathOperator{\soc}{\mathrm{soc}}
\DeclareMathOperator{\topp}{\mathrm{top}}

\DeclareMathOperator{\rad}{\mathrm{rad}}

\newcommand{\comp}[2]{\str_{#1}#2}

\newtheorem{theorem}{Theorem}[section]
\newtheorem{corollary}[theorem]{Corollary}
\newtheorem{lemma}[theorem]{Lemma}

\newtheorem{proposition}[theorem]{Proposition}

\theoremstyle{definition}
\newtheorem{definition}[theorem]{Definition}
\newtheorem{definitionproposition}[theorem]{Definition-Proposition}

\newtheorem{example}[theorem]{Example}

\newtheorem{remark}[theorem]{Remark}

\usepackage[style=alphabetic,sorting=nty,sortlocale = nn\_NO]{biblatex}

\title{Mutating ordered $\tau$-rigid modules with applications to Nakayama algebras}

\date{}

\author{Aslak B. Buan}
\address{Department of Mathematical Sciences, Norwegian University of Science and Technology (NTNU), 7491 Trondheim, NORWAY}
\email{aslak.buan@ntnu.no}

\author{Maximilian Kaipel}
\address{Abteilung Mathematik, Department Mathematik/Informatik der Universität
zu Köln, Weyertal 86-90, 50931 Cologne, GERMANY}
\email{mkaipel@uni-koeln.de}

\author{Håvard U. Terland}
\address{Department of Mathematical Sciences, Norwegian University of Science and Technology (NTNU), 7491 Trondheim, NORWAY}
\email{havard.u.terland@ntnu.no}

\begin{document}

\begin{abstract}
A mutation operation for $\tau$-exceptional sequences of \sloppy modules over any finite-dimensional algebra was recently introduced, generalising the mutation for exceptional sequences of modules over hereditary algebras. We interpret this mutation in terms of TF-ordered $\tau$-rigid modules, which are in bijection with $\tau$-exceptional sequences. As an application we show that the mutation is transitive for Nakayama algebras, by providing an explicit combinatorial description of mutation over this class of algebras.
\end{abstract}

\thanks{The main ideas of this project were conceived during a research visit of the second named author to NTNU Trondheim. MK would like to thank ABB and HUT for their hospitality and a memorable week. The authors were supported by grant number FRINAT 301375 from the Norwegian Research Council. MK was also supported by the Deutsche Forschungsgemeinschaft (DFG, German Research Foundation) -- Project-ID 281071066 -- TRR 191.}

\maketitle

\section{Introduction}

More than thirty years ago the study of exceptional sequences in both triangulated \cite{Bondal1989, Gorodentsev1988, GorodentsevRudakov, Rudakov90} and abelian categories \cite{cbw92,ringel_exceptional} was initiated. Fundamentally, an exceptional sequence consists of exceptional objects, which are objects whose endomorphism rings are division rings and which have no self-extensions of any degree. Moreover, the objects in an exceptional sequence satisfy a 
right-to-left vanishing property of $\Hom$- and $\Ext$-groups.

In a triangulated category, an exceptional sequence is called \textit{complete} if the smallest triangulated subcategory generated by its exceptional objects is the whole category. In this case, there exists a Braid group action by mutation on the set of such sequences \cite{Bondal1989, Gorodentsev1989} and it was conjectured in \cite{BondalPolishchuk} that this action is always transitive. This had been the focus of much attention \cite{IshiiOkawaUehara2021, KussinMeltzer2002,Meltzer95, ChangSchroll2023} until a counterexample was found in \cite{ChangHaidenSchroll}. 

In an abelian category an exceptional sequence is called \textit{complete} if the objects generate the derived category as above. Generally, the existence of complete exceptional sequences is not guaranteed, but has important consequences, see e.g. \cite{ChangSchroll2023,HillePloog2019,Krause2017}. We are particularly interested in categories of modules of finite-dimensional algebras. It is a classical result that complete exceptional sequences exist for hereditary algebras and that the action of the Braid group is transitive \cite{cbw92,ringel_exceptional}.

However, the module category of an arbitrary finite-dimensional algebra generally does not admit complete exceptional sequences. It is therefore natural to consider generalisations of exceptional sequences \cite{Araya99,tauexceptional_buanmarsh,Sen2021}. In particular, the generalisation proposed in \cite{tauexceptional_buanmarsh} utilises $\tau$-tilting theory to ensure the existence of complete $\tau$-exceptional sequences for all finite-dimensional algebras. The $\Ext$-vanishing condition in this generalisation is modified using a key idea of $\tau$-tilting theory which replaces the condition $\Ext^1(N,M) = 0$ by $\Hom(M, \tau N) = 0$, since the two are equivalent for hereditary algebras but generally the second is by \cite{auslandersmalo81} equivalent to the condition $\Ext^1(M, N')=0$ for all modules $N'$ \textit{generated} by $N$, meaning there is an epimorphism $N^r \to N'$ for some integer $r$. We collect all modules generated by $N$ in the subcategory \[\Gen N = \{X \mid \text{there is an epimorphism } N^r \to X \text{ for some integer } r\}.\]

It is natural to ask if there exists a well-defined action of mutation on complete $\tau$-exceptional sequences and whether such an action is transitive and satisfies Braid group relations \cite{Buan_Marsh_2023}. To answer these questions it is useful to consider the more general \textit{signed} $\tau$-exceptional sequences whose additional structure is a positive or negative sign on certain relatively projective modules in a $\tau$-exceptional sequence. Whereas the unsigned $\tau$-exceptional sequences label chains in the poset of wide subcategories \cite{BarnardHanson2022exc}, the signed $\tau$-exceptional sequences give rise to an even richer structure. In fact, they correspond to factorisations of morphism into irreducible ones in the $\tau$-cluster morphism category \cite{IgusaTodorov2017,BuanMarsh2018, buan2023perpendicular}, see also \cite{børve2023,borve2024silt} for more general settings, whose classifying space has been shown to be a $K(\pi,1)$ space in many cases \cite{BarnardHanson2022,HansonIgusaPW2SMC,HansonIgusa2021,IgusaTodorov2017,IgusaTodorov22,Kaipelcatpartfan,KaipelTors}.

An important alternative viewpoint of signed $\tau$-exceptional sequences establishes a close connection with the $\tau$-tilting pairs central to $\tau$-tilting theory: There exists a bijection between signed $\tau$-exceptional sequences of length $r$ and ordered $\tau$-rigid pairs with $r$ non-isomorphic indecomposable direct summands \cite{tauexceptional_buanmarsh}. Alongside establishing a close connection between $\tau$-exceptional sequences and stratifying systems, the authors of \cite{mendozatreffinger_stratifyingsystems} characterise those ordered $\tau$-rigid pairs which correspond to the unsigned $\tau$-exceptional sequences. We will refer to those as \textit{TF-ordered} $\tau$-rigid modules.

Recently, a positive answer to the existence of a mutation of $\tau$-exceptional sequences was provided in \cite{BHM2024}. Similar to the mutation of exceptional sequences of hereditary algebras, which it generalises, the mutation is defined on $\tau$-exceptional pairs. This provides the starting point for the present paper, in which we begin by translating the mutation rules defined in \cite{BHM2024} to a corresponding rule for mutating TF-ordered $\tau$-rigid modules. A key role in this description is played by what we refer to as the $V$-map, see \cref{def:Vmap} and the $E$-map, see \cref{defn:Emap}.

The collection of all $\tau$-tilting pairs has a natural structure of a partially-ordered set via a process of mutation \cite{tau}. Roughly speaking, the $V$-map with respect to a $\tau$-rigid pair $(M,P)$ is a bijection $V_{(M,P)}$ from the indecomposable direct summands of the minimal $\tau$-tilting pair to those of the maximal $\tau$-tilting pair having $(M,P)$ as a direct summand in the partially-ordered set. It was used implicitly in \cite{tauexceptional_buanmarsh} to define the $E$-map with respect to a $\tau$-rigid pair $(M,P)$, which is a map $E_{(M,P)}$ from the $\tau$-rigid pairs whose direct sum with $(M,P)$ forms a $\tau$-rigid pair to the 
relative $\tau$-rigid pairs of a wide subcategory known as the $\tau$-tilting reduction $J(M,P)$, as defined in \cite{jassoreduction, dirrt}. Our first main result is the following formulation of the mutation of TF-ordered $\tau$-rigid modules with two indecomposable direct summands, which coincides with that for the mutation of $\tau$-exceptional pairs of \cite{BHM2024}.

\begin{theorem}(\cref{prop:TFadmissiblebehaviour})\label{thm:introthm1}
    Let $B \oplus C$ be a TF-ordered $\tau$-rigid module. If the corresponding $\tau$-exceptional pair is left regular, see \cite[Def. 3.11]{BHM2024} then the left mutation described in \cite[Sec. 4]{BHM2024} corresponds to the following mutation of TF-ordered $\tau$-rigid modules:
     \[\overline{\varphi}(B \oplus C) = \begin{cases}
    V_{E_C(B)}(C[1]) \oplus E_C(B) &\text{if $C$ is projective,} \\
    V_B(C) \oplus B & \text{if $C$ is generated by $B$,} \\
C \oplus B & \text{otherwise,}\end{cases}\]
where we write $C[1]$ to mean the $\tau$-rigid pair $(0,C)$.
			\end{theorem}

While the mutation of $\tau$-exceptional sequences and thus of TF-ordered $\tau$-tilting modules does not usually satisfy the Braid group relations, see \cite[Exmp. 8.3]{BHM2024} and \cref{exmp:nonbraid}, a partial answer was given to the question of transitivity of this mutation in the case where the algebras has only two isomorphism classes of simple modules \cite[Thm. 0.6]{BHM2024}. In this case, the mutation graph of the mutation of complete $\tau$-exceptional sequences can be obtained from the mutation graph of $\tau$-tilting modules. We apply \cref{thm:introthm1} to prove transitivity of the mutation of complete $\tau$-exceptional sequences for the class of uniserial algebras, also known as Nakayama algebras. 

\begin{theorem}(\cref{thm:transitivemut})\label{thm:introthm2}
    Mutation of TF-ordered $\tau$-tilting modules, and thus of complete $\tau$-exceptional sequences, is transitive for Nakayama algebras.
\end{theorem}

Our proof of this result relies first on a concrete description of both of the minimal and maximal completions of an indecomposable $\tau$-rigid module $M$ to a $\tau$-tilting pair, the former is called the \textit{co-Bongartz completion} $\ccomp(M)$ and is described in \cref{lem:cocompletion} whereas the latter is called the \textit{Bongartz completion}  $\bcomp(M)$ and is described in \cref{prop:Bongartzcompletion}. This allows us to explicitly compute the $V$-map arising in \cref{thm:introthm1} and thus obtain a precise formula for the mutation of TF-ordered $\tau$-tilting modules for Nakayama algebras. To describe the mutation of a TF-ordered $\tau$-rigid module $B \oplus C$ we distinguish between the following cases:
\begin{multicols}{2} 
\begin{enumerate}[leftmargin=2cm]
    \item[(TF-1)] $C \in \proj \L$;
    \begin{enumerate}
        \item[(TF-1a)] and $\Hom(C,B) = 0$;
        \item[(TF-1b)] and $\Hom(C,B) \neq 0$;
    \end{enumerate}
    \bigskip
    \item[(TF-3)] $C \not \in \add (\bcomp(B) \oplus \ccomp(B))$;
\columnbreak
    \item[(TF-2)] $C \in \add\ccomp(B)$;
    \begin{enumerate}
        \item[(TF-2a)] and $B \not \in \proj \L$;
        \item[(TF-2b)] and $B \in \proj \L$;
    \end{enumerate}
    \bigskip
    \item[(TF-4)] $C \in \add \bcomp(B) \setminus \proj \L$.
\end{enumerate}
\end{multicols}
Note that $C \in \add \ccomp(B)
$ if and only if $C \in \Gen B$ (see \cite[Lem. 4.4]{tauexceptional_buanmarsh}). Using this, in combination with \cite[Lem. 1.7]{BHM2024}, it follows that the cases above are mutually 
exclusive.

We emphasise that Case TF-1 corresponds to the case where $C$ is projective in \cref{thm:introthm1}, Case TF-2 corresponds to the case were $C$ is generated by $M$ in \cref{thm:introthm1}, and Case TF-3 corresponds to the remaining case in \cref{thm:introthm1}. Case TF-4 corresponds to the left mutable but left irregular case. With these distinctions we are able to give the following explicit formulas for the mutation of TF-ordered modules for Nakayama algebras.

\begin{theorem}\label{thm:introNakmutation}
    Let $\L$ be a Nakayama algebra and let $B \oplus C$ be a TF-ordered $\tau$-rigid module. Then the left mutation of TF-ordered modules is given by
    \[ \overline{\varphi}(B \oplus C) = \begin{cases}
        C \oplus B & \text{ in Case TF-1a and Case TF-3,}\\
        B \oplus f_C B & \text{ in Case TF-1b and Case TF-4,}\\
        \rad^k B \oplus B & \text{ in Case TF-2a,} \\ 
        P(\rad^k B) \oplus B & \text{ in Case TF-2b,}\\
    \end{cases}\]
    where $C \cong B/\rad^k B$ for some $k \in \{1, \dots, \ell(B)\}$ in Case TF-2, where $P(\rad^k B) \to \rad^k B$ is the projective cover of $\rad^k B$, and where $f_C$ is the torsion-free functor of the torsion pair $(\Gen C, C^\perp)$. 
\end{theorem}
\begin{proof}
    Case TF-1 is \cref{prop:Nakayamacase1}, Case TF-2 is \cref{lem:Nakayamacase2} and Case TF-4 is \cref{thm:Nakayamacase4}, while TF-3 is directly from \cref{thm:introthm1}.
\end{proof}

With the explicit description of \cref{thm:introNakmutation} we first prove that any two TF-orders of the same $\tau$-tilting module are related by a sequence of mutations of TF-ordered modules as in \cref{thm:introNakmutation}, see \cref{lem:TFdecsconnected}. Then we use \cref{thm:introNakmutation} to show that whenever two $\tau$-tilting modules are related by a mutation of $\tau$-tilting modules as defined in \cite{tau}, then some TF-orders of these two modules are related by a sequence of mutations of TF-ordered modules, see \cref{prop:mutationmutation}. Since Nakayama algebras are $\tau$-tilting finite, their mutation graphs of the mutation of $\tau$-tilting modules are connected. Hence by alternating these two processes of obtaining sequences of mutations of TF-ordered $\tau$-tilting modules, we are able to connect any two TF-ordered $\tau$-tilting modules, thus showing that the mutation is transitive.

The structure of the paper is as follows: In Section 2 we provide all necessary background on $\tau$-tilting theory, $\tau$-exceptional sequences and TF-ordered modules. In Section 3 we investigate how the mutation of $\tau$-exceptional sequences may be described using the alternate viewpoint of TF-ordered modules. In Section 4 we describe the Bongartz completion and co-Bongartz completion for Nakayama algebras and prove \cref{thm:introNakmutation}. In Section 5 we demonstrate how this result can be used to prove the transitivity of the mutation for Nakayama algebas and thus \cref{thm:introthm2}. In Section 6 we present many examples of the theory developed in this paper. Moreover, we use the geometric model for $\tau$-tilting modules of \cite{Adachi2016} to demonstrate examples involving Nakayama algebras.

\section{Background}

All algebras in this paper are assumed to be finite-dimensional over an arbitrary fixed field $K$. Let $\Lambda$ be a finite-dimensional algebra, and let $\mods \L$ we denote the category of finite-dimensional left $\Lambda$-modules. We will always assume modules to be basic, and define $|M|$ to be the number of indecomposable direct summands of a module $M$. Throughout this document, let $n$ denote the rank of $\Lambda$, that is, the number of simple $\Lambda$-modules,
up to isomorphism.
Let $\mathcal{X} \subseteq \mods \Lambda$ be a full subcategory. A morphism $M \to N$ is called a \textit{left $\Xcal$-approximation} of $M$ if the induced morphism $\Hom(N,X) \to \Hom(M,X)$ is an epimorphism for all $X \in \Xcal$. Right approximations are defined dually. A map $f: M \to N$ is called \textit{left minimal} if every endomorphism $g \in \End(N)$ satisfying $g \circ f = f$ is an isomorphism. A subcategory $\mathcal{C} \subseteq \mods \Lambda$ is called \textit{functorially finite} if every object in $\mods \Lambda$ has a left $\mathcal{C}$-approximation and a right $\mathcal{C}$-approximation.
We let $\add M$ denote the additive closure of a module $M$. Note that $\add M$ is automatically functorially finite for any module $M$. 

Given a subcategory $\Ccal \subseteq \mods \Lambda$, let us define the $\Hom$-orthogonal categories as follows:
\begin{align*}
    \mathcal{C}^{\perp} &= \{M \in \mods \Lambda \mid \forall X \in \mathcal{C}, \Hom_\Lambda(X,M) = 0\}, \\
    {}^{\perp}\mathcal{C} &= \{M \in \mods \Lambda \mid \forall X \in \mathcal{C}, \Hom_\Lambda(M,X) = 0\}.
\end{align*}

A particularly important class of subcategories of $\mods \Lambda$ are the \textit{torsion classes}, as introduced in \cite{Dickson66}. A torsion class $\Tcal$ is a full subcategory of $\mods \Lambda$ which is closed under quotients and extensions. Its right $\Hom$-orthogonal category $\Tcal^\perp$ is a \textit{torsion-free class}, which by definition is a full subcategory of $\mods \Lambda$ closed under submodules and extensions. Given a module $M \in \mods \Lambda$, such a \textit{torsion pair} $(\Tcal, \Tcal^\perp)$ induces by \cite{Dickson66} a short exact sequence
\[ 0 \to t_{\Tcal}M \to M \to f_{\Tcal}M \to 0,\]
uniquely defined by the requirements $t_{\Tcal}M \in \Tcal$ and $f_{\Tcal}M \in \Tcal^\perp$ . Another central class of subcategories of $\mods \Lambda$ are the \textit{wide subcategories} first considered in the context of abelian categories in \cite{Hovey2001}. A full subcategory $\Wcal \subseteq \mods \Lambda$ is called wide if it is closed under kernels, cokernels and extensions. Torsion classes, torsion-free classes and wide subcategories all form partially ordered sets under inclusion, which we denote by $\tors \Lambda$, $\torf \Lambda$ and $\wide \Lambda$ respectively.

\subsection{$\tau$-tilting theory}
Following \cite{tau}, we give in this section an introduction to $\tau$-tilting theory where central examples of wide subcategories and torsion pairs appear. Let $\tau$ denote the Auslander-Reiten translation of $\mods \Lambda$. Denote by $\proj \Lambda$ the full subcategory of $\mods \Lambda$ consisting of projective $\Lambda$-modules. 

\begin{definition}\cite[Def. 0.1]{tau}
     A module $M \in \mods \Lambda$ is called \textit{$\tau$-rigid} if $\Hom(M, \tau M) = 0$ and \textit{$\tau$-tilting} if additionally $|M| = n$.
\end{definition}

We also need the notion of $\tau$-rigid and $\tau$-tilting \textit{pairs}. These are pairs of modules, and we begin by establishing some conventions on such pairs more generally, which play a central role throughout our paper.

\begin{definition}
    A \textit{pair of modules over $\Lambda$} is a pair $(M,N) \in \mods \Lambda \times \mods \Lambda$. We consider a pair $(M',N')$ to be a direct summand of a pair $(M,N)$ if $M'$ is a direct summand of $M$ and $N'$ is a direct summand of $N$.  Furthermore, pairs of modules of the form $(M,0)$ or $(0,M)$ play a special role in this paper, and we call these pairs of modules \textit{formally signed modules}. 
    The pair $(M,N)$ is called \textit{basic}, if
    both $M$ and $N$ are basic. We will always only work with basic pairs of modules.
    
    A formally signed module of the form $(M,0)$ is said to have sign $0$ and a formally signed module of the form $(0,M)$ is said to have sign $1$. A formally signed module with sign $0$ is sometimes called an unsigned module or simply a module. By abuse of notation, we borrow the shift-functor notation and often denote a formally signed module $M$ with sign $i$ by $M[i]$. Extending this analogy, we define $(M[i])[j] = M[i+j]$ for integers $i,j$ such that $i+j \in \{0,1\}$. A \textit{signed projective} is a formally signed module $P[1]$ or equivalently a pair $(0,P)$ where $P$ is a projective module.
\end{definition}

The central objects of $\tau$-tilting theory may now be defined.

\begin{definition}\cite[Def. 0.3]{tau}
    A pair $(M,P) \in \mods \Lambda \times \proj \Lambda$ is called $\tau$-rigid if $M$ is $\tau$-rigid and $\Hom(P,M) =0$. It is called $\tau$-tilting if additionally $|M|+|P|=n$.
\end{definition}

\begin{remark}
    Note that in the definition of $\tau$-rigid pairs, the second coordinate of the pair must be a projective module. We will later work with other pairs of modules where the second coordinate need not be projective.
\end{remark}

We denote by $\stt \Lambda$ the collection of $\tau$-tilting pairs of $\mods \Lambda$, and by $\str \Lambda$ the collection of $\tau$-rigid pairs. For a $\tau$-rigid pair $T$ in $\str \Lambda$, we denote by $\comp{T}{\Lambda}$ the set of $\tau$-rigid pairs over $\Lambda$ whose direct sum with $T$ is basic and a $\tau$-rigid pair. If $M$ is $\tau$-rigid, then $(\Gen M,M^\perp)$ is a torsion pair of $\mods \Lambda$ by \cite[Thm. 5.10]{auslandersmalo81}, and we denote by $f_M$ the torsion-free functor associated to this torsion pair. Torsion classes arising in this way are functorially finite and conversely each functorially finite torsion class arises as $\Gen M$ of some $\tau$-rigid module $M \in \mods \Lambda$ by \cite[Prop. 1.1, Prop. 1.2]{tau}. We denote the subset of functorially finite torsion classes by $\ftors \Lambda$. Restricting to $\tau$-tilting pairs, we obtain by \cite[Thm. 2.7]{tau} a bijective map $\Gen\colon \stt \Lambda \rightarrow \ftors \Lambda$ given by 
\[(M,P) \longmapsto \Gen M.\] 

Given a full subcategory $\Xcal \subseteq \mods \Lambda$, a module $X \in \Xcal$ is called \textit{$\Ext$-projective in $\Xcal$} if \sloppy $\Ext_\Lambda^1(X, \Xcal) = 0$.  Moreover, we call an $\Ext$-projective module $X \in \Xcal$ \textit{split in $\Xcal$} if every epimorphism in $\Xcal$ with target $X$ splits and call an $\Ext$-projective module $X \in \Xcal$ \textit{non-split in $\Xcal$} if it is not split in $\Xcal$. 

When $\Xcal$ has a finite number of $\Ext$-projectives, we denote by $\eproj(\Xcal)$ the direct sum of one copy of each of these indecomposable $\Ext$-projective modules up to isomorphism. In this case we also denote by $\eproj_s(\Xcal)$ (resp. $\eproj_{ns}(\Xcal)$) the direct sum of one copy of each indecomposable split (resp. non-split) $\Ext$-projective module in $\Xcal$ up to isomorphism. We can then express the inverse of $\Gen\colon \stt \Lambda \rightarrow \ftors \Lambda$ as defined in the previous paragraph to be the map $\eproj\colon \ftors \Lambda \xrightarrow[]{} \stt \Lambda$ given by \[\Tcal \longmapsto  (\eproj(\Tcal),P)\] where $P$ is the unique projective $\Lambda$-module such that $(\eproj(T),P)$ is a $\tau$-tilting pair. This bijection induces the structure of a partially-ordered set on $\stt \Lambda$. Under this structure, $(\Lambda,0)$ is the greatest element and $(0,\Lambda)$ the least element. The poset-structure on $\stt \Lambda$ gives a poset structure also on $\stt_T \Lambda$ for any $\tau$-rigid pair $T$. It turns out that $\stt_T \Lambda$ also has a greatest and least element, the Bongartz completion of $T$ and the co-Bongartz completion of $T$, respectively\cite[Thm. 4.4]{dirrt}\cite[2.10]{tau}. We will now discuss these completions in more detail.

\begin{theorem}\cite[Thm. 2.10]{tau}
    Let $M$ be a $\tau$-rigid module over an algebra $\Lambda$. Then ${}^\perp \tau M$ is a functorially finite torsion class and $\bcomp(M) = \eproj({}^\perp \tau M)$ is a $\tau$-tilting module such that $M \in (\add \bcomp(M))$ and $\Gen(\bcomp(M)) = {}^\perp \tau M$. 
\end{theorem}

We call $\bcomp(M)$ the \textit{Bongartz completion of $M$}. In \cite[Sec. 4]{dirrt} the authors introduce Bongartz and co-Bongartz completions for $\tau$-rigid pairs. Following them, note that for a $\tau$-rigid pair $(M,P)$ over $\Lambda$ where $P \cong \Lambda e$ for some idempotent $e \in \Lambda$, we have that $M$ is a $\tau$-rigid module over $\Lambda/\langle e\rangle$. Conversely, $\tau$-rigid modules $N$ over $\Lambda/\langle e\rangle$ give $\tau$-rigid pairs $(N,P)$ over $\Lambda$. We then define the Bongartz completion of $(M,P)$ to be the $\tau$-tilting pair \[\bcomp(M,P) = (\bcomp_{\Lambda/\langle e\rangle}(M),P)\]
where $\bcomp_{\Lambda/\langle e\rangle}(M)$
denotes the Bongartz completion of $M$ in $\mods 
\Lambda/\langle e\rangle$.

Dually, we define the \textit{co-Bongartz completion} of a $\tau$-rigid pair $(M,P)$ to be the unique $\tau$-tilting pair whose unsigned module component generates $\Gen M$: \[\ccomp(M,P) =(\eproj(\Gen M), Q),\] where $Q$ is the unique basic projective module making it into a $\tau$-tilting pair.

For a $\tau$-rigid pair $T$, the (co-)Bongartz complement of $T$ is the unique $\tau$-rigid pair $\bcomplement{T}$ (respectively $\ccomplement{T}$) such that $\bcomplement{T} \oplus T = \bcomp(T)$ (respectively $\ccomplement{T} \oplus T = \ccomp(T)$). Note that any Bongartz complement will be unsigned.

\begin{proposition}\label{prop:co-bongartz-summand-characterization}
    Let $T = (M,P)$ be a $\tau$-rigid pair. A signed module $X \in \stt_T \Lambda$ is a summand of the co-Bongartz complement of $T$ if and only if $X$ is signed projective or a module in $\Gen M$.
\end{proposition}

\begin{proof}
    Assume first that $X$ is signed projective or in $\Gen M$. 
    Let $T' = T \oplus X = (M',P')$. Let $C = (N,Q)$ be the co-Bongartz completion of $T$ and let $C' = (N',Q')$ be the co-Bongartz completion of $T'$. By our assumption on $X$, we have $\Gen M = \Gen M'$. Since $C$ and $C'$ are co-Bongartz completions, we thus have \[\Gen C = \Gen M = \Gen M' = \Gen C'.\] But since $C$ and $C'$ are $\tau$-tilting pairs corresponding to the same torsion class, we must have $C = C'$.

    Conversely, let $X$ be a direct summand of the co-Bongartz completion $C = (M',P')$ of $T = (M,P)$. Assume that $X$ is not a shifted projective. Then $X \in \add M'$ so $X \in \Gen M' = \Gen M$ as wanted. This finishes the proof.
\end{proof}

The Bongartz completion and co-Bongartz completion of a $\tau$-rigid pair $T$ are distinct pairs unless $T$ is $\tau$-tilting. If $T$ is almost complete, that is $|T| = n-1$, then there are exactly two $\tau$-tilting pairs $T^+$ and $T^-$ having $T$ as a direct summand, namely the Bongartz completion and co-Bongartz completion of $T$. Further, these two completions are related via $\tau$-tilting mutation \cite[Def.-Prop. 2.28]{tau}.

\begin{definitionproposition}\label{prop:tau_mutation_air}\cite[Def.-Prop. 2.28]{tau} Let ${T =}(M \oplus X,P)$ be a $\tau$-tilting pair with $X$ indecomposable. If $X \notin \Gen M$ then we may \textit{left mutate} $T$ at $X$. Let $f\colon X \to M'$ be a minimal left $\add M$-approximation of $X$. Let $Y = \coker f$. Then the left mutation of $(M \oplus X,P)$ at $X$ is defined as 

     \[\mu_X(M \oplus X,P) = \begin{cases}
         (M \oplus Y,P) & \text{if $Y \neq 0$,} \\ 
         (M,P \oplus Q) & \text{otherwise,}
     \end{cases}\] 
     where in the last case, $Q$ is the unique projective module making $(M,P \oplus Q)$ a $\tau$-tilting pair.
\end{definitionproposition}

Another viewpoint commonly used to study $\tau$-tilting theory is that of (2-term) silting theory as introduced in \cite{KellerVossieck}, where we instead of working with $\tau$-rigid pairs work with 2-term pre-silting objects in the bounded homotopy category of finitely generated projective modules $K^b(\proj \Lambda)$. This is also related to the approach in \cite{DerksenFei2015}. We refer to \cite[Sec. 3]{tau}
for definitions and details on (2-term) silting theory.

\begin{definition}
    Given a $\tau$-rigid pair $T = (M,P)$ over $\Lambda$, we define by $\mathbb{P}_T$ the two-term complex concentrated in degrees $-1$ and $0$ given by \[\mathbb{P}_T\colon \quad  P_{-1} \oplus P \xrightarrow{\scriptscriptstyle \begin{pmatrix} \alpha & 0 \end{pmatrix}} P_0 \quad \in K^b(\proj \Lambda), \]
    where $P_{-1} \xrightarrow{\alpha} P_0 \to M$ is a minimal projective presentation of $M$.

\end{definition}

By \cite[Thm. 3.1]{tau}, $\mathbb{P}_{(-)}$ gives by a bijection between $\str \Lambda$ and the $2$-term presilting objects living in degrees $-1$ and $0$ of $K^b(\proj \Lambda)$, whose inverse is given by associating to a $2$-term pre-silting complex $\mathbb{P} \cong \mathbb{Q} \oplus P[1]$  the $\tau$-rigid pair $(H^0(\mathbb{Q}),P)$. The connection with silting  justifies calling a summand $(0,Q)$ of a $\tau$-rigid pair a shifted projective, since it can be seen as the stalk complex $Q[1] \in K^b(\proj \Lambda)$.

Using the Bongartz completion of a module, Jasso defined in \cite{jassoreduction} what is known as $\tau$-tilting reduction or Jasso reduction. The work of Jasso has been extended for example by \cite{dirrt}, \cite{børve2023} and \cite{buan2023perpendicular}. Let $(M,P)$ be a $\tau$-rigid pair over an algebra $\L$ with rank $n$. The \textit{$\tau$-perpendicular category} $J(M,P)$, first studied by Jasso for the case $P = 0$ and then generalised by \cite{dirrt}, is defined as 
\[J(M,P) = M^\perp \cap {}^\perp \tau M \cap P^\perp.\]
This subcategory of $\mods \Lambda$ is wide by \cite[Thm. 4.12]{dirrt}, \cite[Cor. 3.25]{Br_stle_2019} and \cite[Thm. 1.4]{jassoreduction}, and is further equivalent to a module category of an algebra of rank $n - |M| - |P|$, say $J(M,P) \cong \mods \Gamma$. The $\tau$-tilting theory of $\Gamma$ may then be studied using the $\tau$-tilting theory of $\Lambda$: There is a mutation-preserving bijection between the $\tau$-tilting pairs over $\Gamma$ and the $\tau$-tilting pairs of $\Lambda$ having $(M,P)$ as a summand by \cite[Thm. 1.1]{jassoreduction} and \cite[Thm. 4.12]{dirrt}.

\subsection{$\tau$-exceptional sequences and TF-ordered $\tau$-rigid modules}
Introduced as a generalisation of signed exceptional sequences \cite{IgusaTodorov2017}, signed $\tau$-exceptional sequences were defined and studied in \cite{tauexceptional_buanmarsh}.

\begin{definition}\cite[Def. 1.3]{tauexceptional_buanmarsh}\label{Def:signedtauexcep}
    A sequence of indecomposable formally signed modules \[(M_1[s_1],M_2[s_2],\dots,M_t[s_t])\] is called a \textit{signed $\tau$-exceptional sequence} of length $t$ in $\mods \Lambda$ if:
    \begin{enumerate}
        \item $M_t[s_t] \in \str \Lambda$, and
        \item if $t > 1$, then $(M_1[s_1],\dots,M_{t-1}[s_{t-1}])$ is a signed $\tau$-exceptional sequence of length $t-1$ in $J(M_t) (\simeq \mods \Lambda')$.
    \end{enumerate}
\end{definition}

Note the recursive nature of the above definition, where keeping track of the relevant ambient module category is key: A module $M$ may be $\tau$-rigid (or projective) in a wide subcategory $\mods \Lambda' \simeq \Xcal \subset \mods \Lambda$ while not being $\tau$-rigid (or projective) in $\mods \Lambda$.

A signed $\tau$-exceptional sequence where all elements have sign $0$ is called an \textit{unsigned $\tau$-exceptional sequence} or simply a $\tau$-exceptional sequence. By a \textit{complete} (signed) $\tau$-exceptional sequence we mean a sequence of length equal to the rank of the algebra in question.

In \cite{tauexceptional_buanmarsh}, an explicit bijection between the set of signed $\tau$-exceptional sequences of length $t$ in $\mods \Lambda$ and the set of \textit{ordered} basic $\tau$-rigid pairs over $\Lambda$ with $t$ indecomposable direct summands is given.

\begin{definition}

An {\em ordered module} is a pair $(M, \sigma)$, where $M$ is a module and $\sigma$ a bijection $\{1,2,\dots,|M|\} \to \ind(\add M)$ where $\ind(\mathcal{C})$ is the set consisting of the indecomposable objects in a category $\mathcal{C}$ up to isomorphism. Similarly, an ordered pair of modules $(M,N)$ is given by a pair $((M,N),\sigma)$ where $\sigma$ is a bijection $\{1,2,\dots,|M| + |N|\} \to \ind(\add M) \cup \{X[1] \mid X \in \ind(\add N)\}$. We will now define some notions for ordered modules and remark that the corresponding notions can easily be translated to the context of ordered pairs of modules.

Given an ordered module $(M,\sigma)$ and a permutation $\phi\colon \{1,2,\dots,|M|\} \to \{1,2,\dots,|M|\}$, we define $\phi(M,\sigma) = (M,\sigma\circ \phi^{-1})$. Letting $\mathfrak{S}_n$ denote the symmetric group of permutations on $\{1,2,3,\dots,n\}$, the above gives a left $\mathfrak{S}_{n}$-action on ordered modules with $n$ summands, permuting their summands.

For an ordered module $(M,\sigma)$ and an integer $1 \leq i \leq |M|$, we let $(M,\sigma)_i = \sigma(i)$. Also, we define for integers $1 \leq i \leq |M|$ the ordered modules $(M,\sigma)_{<i}$ and $(M,\sigma)_{>i}$ as $(\bigoplus_{j = 1}^{i-1}\sigma(j),\sigma_{<i})$ and $(\bigoplus_{j = i+1}^{|M|} \sigma(j),\sigma_{>i})$ where $\sigma_{<i}$ is the restriction of $\sigma$ to $\{1,2,3,\dots,i-1\}$ and $\sigma_{>i}$ is defined by setting $\sigma_{>i}(x) = \sigma(x+i)$ for $1 \leq x \leq |M|-i$. We define $(M,\sigma)_{\geq i}$ and $(M,\sigma)_{\leq i}$ analogously, so for example $(M,\sigma)_{\leq 1} = M_1$ and $(M,\sigma)_{<1} = 0$.

By abuse of notation, we will usually denote an ordered module $(M,\sigma)$ simply by $M$ or by the sum $\sigma(1) \oplus \sigma(2) \oplus \dots \oplus \sigma(|M|)$. Note also that for a given decomposition of a module $M$ into indecomposable direct summands $M \cong M_1 \oplus M_2 \oplus \dots \oplus M_t$ there is a canonical choice of ordered module $(M,\sigma)$, namely given by defining $\sigma(i) = M_i$.

\end{definition}

\begin{example}
    Let $M = M_1 \oplus M_2 \oplus M_3$ denote an ordered module, which we formally may interpret as the object $M = (\bigoplus_{i = 1}^3 M_i,i \mapsto M_i)$. Let $\theta$ be the cycle $(1 \quad 2 \quad  3)$. Then \[\theta(M_1 \oplus M_2 \oplus M_3) = (\bigoplus_{i =1}^3 M_i,i \mapsto M_{\theta^{-1}(i)}) = M_3 \oplus M_1 \oplus M_2.\]
\end{example}

We will often utilise the bijection between ordered $\tau$-rigid pairs and signed $\tau$-exceptional sequences and now recall some notation from \cite{tauexceptional_buanmarsh}.

\begin{proposition}\cite[Thm. 5.4]{tauexceptional_buanmarsh}\label{prop:reduction}
    Let $\mathcal{C}$ be a module category. There is a bijection
    \begin{center}

        $\{\text{Ordered $\tau$-rigid pairs in $\mathcal{C}$ with $t$ indecomposable direct summands}\}$ \\
      $ \Psi^\mathcal{C}_t \Big\downarrow $ \\

      $  \{\text{Signed $\tau$-exceptional sequences in $\mathcal{C}$ of length $t$}\}$
     \end{center} and we denote by $\Psi^\mathcal{C}$ the bijection between all ordered $\tau$-rigid pairs in $\mathcal{C}$ and all signed $\tau$-exceptional sequences in $\mathcal{C}$. Formally, we let $\Psi^\mathcal{C}(T) = \Psi^\mathcal{C}_{|T|}(T)$. In contexts where there is no ambiguity about which category we are considering, we may drop the superscript.
\end{proposition}

We say that $T \in \str \mathcal{C}$ \textit{induces} the $\tau$-exceptional sequence $\Psi^\mathcal{C}(T)$. In \cite{mendozatreffinger_stratifyingsystems}, the authors describe explicitly the ordered $\tau$-rigid pairs that induce unsigned $\tau$-exceptional sequences, and show that these are exactly the so-called \textit{TF-ordered} $\tau$-rigid modules.

\begin{definition}\cite[Def. 3.1]{mendozatreffinger_stratifyingsystems}
    An ordered $\tau$-rigid $\Lambda$-module $M_1 \oplus M_2 \oplus \dots \oplus M_t$ is called TF-ordered if $M_i \notin \Gen (M_{>i})$ for all $1 \leq i \leq t-1$. 
\end{definition}

Actually, for any $\tau$-rigid module $M$, there is always
at least one ordering of its summands such that $M$ with this ordering
is TF-ordered, by \cite[Thm. 1.2(a)]{mendozatreffinger_stratifyingsystems}. We will call such orderings the {\em TF-orders of $M$.}
Recall that $f_N$ denotes the torsion-free functor associated to the torsion pair $(\Gen N, N^\perp)$ for a $\tau$-rigid module $N$.

Using the above notion, Mendoza and Treffinger give the following characterisation of (unsigned) $\tau$-exceptional sequences:

\begin{theorem}\cite[Thm. 5.1]{mendozatreffinger_stratifyingsystems}\label{thm:mt}
    Let $M = M_1 \oplus M_2 \oplus \dots \oplus M_t$ be a TF-ordered $\tau$-rigid module over an algebra $\Lambda$. Then 
    \[(f_{M_{>1}}(M_1),f_{M_{>2}}(M_2),\dots,M_t) = \Psi(M)\]
    and all $\tau$-exceptional sequences over $\Lambda$ are uniquely given by $\Psi(T)$ for some TF-ordered $\tau$-rigid module $T$. Thus $\Psi$ restricts to a bijection between TF-ordered $\tau$-rigid modules and  unsigned $\tau$-exceptional sequences over $\Lambda$.
\end{theorem}

\section{Mutation of $\tau$-exceptional sequences and corresponding TF-ordered modules}

In this section we will utilise the correspondence between $\tau$-exceptional sequences and TF-ordered $\tau$-rigid modules to study how the mutation of $\tau$-exceptional sequences as introduced in \cite{BHM2024} corresponds to a mutation on TF-ordered $\tau$-rigid
modules. 

\subsection{The $V$ and $E$-maps}
We have seen, in \cref{thm:mt}, how to compute $\Psi(M)$ for a TF-ordered $\tau$-rigid module $M$. To compute $\Psi$ generally, i.e to describe how ordered $\tau$-rigid pairs correspond to signed $\tau$-exceptional sequences, we need the $E$-map as introduced in \cite{tauexceptional_buanmarsh}. To define this map, an explicit bijection between the summands of the Bongartz complement of a $\tau$-rigid pair $T$ and the summands of the co-Bongartz complement of the same pair is introduced in \cite{tauexceptional_buanmarsh}. As this map plays a central role in this paper, we begin by recalling their work and denote this map by $V$.

\begin{definition}\label{def:Vmap}
    Let $T$ be a $\tau$-rigid pair. Recall that $\ccomplement{T}$ and $\bcomplement{T}$ are the co-Bongartz complement and Bongartz complement of $T$ respectively. Define $V_T\colon \add (\ccomplement{T}) \to \add(\bcomplement{T})$ as follows: Let $X \in \ind( \add(\ccomplement{T}))$ and let $\mathbb{P}_X$ and $\mathbb{P}_T$ be the two-term presilting objects associated to $X$ and $T$ respectively. Let now $\alpha\colon \mathbb{P}_T^+ \to \mathbb{P}_X$ be a minimal right $\add(\mathbb{P}_T)$-approximation of $\mathbb{P}_X$ and set $\mathbb{Y} = \cone(\alpha)[-1]$. Then $V_T(X)$ is defined to be $H^0(\mathbb{Y})$. 
\end{definition}

Sometimes we can simplify the computation of the $V$-map. A few special cases are discussed below.			

\begin{remark}
    If $T = (0,0)$ in \cref{def:Vmap}, the $V$-map as defined may be computed and will send a shifted projective $P[1]$ to the module $P$. 
\end{remark}

\begin{lemma}\label{lem:v-map-on-shifted-projectives}\cite[Sec. 4]{tauexceptional_buanmarsh}
    Let $P[1] \oplus Q[1]$ be a basic shifted projective. Then $V_{Q[1]}(P[1]) = f_Q P$.
\end{lemma}
			
\begin{proof}
    We prove the statement here for completeness, but remark that the lemma follows from \cite[Def. 4.10, Rmk. 4.11]{tauexceptional_buanmarsh}.
    Let $\mathbb{P}_{P[1]}$ and $\mathbb{P}_{Q[1]}$ denote the shifted stalk complexes of $P$ and $Q$ respectively, so the $\add(\mathbb{P}_{Q[1]})$-approximation $\alpha$ of $\mathbb{P}_{P[1]}$ considered in \cref{def:Vmap} is given by a (shifted) $\add Q$-approximation $\beta$ of $P$. The negative shift of the cone of $\alpha$ is then given by the two-term complex $\cone (\alpha) [-1]\colon Q^+ \xrightarrow{\beta} P$. As $\beta$ is a right approximation, the image of $\beta$ must be equal to the trace of $Q$ in $P$, so $H^0(\cone(\alpha)[-1]) = f_Q P$.
\end{proof}

The following special cases will be useful in later sections.
 
\begin{lemma}\label{lem:gen_mut}
	Let $L,M,N$ be indecomposable $\tau$-rigid modules.
	\begin{itemize}
	\item[(a)] If there exists a right exact sequence
    \[ L \to M' \to N \to 0\]
	with $M' \in \add M$, $L \in \add(\bcomp(M))$ and $N \in \add(\ccomp(M))$, and the first map is a left $\add M$-approximation, then $V_M(N) \cong L$.
	\item[(b)] If there exists an exact sequence
					\[ 0 \to L \to M' \to N \to 0\]
					with $M' \in \add M$, $L \in \add(\bcomp(M))$ and $N \in \add(\ccomp(M))$, then $V_M(N) \cong L$.
	   \item[(c)] If $N \in \add(\bcomp(M))$ and $N[1] \in \add(\ccomp(M))$, then $V_M(N[1]) \cong N$.
    \end{itemize}
\end{lemma}

\begin{proof}
    Part (a) is a reformulation of \cite[Def. 3.1, part I(b)]{BuanMarsh2018}, noting that the map $L \to M$ is automatically left minimal, since $N$ is indecomposable and not in $\add M$ by assumption.
    
    For part (b), since $N \in \add(\ccomp(M))$, we have that $M \oplus N$ is $\tau$-rigid and so $M \in {}^\perp \tau N$, and hence $\Ext^1(N,M) = 0$ by \cite[Prop. 5.8]{auslandersmalo81}. Applying $\Hom(-,M)$ to the short exact sequence gives that $\Hom(M', M) \to \Hom(L,M)$ is an epimorphism, and so the map $L \to M'$ is a left $\add M$-approximation. It then follows from (a)
	that $V_M(N) \cong L$.
    
	For part (c), we apply \cite[Def. 3.1, part I(c)]{BuanMarsh2018}. Since $N \in \add(\bcomp(M))$, the exact sequence $N \to \bcomplement{T} \to M' \to 0$ with $N \to \bcomplement{T}$ a minimal left $\bcomp(M)$ approximation simply becomes $N \to N \to 0$ i.e. $\bcomplement{T}\cong N$. The claim follows directly from this.  
\end{proof}

We now recall the definition of the $E$-map as first introduced in \cite{tauexceptional_buanmarsh}.

\begin{definition}\cite[Prop. 5.6, Prop. 5.8, Prop 5.11]{tauexceptional_buanmarsh}\label{defn:Emap} For a $\tau$-rigid pair $T = (M,P)$ in a module category $\mathcal{C}$, we define the \textit{reduction map} \[E^\mathcal{C}_T\colon \comp{T}{\mathcal{C}} \to \comp{}{J(T)}.\] 

    Assume first that $T$ is on the form $(M,0)$ or $(0,P)$. We then define $E_T(X)$ for an indecomposable $X \in \str_T \mathcal{C}$ as 
    \[E^\mathcal{C}_T(X) = \begin{cases} f_T X & \text{if both $T$ and $X$ are $\tau$-rigid modules, and $X \not \in \Gen T$,}  \\X & \text{if $T$ is a shifted projective and $X$ is a $\tau$-rigid module,} \\ (f_{T[-1]}(X[-1]))[1] & \text{if $T$ is a shifted projective and $X$ is a shifted projective,} \\ f_T(V_T(X))[1] & \text{otherwise.}  \end{cases}\]
				
    Note that we are in the last case if $T$ is a $\tau$-rigid module and $X \in \Gen T$ or $X$ is a shifted projective. By \cref{prop:co-bongartz-summand-characterization} this is equivalent to $X$ being a summand of the co-Bongartz complement of $T$, hence $V_T(X)$ is well-defined. 

    Further, if $T = (M,P)$ for $M \neq 0$ and $P \neq 0$, we define $E^\mathcal{C}_T(X) = E^{J(M)}_{E^{\mathcal{C}}_M(P[1])}( E^\mathcal{C}_M(X))$. Lastly, we extend $E^{\mathcal{C}}_T$ additively, so that for $U \in \comp{T}{\mathcal{C}}$ we have $E^{\mathcal{C}}_T(U) = \bigoplus_{1 \leq i \leq t}E^{\mathcal{C}}_T(U_i)$ for a decomposition of $U$ into indecomposable direct summands $U = U_1 \oplus U_2 \oplus \dots \oplus U_t$.
				
\end{definition}

\begin{remark}
  Note that an additive function $f\colon \mathcal{C} \to \mathcal{D}$ between sets of (pairs of) modules taking indecomposable objects to indecomposable objects may be considered a function on \textit{ordered} (pairs of) modules by defining $f(M,\sigma) = (f(M),f \circ \sigma)$ for $M \in \mathcal{C}$.

    In particular, we will let $E^\mathcal{C}_T$ take ordered $\tau$-rigid pairs in $\str_T\mathcal{C}$ to ordered $\tau$-rigid pairs in $\str J(T)$.
\end{remark}

By \cite[Prop. 5.6,Prop 5.11]{tauexceptional_buanmarsh} the $E$-map defined above is bijective. When there is no ambiguity regarding the ambient category $\mathcal{C}$ we may omit the superscript. The following formula regarding associativity of the $E$-map was first investigated in the $\tau$-tilting finite case in \cite{BuanMarsh2018} and in the general case in \cite{buan2023perpendicular}.

\begin{theorem}\cite[Thm. 6.12]{buan2023perpendicular}\label{thm:E-associative}
    Let $T$ be a basic $\tau$-rigid pair in $\mathcal{C}$. The reduction induced by the $E$-map is associative in the sense that \[E^\mathcal{C}_T = E^{J(M)}_{E^{\mathcal{C}}_M(N)} \circ E^\mathcal{C}_M\] holds for any decomposition $T = M \oplus N$.
\end{theorem}

Using a reformulation of \cite[Thm. 3.15]{jassoreduction} we obtain a dual to \cref{prop:co-bongartz-summand-characterization}, characterising exactly when $X \in \stt_T \mathcal{C}$ is a direct summand of the Bongartz complement of $T$.

\begin{theorem}\cite[Thm. 3.15]{jassoreduction}\label{thm:E-map-order-preserving}
    Let $T = (M,P)$ be a $\tau$-rigid pair in a module category $\mods \Lambda$. Then $E_T$ restricts to an order-preserving bijection between $\tau$-tilting pairs on the form $T \oplus U$ and $\tau$-tilting pairs in $J(T)$.
\end{theorem}

\begin{proof}
    Since $E_T\colon \str_T \Lambda \to \str J(T)$ is a bijection and $|E_T(U)| = |U|$ for all $U \in \str_T \Lambda$, we only need to prove the order-preserving part of the statement.

    Let $U_1 = (N_1,Q_1)$ and $U_2 = (N_2,Q_2)$ be $\tau$-rigid pairs such that $U_1 \oplus T$ and $U_2 \oplus T$ are $\tau$-tilting pairs. Assume that $U_1 \oplus T \geq U_2 \oplus T$, meaning that $\Gen (M \oplus N_1) \supseteq \Gen(M \oplus N_2)$.

    Let $E_T(U_1) = (N'_1,Q'_1)$ and $E_T(U_2) = (N'_2,Q'_2)$. Then $N'_1 \cong f_M(N_1 \oplus M)$ and $N'_2 \cong f_M(N_2 \oplus M)$, and by \cite[Thm. 3.15]{jassoreduction} we have therefore have $N'_1 \geq N'_2$ in $J(M)$ and therefore $N'_1 \geq N'_2$ in $J(M) \cap P^{\perp} = J(T)$.
\end{proof}

\begin{corollary}\label{cor:Bong-proj}
    Let $T$ be a $\tau$-rigid pair in a module category $\mathcal{C}$. Then $X \in \stt_T \mathcal{C}$ is a direct summand of $\bcomp(T)$ if and only if $E_T(X)$ is projective in $J(T)$.
\end{corollary}

\begin{proof}
    Let $B$ be the Bongartz complement of $T$,
    that is: $\bcomp(T) = B \oplus T$.
    Then $E_T(B)$ must be maximal in $\stt J(T)$
    by \cref{thm:E-map-order-preserving}. The maximal element of $\stt J(T)$ is exactly the sum of the projectives in
   $J(T)$.
\end{proof}

We now describe the the map $\Psi$ from \cref{prop:reduction}, see \cite[Rmk. 5.13]{tauexceptional_buanmarsh} for further details.
\begin{theorem}\cite[Thm. 5.4]{tauexceptional_buanmarsh}\label{def:def-of-psi}
    For any module category $\mathcal{C}$ there is a bijection
    \begin{center}

        $\{\text{Ordered $\tau$-rigid pairs in $\mathcal{C}$ with $t$ indecomposable direct summands}\}$ \\
      $ \Psi^\mathcal{C}_t \Big\downarrow $ \\

      $  \{\text{Signed $\tau$-exceptional sequences in $\mathcal{C}$ of length $t$}\}$
     \end{center}

     given by taking an ordered $\tau$-rigid pair $T = T_1 \oplus T_2 \oplus \dots \oplus T_t$ in $\mathcal{C}$ to
\[\Psi^{\mathcal{C}}_t(T) = \begin{cases} (T_1) & \text{if $t = 1$,} \\ 
                                       (\Psi^{J(T_t)}_{t-1}(E_{T_t}(T_1) \oplus E_{T_t}(T_2) \oplus \dots \oplus E_{T_t}(T_{t-1})),T_t) & \text{otherwise.}             \end{cases}\]

\end{theorem}
Note that in the above theorem, we ignore all but the outermost parenthesis, so for example we write $(X,Y,Z)$ instead of $(((X),Y),Z)$ as naively dictated by the definition. Further, let $\Psi^\mathcal{C}(T) = \Psi^\mathcal{C}_{|T|}(T)$ so that we may drop the subscript if wanted. When there is no ambiguity about which category we are considering, we may also drop the superscript.
			
Using \cref{thm:E-associative}, we get the following useful formula, which generalises \cref{thm:mt}.

\begin{proposition}
    For $T = T_1 \oplus T_2 \oplus \dots \oplus T_t$ an ordered $\tau$-rigid pair with $t$ indecomposable direct summands, we have
    \[\Psi(T) = (E_{T_{>1}}(T_1),E_{T_{>2}}(T_2),\dots,E_{T_{t}}(T_{t-1}),T_t).\]
\end{proposition}

\begin{proof}

    We proceed by induction on $t$. For $t = 1$, we have $\Psi(T_1) = (T_1)$ as wanted, by \cref{def:def-of-psi}.

    Assume now that the claim holds for all ordered $\tau$-rigid pairs with at most $t$ summands and over any algebra. Let now $T = T' \oplus T_{t+1} = T_1 \oplus T_2 \oplus \dots \oplus T_t \oplus T_{t+1}$ be an ordered $\tau$-rigid pair with $|T| = t+1$.

    Using the definition of $\Psi$, we get the first equality in the computation below:
    \begin{align*}
        \Psi(T) &= (\Psi^{J(T_{t+1})}_{t}(E_{T_{t+1}}(T')),T_{t+1}) \\
                &= ((E^{J(T_{t+1})}_{E_{T_{t+1}}(T')_{>1}}(E_{T_{t+1}}(T')_1),\dots,E^{J(T_{t+1})}_{E_{T_{t+1}}(T')_{t}}(E_{T_{t+1}}(T')_{t-1}),E_{T_{t+1}}(T')_t),T_{t+1}) \\
                &= ((E_{T_{>1}}(T'_1),\dots,E_{T_{>t-1}}(T'_{t-1}),E_{T_{t+1}}(T'_t)),T_{t+1}) \\
                &= (E_{T_{>1}}(T_1),\dots,E_{T_{>t-1}}(T_{t-1}),E_{T_{t+1}}(T_t),T_{t+1}). 
                 \end{align*}
    The second equality follows from our induction hypothesis (since $E_{T_{t+1}}(T')$ is a $\tau$-rigid module in $J(T_{t+1})$ with $t$ indecomposable direct summands) and the third equality follows from \cref{thm:E-associative}. Finally, removing the inner pair of parentheses and rewriting the sequence (as $T'_i = T_i$ for $i \leq t$) makes it clear that we have obtained the wanted result.
\end{proof}

\subsection{Some operations on TF-ordered $\tau$-rigid modules}
The mutation operation on $\tau$-exceptional sequences, as introduced in \cite{BHM2024}, induces a mutation on TF-ordered $\tau$-tilting modules. We now discuss some operations that are useful in explicitly describing this induced mutation.

We first consider two operations on (signed) $\tau$-exceptional sequences introduced in \cite{Buan_Marsh_2023}. Let $X = (X_1,X_2,\dots,X_n)$ be a signed $\tau$-exceptional sequence corresponding to the ordered $\tau$-rigid pair $T = T_1 \oplus T_2 \oplus \dots \oplus T_n$. The action of transposing summands in $T$ will clearly induce a new signed $\tau$-exceptional sequence. Also, if $X_i = P[s]$ where $P$ is projective in $J(T_{i+1} \oplus T_{i+2} \oplus \dots \oplus T_n)$, then replacing $P[s]$ with $P[1-s]$ gives a new signed $\tau$-exceptional sequence. We define the following operations.
			
\begin{definition}\cite[Thm. 3.4, Lem. 4.3]{Buan_Marsh_2023}
    We define $\overline{\pi}_i$ as the transposition $(i,i+1)$. Note that $\overline{\pi}_i$ can be considered an element of $\mathfrak{S}_j$ for all $j \geq i+1$. For an ordered pair or module \[T= T_1 \oplus \dots \oplus T_i \oplus T_{i+1} \oplus \dots \oplus T_t\] we thus have \[\overline{\pi}_i(T) = T_1 \oplus \dots \oplus T_{i+1} \oplus T_i \oplus \dots \oplus T_t.\]

    Let $X = (X_1,X_2,\dots,X_t)$  be an arbitrary signed $\tau$-exceptional sequence. We define $\pi_i(X) = (\Psi \circ \overline{\pi}_i \circ \Psi^{-1})(X)$. Further, if $X_i = P[x]$ for some relative projective $P$, then $s_i(X)$ is defined as the $\tau$-exceptional sequence $(X_1,X_2,\dots,
    X_{i-1}, P[1-x], X_{i+1}, \dots,X_t)$, and we lastly define $\overline{s}_i = \Psi^{-1} \circ s_i \circ \Psi$.
\end{definition}

\begin{remark}
 Note that by \cite[Sec. 3]{Buan_Marsh_2023},
 we have that $\pi_1(B,C) = (E_{E^{-1}_C(B)}(C),E^{-1}_C(B))$.
 
\end{remark}

The following special cases will be useful later.

\begin{lemma}\label{lem:v-map-on-ordered-tau-rigid-rank-2}
    Let $T = M \oplus N$ be an ordered $\tau$-rigid pair such that $E_N(M)$ is a projective module in $J(N)$. Then we have the following equality of ordered pairs: 
    \[\overline{s}_1(M \oplus N) = V_{N}^{-1}(M) \oplus N.\]
\end{lemma}

\begin{proof} 

    Note first that $V_{N}^{-1}(M)$ is well-defined as by 
\cref{cor:Bong-proj} we have that $E_N(M)$ being projective in $J(N)$ is equivalent to $M$ being in the Bongartz complement of $N$.
By definition of  $\overline{s}_1$, the claimed equality is equivalent to
$\Psi(V^{-1}_N(M) \oplus N) = s_1(\Psi(M \oplus N))$, and 
by definition we have
\[s_1(\Psi(M \oplus N)) = s_1(E_N(M) \oplus N) = 
E_N(M)[1] \oplus N.\]
On the other hand, we have $\Psi(V^{-1}_N(M) \oplus N) = 
E_N(V^{-1}_N(M)) \oplus N$, so it is sufficient to show that 
\[E_N(V^{-1}_N(M)) = E_N(M)[1].\]

    Consider first the case where $N$ has sign $0$, i.e is a module. Since $V_{N}^{-1}(M)$ is a direct summand of the co-Bongartz completion of $N$, it is either generated by $N$ or a shifted projective by \cref{prop:co-bongartz-summand-characterization}, so that we obtain,
    \begin{align*}
        E_N(V_{N}^{-1}(M)) &= f_N(V_N(V_N^{-1}(M)))[1] = f_N(M)[1] = E_N(M)[1],
    \end{align*}
    where the first and last equalities follow from \cref{defn:Emap}, noting that $M \notin \Gen N$ as otherwise $E_N(M)$ would not be projective in $J(M)$.

    If now $N$ has sign $1$, then $V^{-1}_{N}(M) = Q$ must be a shifted projective since the co-Bongartz complement of $N$ only consists of shifted projectives in this case. Also, using \cref{defn:Emap} and \cref{lem:v-map-on-shifted-projectives}, we have that 
    \[E_N(Q) = f_{N[-1]}(Q[-1])[1] = V_{N}(Q)[1] = M[1].\]
    Lastly we observe that $M[1] = E_N(M)[1]$ because $E_N$ is the identity on modules (i.e of sign $0$) by \cref{defn:Emap}. This finishes the proof.
\end{proof}

In the case where we work with longer $\tau$-exceptional sequences, the below generalisation is useful.

\begin{lemma}\label{lem:v-map-on-tf-ordered}
    Let $M = (P[0],M_2,\dots,M_t)$ be an (unsigned) $\tau$-exceptional sequence, with $\Psi^{-1}(M) = T = T_1 \oplus T_2 \oplus \dots \oplus T_t$ and such that $P$ is relative projective in $J(T_{>1})$. Then \[\overline{s}_1(T) = V^{-1}_{T_{>1}}(T_1) \oplus T_2 \oplus \dots \oplus  T_t.\]
\end{lemma}
			
\begin{proof}
    Clearly $E_{T_{>i}}(T_i) = M_i$ for all $i > 1$, by definition of $M$ and $T$ and \cref{def:def-of-psi}. We then compute 
    \begin{align*}
    E_{T_{>1}}(V^{-1}_{T_{>1}}(T_1)) = f_{T_{>1}}(V_{T_{>1}}(V^{-1}_{T_{>1}}(T_1))[1] = f_{T_{>1}}(T_1)[1] = P[1].
    \end{align*} The first equality above follows from the definition of $E$, noting that $V^{-1}_{T_{>1}}(T_1)$ must be generated by $T_{>1}$ or be a shifted projective, as it is a summand of the co-Bongartz complement of $T_{>1}$.
\end{proof}

\begin{remark}
    Note that the proof of the above lemma is dependent on the assumption that $M$ is an unsigned $\tau$-exceptional sequence.
\end{remark}

\begin{lemma}\label{lem:s2-on-tf-ordered}
    Let $T = M \oplus N$ be a TF-ordered $\tau$-rigid module such that $N$ is projective. Then 
    \[\overline{s}_2(M \oplus N) = E_N(M) \oplus N[1].\]
\end{lemma}

\begin{proof}
We have 
\begin{align*}
\overline{s_2}(M \oplus N) &= \psi^{-1} \circ s_2 \circ \psi (M \oplus N) \\&=
\psi^{-1} \circ s_2 (E_N(M) \oplus N) \\ &=  \psi^{-1} (E_N(M) \oplus N[1])
 \\& = E^{-1}_{N[1]}(E_N(M)) \oplus N[1]  
\\ &= E_N(M) \oplus N[1].    
\end{align*}

The last equality is given by \cref{defn:Emap}, using that
$M \not \in  \Gen N$ by assumption.
\end{proof}

Given any two TF-orders $(M,\sigma)$ and $(M,\phi)$ on the same $\tau$-rigid module $M$, there is  of course a permutation $\theta$ such that $\theta(M,\sigma) = (M,\phi)$, and any such permuation may be decomposed into a composition of transpositions. Below we show that there is such a permutation $\theta$ which can be written as a compositions of transpositions such that applying these transpositions in order to $(M,\sigma)$ yields a TF-order at every step

\begin{proposition}\label{lem:TFcombinatorialalgo}
    Let $A = (T,\sigma_A)$ and $B = (T,\sigma_B)$ be two TF-orders on the same underlying module $T$ with $|T| = t$. Then there is a sequence of transpositions $\overline{\pi}_{b_1},\overline{\pi}_{b_2},\dots,\overline{\pi}_{b_i}$ such that \[\overline{\pi}_{b_i} \circ \overline{\pi}_{b_{i-1}}\circ \dots \circ \overline{\pi}_{b_1}(B) = A\] and $\overline{\pi}_{b_j} \circ \overline{\pi}_{b_{j-1}} \circ \dots \circ \overline{\pi}_{b_1}(T)$ is TF-ordered for all $1 \leq j \leq i$.
\end{proposition}

\begin{proof}
    For two TF-ordered modules $A' = (T',\sigma_{A'})$ and $B' =(T',\sigma_{B'})$ on the same underlying module $T'$ with $|T'| = t'$, we first define the agreeableness between $A'$ and $B'$ as \[\mathrm{Ag}(A',B') = \max\{i \in \{0,1,2,\dots,t'\} \mid A'_{\leq i} = B'_{\leq i}\}, \] so for instance $\mathrm{Ag}(A',B') = t'$ if and only if $A' = B'$ and $\mathrm{Ag}(A',B') = 0$ if and only if $A'_1 \neq B'_1$. Also, a sequence of transpositions $\overline{\pi}_{b_1},\overline{\pi}_{b_2},\dots,\overline{\pi}_{b_i}$ is called $B'$-\textit{admissible} if $\overline{\pi}_{b_j} \circ \overline{\pi}_{b_{j-1}} \circ \dots \circ \overline{\pi}_{b_1}(B')$ is TF-ordered for all $1 \leq j \leq i$.
    
    Let now $\overline{\pi}_{c_1},\overline{\pi}_{c_2},\dots,\overline{\pi}_{c_i}$ be a $B$-admissible sequence of transpositions maximising $$\mathrm{Ag}(A,\overline{\pi}_{c_i} \circ \overline{\pi}_{c_{i-1}} \circ \dots \circ \overline{\pi}_{c_1}(B)),$$ i.e such that for all $B$-admissible sequences of transpositions $\overline{\pi}_{d_1},\overline{\pi}_{d_2},\dots,\overline{\pi}_{d_r}$ we have \[\mathrm{Ag}(A,\overline{\pi}_{c_i} \circ \overline{\pi}_{c_{i-1}} \circ \dots \circ \overline{\pi}_{c_1}(B)) \geq \mathrm{Ag}(A,\overline{\pi}_{d_r} \circ \overline{\pi}_{d_{r-1}} \circ \dots \circ \overline{\pi}_{d_1}(B)).\]

    Let $\theta = \overline{\pi}_{c_{i}} \circ \overline{\pi}_{c_{i-1}} \circ \dots \circ \overline{\pi}_{c_1}$. If $\mathrm{Ag}(A,\theta(B)) = t$, then we have identified the wanted sequence of transpositions. Assume for the sake of contradiction that $\mathrm{Ag}(A,\theta(B)) = g < t$.
    
    Let $A' = A_{>g}$ and $B' = \theta(B)_{>g}$. We will now identify a permutation $\phi$ such that $\mathrm{Ag}(A',\phi(B')) \geq 1$ and such that $\phi$ can be written as the composition of a $B'$-admissible sequence of transpositions:
    
    Note that $A'$ and $B'$ are TF-ordered modules with the same underlying module and with agreeableness $\mathrm{Ag}(A',B') = 0$. Let $i$ be the unique integer so that $B'_i = A'_1$ and note that $i > 1$. For an integer $j \in \{1,2,\dots,i-1\}$, we now decompose the cycle \[\phi_j = (j \quad j+1 \quad \dots \quad i-1 \quad i) = \overline{\pi}_j \circ \overline{\pi}_{j+1} \circ \dots \circ \overline{\pi}_{i-1}\] and consider the two ordered modules \begin{align*}
        B' &= B'_1 \oplus \dots \oplus B'_{j-1} \oplus B'_j \oplus B'_{j+1} \oplus B'_{j+2} \oplus \dots \oplus B'_{i-1} \oplus \; \;B'_{i}\; \; \oplus B'_{i+1} \oplus \dots \oplus B'_{t-g} & \text{and } \\
        \phi_j(B') &= B'_1 \oplus \dots \oplus B'_{j-1} \oplus B'_i \oplus \; \;B'_{j} \;\; \oplus B'_{j+1}\oplus \dots \oplus B'_{i-2} \oplus B'_{i-1} \oplus B'_{i+1} \oplus \dots \oplus B'_{t-g}. 
    \end{align*}

    We argue that $\phi_j(B')$ is TF-ordered. Note that $A'_1 = B'_i = \phi_{j}(B')_j$, and since $A'$ is TF-ordered we have $A'_1 \notin \Gen A'_{>1} = \Gen(\phi_{j}(B')_{<j} \oplus \phi_{j}(B')_{>j})$, so $\phi_{j}(B')_j \notin \Gen(\phi_{j}(B')_{>j})$. Now, we need to verify that $\phi_{j}(B')_k \notin \Gen(\phi_{j}(B')_{>k})$ for $k \neq j$. For $k < j$ and $k > i$, we have $\phi_{j}(B')_k = B'_k$ and $\phi_{j}(B')_{>k} \cong B'_{>k}$ as modules, so since $B'$ is TF-ordered, we have $\phi_{j}(B')_k \notin \Gen(\phi_{j}(B')_{>k})$. Let now $j < k \leq i$. Then $\phi_{j}(B')_k = B'_{k-1}$ and $\phi_j(B')_{>k} \oplus B'_i \cong B'_{>k-1}$ so $\Gen(\phi_j(B')_{>k}) \subseteq \Gen(B'_{>k-1})$. Again since $B'$ is TF-ordered, $B'_{k-1} \notin \Gen(B'_{>k-1})$, so $\phi_{j}(B')_k \notin \Gen(\phi_j(B')_{>k})$. Thus $\phi_j(B')$ is TF-ordered for all $1 \leq j \leq i-1$.

    Observe now that $\mathrm{Ag}(A',\phi_1(B')) \geq 1$, as $A'_1 = \phi_1(B')_{1}$. But then adjusting the transpositions appearing in $\phi_1$ to act on $\theta(B)$ instead of on $\theta(B)_{>g} = B'$, we get a $B$-admissible sequence of transpositions $\overline{\pi}_{c_1},\overline{\pi}_{c_2},\dots,\overline{\pi}_{c_i},\overline{\pi}_{g+i-1}, \dots, \overline{\pi}_{g+1}$ such that \[\mathrm{Ag}(A,\overline{\pi}_{g+1} \circ \dots \circ \overline{\pi}_{g+i-1} \circ \overline{\pi}_{c_i}\circ \overline{\pi}_{c_i-1} \circ \dots \circ \overline{\pi}_{c_1}(B)) = g + \mathrm{Ag}(A',\phi_1(B')) \geq g + 1 > g,\] a contradiction to the maximality-assumption on $\overline{\pi}_{c_1},\overline{\pi}_{c_2},\dots,\overline{\pi}_{c_i}$.
\end{proof}

\subsection{Mutation of $\tau$-exceptional sequences of length $2$}

We now have the machinery to discuss mutation of $\tau$-exceptional sequences, as introduced in \cite{BHM2024}. In particular, we will give an explicit description of the action on TF-ordered $\tau$-rigid modules induced by left regular mutation on $\tau$-exceptional sequences of length $2$. We are therefore in this section primarily interested in \textit{regular} (left) mutation of $\tau$-exceptional sequences of length $2$ which are called $\tau$-exceptional \textit{pairs}. This mutation can be applied to \textit{left regular} pairs:

\begin{definition}\cite[Def. 3.11a]{BHM2024}
    Let $T = M \oplus N$ be a TF-ordered $\tau$-rigid $\Lambda$-module. Then $\Psi(T)$ is called \textit{left regular} if $N \in \proj \L$ or $N \notin \add(\bcomp(M))$. Otherwise $\Psi(T)$ is called \textit{left irregular}.
\end{definition}

We remark that there is also \textit{left irregular} mutation which may be applied to $\tau$-exceptional pairs that are not left regular but respect certain other technical conditions.

\begin{definition}\cite[Def. 4.1]{BHM2024}
    Let $T = M \oplus N$ be a TF-ordered $\tau$-rigid $\Lambda$-module. Then $\Psi(T)$ is called \textit{left mutable} if $\Psi(T)$ is left regular or the minimal torsion class containing $J(M \oplus N)$ is functorially finite.
\end{definition}

For $\tau$-tilting finite algebras, all $\tau$-exceptional pairs which are not left regular can be (irregularly) left mutated, and it follows from \cite[Cor. 4.8]{BHM2024} that this left irregular mutation can be expressed, in the $\tau$-tilting finite case, as a composition of the inverse of left mutations. This justifies in part our interest in regular (left) mutations, which we now proceed to describe.

We will use the operations $s_i$ and $\pi_i$ defined in the previous subsection to define (left) regular mutation $\varphi$: For a left regular $\tau$-exceptional pair $(B,C)$, the pair $\varphi(B,C)$ can be computed by applying $s_2$ if possible, then $\pi_1$ and lastly $s_1$ if after applying $\pi_1$ we do not have an unsigned $\tau$-exceptional pair. More precisely:

\begin{definition}\cite[Def.-Prop. 4.3]{BHM2024}\label{def:left_regular_mutation_pairs}
    Let $(B,C)$ be a left regular $\tau$-exceptional pair. Then the left mutation of $(B,C)$ is defined as \[\varphi(B,C) = 
	\begin{cases} s_1 \circ \pi_1 \circ s_2 & \text{if $C$ is projective,} \\ s_1 \circ \pi_1 & \text{if $\pi_1(B,C)$ is not unsigned,} \\
	\pi_1 & \text{otherwise.}
\end{cases}\] 
\end{definition}

\begin{remark}
    The above definition is a reformulation of the original definition of regular mutation of $\tau$-exceptional pairs given in \cite[Def.-Prop. 4.3]{BHM2024}. We here recall the original definition by \cite{BHM2024} and explain why the two formulations are equivalent.

    Let $(B,C)$ be a left regular $\tau$-exceptional pair, and let \[B' = \begin{cases}
        E^{-1}_{C[1]}(B) & \text{if $C$ is projective,} \\
        E^{-1}_C(B) & \text{otherwise.}
    \end{cases}\] 
    
    Then \cite{BHM2024} define the left regular mutation $\varphi(B,C)$ as \[\varphi(B,C) = \begin{cases} (|E_{B'}(C[1])|,B') & \text{if $C$ is projective,} \\ (|E_{B'}(C)|,B') & \text{otherwise,}\end{cases}\] where for a signed module $X = M[i]$ we let $|X| = M[0]$. 

We now verify that the above expression is equivalent to \cref{def:left_regular_mutation_pairs}. Assume first that $C$ is projective. Then \[s_1 \circ \pi_1 \circ s_2(B,C) = s_1 \circ \pi_1(B,C[1]) = s_1(E_{B'}(C[1]),B') = (|E_{B'}(C[1])|,B')\] as wanted, where the last equality follows from noting that $E_{B'}(C[1])$ necessarily has sign $1$.

  Assume now that $C$ is not projective. If $\pi_1(B,C)$ is not unsigned then \[s_1 \circ \pi_1(B,C) = s_1(E_{B'}(C),B') = (|E_{B'}(C)|,B').\] Lastly, if $\pi_1(B,C)$ is unsigned then \[\pi_1(B,C) = (E_{B'}(C),B') = (|E_{B'}(C)|,B').\] Thus in all cases, \cref{def:left_regular_mutation_pairs} agrees with the definition given in \cite[Def.-Prop. 4.3]{BHM2024}.
\end{remark}

\begin{definition}\label{def:irregularmut}
    Let $\Lambda$ be a $\tau$-tilting finite algebra and let $M \oplus N$ be TF-ordered such that $\Psi(M \oplus N)$ is a left irregular but left mutable $\tau$-exceptional pair. Define \[ X = \eproj_s(\Gen(\eproj_{ns}(\Filt(\Gen(J(M \oplus N)))))) \text{   and   } \quad Y = \eproj_{ns}(\Filt(\Gen(J(M \oplus N))))/X.\]
    Then the left mutation of $\Psi(M \oplus N)$ is defined as $\varphi(E_N(M),N) = (E_Y(X),Y)$. 

    Note that $\Filt(\Gen(J(M \oplus N))))$ is the smallest
    torsion class containing $J(M \oplus N)$, and see \cite[Sec. 4]{BHM2024} for further details concerning irregular mutation. 
\end{definition}

\begin{remark}
  We remark that for $\tau$-tilting finite algebras, all left irregular $\tau$-exceptional pairs are left mutable since all torsion classes are functorially finite in this case \cite[Thm. 2.7]{tau}.
\end{remark}

Mutation of $\tau$-exceptional pairs lifts to a mutation $\overline{\varphi}(M \oplus N) = X \oplus Y$ of TF-ordered $\tau$-rigid modules. In the next section, we will restrict to Nakayama algebras, and in this case we will, in \cref{thm:Nakayamacase4}, explicitly describe the modules $X$ and $Y$ as in \cref{def:irregularmut} in terms of $M$ and $N$.

We now describe how the action of left regular mutation on $\tau$-exceptional pairs and sequences translates to TF-ordered $\tau$-rigid modules. We denote by $\overline{\varphi}$ the induced operation on TF-ordered $\tau$-rigid modules coming from the operation $\varphi$ on $\tau$-exceptional pairs:

\begin{definition}        
        Let $T = T_1 \oplus T_2$ be a TF-admissibly ordered $\tau$-rigid module, such that $\Psi(T)$  is left mutable. We then define \[\overline{\varphi}(T) = \Psi^{-1} \circ \varphi \circ \Psi(T).\]
\end{definition}

\begin{theorem}
    
\label{prop:TFadmissiblebehaviour}
    Let $M \oplus N$ be a TF-admissibly ordered $\tau$-rigid module inducing a left regular $\tau$-exceptional pair. We then have \[\overline{\varphi}(M \oplus N) = \begin{cases}
    V_{E_N(M)}(N[1]) \oplus E_N(M) &\text{if $N$ is projective,} \\
    V_M(N) \oplus M & \text{if $N \in \Gen M$,} \\
N \oplus M & \text{otherwise.}\end{cases}\]
			\end{theorem}

\begin{proof}
    Assume first that $N$ is projective, which is equivalent to being in the first case in \cref{def:left_regular_mutation_pairs}. In that case we obtain
\begin{align*}
       \overline{\varphi}(M \oplus N) &= \overline{s}_1 \circ \overline{\pi}_1 \circ \overline{s}_2(M \oplus N) & (\text{Using \cref{def:left_regular_mutation_pairs}}) \\       
                               &= \overline{s}_1 \circ \overline{\pi}_1(E_N(M) \oplus N[1]) &  \text{(By \cref{lem:s2-on-tf-ordered})}
                               \\
                               &= \overline{s}_1 (N[1] \oplus E_N(M)) & 
        \text{(By definition of $\overline{\pi}_1$)}\\
                               &= V_{E_N(M)}(N[1]) \oplus E_N(M) & \text{(Using  \cref{lem:v-map-on-tf-ordered})}
\end{align*}

    The condition that $N \in \Gen M$ is equivalent to $N \oplus M$ not being TF-ordered or equivalently that $E_M(N)$ is a shifted projective so that $(E_M(N),M) = \pi_1(E_N(M),N)$ is not unsigned, so we are in the second case of \cref{def:left_regular_mutation_pairs}. We compute \[\overline{\varphi}(M \oplus N) = \overline{s}_1 \circ \overline{\pi}_1(M \oplus N) = \overline{s}_1(N \oplus M) = V_M(N) \oplus M\] where the last equality follows from \cref{lem:v-map-on-tf-ordered}.

    If the first two conditions are not met, we must be in the last case of \cref{def:left_regular_mutation_pairs} so we have $\overline{\varphi}(M \oplus N) = \overline{\pi}_1(M \oplus N) = N \oplus M$ as wanted.
\end{proof}

\subsection{Mutation of $\tau$-exceptional sequences of length $>2$}

The mutation of $\tau$-exceptional sequences of length $2$ induces a mutation on longer $\tau$-exceptional sequences.

\begin{definition}\cite[Def. 5.2]{BHM2024}
    Let $E = (X_1,X_2,X_3,\dots,X_n)$ be a $\tau$-exceptional sequence in $\mods \L$, induced by the TF-ordered $\tau$-tilting module $(T_1,T_2,\dots,T_n)$. Then $(X_1,X_2,\dots,X_{i},X_{i+1})$ is a $\tau$-exceptional sequence in $J(T_{>i+1})$, and $(X_{i},X_{i+1})$ is a $\tau$-exceptional pair in this category. If this pair is left regular (resp. left mutable) in $J(T_{>i+1})$, we say that $E$ is left $i$-regular (resp. left $i$-mutable). Assume $\varphi^{J(T_{>i+1})}(X_{i},X_{i+1}) = (Y_{i+1},Y_{i})$.
    Then we have \[\varphi_{i}(E) = (X_1,X_2,\dots,Y_{i+1},Y_{i},X_{i+2},\dots,X_n).\]
\end{definition}

The left mutation on longer $\tau$-exceptional sequences also lifts to a mutation on TF-ordered $\tau$-tilting modules which we now aim to describe.

\begin{definition}
    Let $E = (X_1,X_2,X_3,\dots,X_n)$ be a left $i$-mutable $\tau$-exceptional sequence induced by the TF-ordered $\tau$-tilting module $T = T_1 \oplus T_2 \oplus \dots \oplus T_n$. Then we define \[\overline{\varphi}_i(T) = \Psi^{-1} \circ \varphi \circ \Psi(T).\]
\end{definition}

To describe the above defined operation, it is a useful fact that $\tau$-tilting reductions are preserved by left mutations. More precisely, we have the following:
\begin{lemma}\cite[Def.-Prop. 4.3 and 4.4]{BHM2024}\label{prop:jasso_unchanged_after_mutation}
    Let $T = T_1 \oplus T_2$ be a $\tau$-rigid pair inducing a $\tau$-exceptional sequence $(X_1,X_2)$. If $T$ induces a left-mutable $\tau$-exceptional pair then $J(\overline{\varphi}(T)) = J(T)$.
\end{lemma}

The following was first observed by Mendoza and Treffinger.

\begin{proposition}\cite[Cor. 5.3]{mendozatreffinger_stratifyingsystems}\label{prop:tf-reduced-still-tf}
    Let $T_1 \oplus T_2 \oplus \dots \oplus T_t$ be a TF-ordered $\tau$-rigid module and $1 < i < t$. Then \[E_{T_{>i}}(T_{\leq i})\] can be considered an ordered object in $\stt J(T_{>i})$ and is further TF-ordered.
\end{proposition}

\begin{proof}
    Let $U = E_{T_{>i}}(T_{\leq i})$. Let $1 \leq j \leq i$. By \cref{thm:E-associative}, \[ E^{J(T_{>i})}_{U_{>j}}(U_j) = E_{T_{>j}}(T_j),\] which as $T$ is TF-ordered cannot be a shifted projective. Thus $U_j \notin \Gen{U_{>j}}$ for all $1 \leq j \leq i$.
\end{proof}

We now describe how $\overline{\varphi}_i$ acts on a TF-ordered $\tau$-rigid module.

 \begin{theorem}\label{thm:TFfullbehaviour}
    Let $T = T_1 \oplus T_2 \oplus \dots \oplus T_t$ be a TF-ordered $\tau$-rigid module in a module category $\mods \L$ such that $\Psi(T)$ is left $t-1$-mutable. Let $M = T_{t-1} \oplus T_t$ and $N = \overline{\varphi}(T_{t-1} \oplus T_t)$. Then we have
    \[\overline{\varphi}_{t-1}(T) = \left(\bigoplus_{1 \leq i \leq t-2} E^{-1}_N(E_M(T_i))\right) \oplus \overline{\varphi}(T_{t-1} \oplus T_t). \]
\end{theorem}

 \begin{proof}
 
    Let $U = \left(\bigoplus_{1 \leq i \leq t-2} E^{-1}_N(E_M(T_i))\right) \oplus \overline{\varphi}(T_{t-1} \oplus T_t)$. As $J(N) = J(M)$ by \cref{prop:jasso_unchanged_after_mutation}, $E_M(T_i)$ is in the image of $E_N$ for all $i$ and in particular $E^{-1}_N(E_M(\bigoplus_{1 \leq i \leq t-2}(T_i))) \in \str_N \L$ so $U$ is indeed $\tau$-rigid. We will now directly compute \[\Psi(U) = (E_{U>1}(U_1),E_{U>2}(U_2),\dots,U_t)\]
    and verify that $\Psi(U) = \varphi_{t-1}(\Psi(T))$. Note that by definition of $\varphi_{t-1}$, we have that $\Psi(T)$ and $\Psi(\overline{\varphi}_{t-1}(T))$ should agree in all but the last two coordinates. Also by definition of $\varphi_{t-1}$, the last two summands of $\overline{\varphi}_{t-1}(T)$ are given by $\overline{\varphi}(T_{t-1} \oplus T_t)$. We now check that $\Psi(U)$ agrees with $\Psi(T)$ in the first $t-2$ coordinates.
    
    We decompose $U = U' \oplus \overline{\varphi}(T_{t-1} \oplus T_t) = U' \oplus N$ and $T = T' \oplus M$. Let now $1 \leq i \leq t - 2$. Then
    \begin{align*}
         E_{U_{>i}}(U_i) &= E^{J(N)}_{E_N(U'_{>i})}(E_N(U_i)) & \text{(by \cref{thm:E-associative})}\\
                      &= E^{J(N)}_{E_N(U'_{>i})}(E_N(E^{-1}_N(E_M(T_i))))  & \text{(by definition of $U$ and since $1 \leq i \leq t-2$)}\\
                      &= E^{J(N)}_{E_N(U'_{>i})}(E_M(T_i)) \\
                      &= E^{J(N)}_{E_M(T'_{>i})}(E_M(T_i)) & \text{(since $E_N(U'_j) = E_M(T'_j)$ for any $1 \leq j \leq t-2$)}\\
                      &= E^{J(M)}_{E_M(T'_{>i})}(E_M(T_i)) & \text{(by \cref{prop:jasso_unchanged_after_mutation})}\\
                      &= E_{T_{>i}}(T_i) & \text{(by \cref{thm:E-associative})}
    \end{align*}
    as wanted. 
 \end{proof}

 The following proposition is useful to describe the action of $\overline{\varphi}_i$ on a TF-ordered module with more than $i+1$ direct summands. In particular, combining \cref{thm:TFfullbehaviour} with \cref{prop:TFmutations_in_middle} shows that one can compute any left regular mutation $\overline{\varphi}_i$ using a generous amount of reductions ($E$-maps) and \cref{prop:TFadmissiblebehaviour}.

\begin{proposition}\label{prop:TFmutations_in_middle}
    Let $T = T_1 \oplus T_2 \oplus \dots \oplus T_t$ be a TF-ordered module that is left $i$-mutable for $i < t$. Then \[\overline{\varphi}_i(T) = E^{-1}_{T_{>i+1}}(\overline{\varphi}^{J(T_{>i+1})}_{i}(E_{T_{>i+1}}(T_{\leq i+1}))) \oplus T_{>i+1}.\]
\end{proposition}

\begin{proof}
    Let $U' = E^{-1}_{T_{>i+1}}(\overline{\varphi}^{J(T_{>i+1})}_{i}(E_{T_{>i+1}}(T_{\leq i+1})))$ and $U = U' \oplus T_{>i+1}$. We will verify that $\Psi(U) = \varphi_i(\Psi(T))$. Let first $j > i+1$. Then \[ \Psi(U)_j = E_{U_{>j}}(U_j) = E_{T_{>j}}(T_j) = \Psi(T)_j  = \varphi_i(\Psi(T))_j,\] where the last equality follows from noting that $\varphi_i$ leaves coordinates $j \notin \{i,i+1\}$ unchanged. Let now $j \leq i+1$. Then
    \begin{align*}
        \Psi(U)_j &= E_{U_{>j}}(U_j) \\
                  &= E^{J(T_{>i+1})}_{E_{T_{>i+1}}(U'_{>j})}(E_{T_{>i+1}}(U'_j)) & \text{(by \cref{thm:E-associative})} \\
                  &= E^{J(T_{>i+1})}_{(\overline{\varphi}^{J(T_{>i+1})}_i(E_{T_{>i+1}}(T_{\leq i+1})))_{>j}}(\overline{\varphi}^{J(T_{>i+1})}_i(E_{T_{>i+1}}(T_{\leq i+1}))_j) & \text{(by definition of $U'$)}\\
                  &= \Psi^{J(T_{>i+1})}(\overline{\varphi}^{J(T_{>i+1})}_i(E_{T_{>i+1}}(T_{\leq i+1})))_j & \text{(by definition of $\Psi$)}\\
                  &= (\varphi^{J(T_{>i+1})}_i(\Psi^{J(T_{>i+1})}(E_{T_{>i+1}}(T_{\leq i+1}))))_j & \text{(since $\overline{\varphi}_i=\Psi^{-1} \circ \varphi_i \circ \Psi$)}\\
                  &= (\varphi^{}_i(\Psi^{}(T_{\leq i+1} \oplus T_{>i+1})))_j & \text{}\\
                  &= \varphi^{}_i(\Psi^{}(T))_j.
    \end{align*}

    The second to last equality may be verified by noting that $\Psi^{J(T_{>i+1})}(E_{T_{>i+1}}(T_{\leq i+1}))$ is a sub-sequence of $ \Psi^{}(T_{\leq i+1} \oplus T_{>i+1})$. Indeed, for integers $t > i > j$ we can write
    \begin{align*}
        \Psi(T)_j &= E_{T_{>j}}(T_j) \\
                  &= E^{J(T_{>i})}_{E_{T_{>i}}(\bigoplus_{k = j+1}^{i}(T_k))}(E_{T_{>i}}(T_j)) & \text{(by \cref{thm:E-associative})} \\
                  &= E^{J(T_{>i})}_{(E_{T_{>i}}(T_{\leq i}))_{>j}}(E_{T_{>i}}(T_{\leq i})_{j}) \\
                  &= \Psi^{J(T_{>i})}(E_{T_{>i}}(T_{\leq i}))_j & \text{(by definition of $\Psi$).}
    \end{align*}
    Thus $B = \Psi^{}(T_{\leq i+1} \oplus T_{>i+1})$ and $C = \Psi^{J(T_{>i+1})}(E_{T_{>i+1}}(T_{\leq i+1}))$ are two $\tau$-exceptional sequences of length $t$ and $i+1$ respectively where the two agree in the first $i+1$ coordinates and $C$ lies in $J(T_{>i+1})$. Thus we must have $\varphi_i(B)_j = \varphi^{J(T_{>i+1})}_i(C)_j$ for all $j \leq i+1$. This finishes the proof.
\end{proof}

In the later sections we will use the above results in the context of Nakayama algebras. We  illustrate the results using a non-Nakayama algebra of rank $3$ in \cref{exmp:nonNak1}.

\section{Mutation of TF-ordered modules for Nakayama algebras}\label{sec:Nakmutation}

We now consider Nakayama algebras and give explicit formulas for mutation of their TF-ordered $\tau$-rigid modules, including the irregular case. For a thorough background on Nakayama algebras we refer to \cite[Chap. V]{assem_skowronski_simson_2006}, and now introduce the notions needed in this paper.

Recall that for a module $M \in \mods \L$, we have a descending sequence of submodules 
\begin{equation}\label{eq:radicalseries} M \supseteq \rad M \supseteq \rad^2 M \supseteq \dots \supseteq \rad^{\ell-1} M \supseteq \rad^\ell M = 0. \end{equation}
We call the least integer $\ell$ such that $\rad^\ell M = 0$ the \textit{(Loewy) length} of $M$ and denote it by $\ell(M)$. If the module $M$ has a unique composition series, or equivalently, see \cite[Lem. V.2.2]{assem_skowronski_simson_2006}, if \cref{eq:radicalseries} is a composition series, then $M$ is called \textit{uniserial}. 

\begin{definition}
An algebra $\L$ is a \textit{Nakayama algebra} if every indecomposable projective and every indecomposable injective $\L$-module is uniserial. 
\end{definition}

As a consequence of \cite[Thm. V.2.6]{assem_skowronski_simson_2006}, a basic finite-dimensional algbebra $\L$ is a Nakayama algebra if and only if it can be written as $\L \cong KQ/I$, where $Q$ is a union of quivers of the following types:
\begin{equation} \label{eq:Nakayamaquiver}
\begin{tikzcd}[column sep=5mm, row sep=3mm]
    &&&&&&&& 2 \arrow[ld] & 3 \arrow[l] \\
    \overrightarrow{\A}_n: &n-1 \arrow[r] & n-2 \arrow[r] &\dots \arrow[r] & 1 \arrow[r] &0, & \overrightarrow{\Delta}_n: &1 \arrow[rd] & && \vdots \arrow[lu]\\
    &&&&&&&& 0 \arrow[r] & n-1 \arrow[ru]
\end{tikzcd}\end{equation}
where $n \geq 1$.

For a connected Nakayama algebra, the underlying quiver is either $\overrightarrow{\A}_n$ of $\overrightarrow{\Delta}_n$. In particular, there exists a natural cyclic order $0<1< \dots < n-1<0$ on the vertices and hence on the simple modules which we want use to formulate the statements of certain results, see \cref{prop:Bongartzcompletion}, \cref{lem:cocompletion}, \cref{lem:indexorder}. For a disconnected Nakayama algebra, every connected component $Q'$ of its quiver is either $\overrightarrow{\A}_n$ of $\overrightarrow{\Delta}_n$, and therefore we have a cyclic order on the vertices of the subquiver $Q'$ given by $0_{Q'} < 1_{Q'} < \dots < (n-1)_{Q'} < 0_{Q'}$.
Let $\Lambda$ be a Nakayama algebra. The following characterisation of the indecomposable $\tau$-rigid modules is a useful observation for studying the $\tau$-tilting theory of Nakayama algebras.

\begin{proposition}\cite[Prop. 2.5]{Adachi2016}\label{prop:Adachilength}
Let $M$ be a non-projective indecomposable module, then $M$ is $\tau$-rigid if and only if $\ell(M) < n$. 
\end{proposition}

This allows us to describe non-projective indecomposable $\tau$-rigid modules $\Lambda$-modules uniquely by their simple top and their simple socle. If $\Lambda$ is connected we denote by $M_{s,t}$ the unique $\tau$-rigid module with $\topp M_{s,t} \cong S(t)$ and $\soc M_{s,t} \cong S(s+2)_n$, where the notation $(a+b)_n$ means addition modulo $n$. If $\Lambda$ is not connected, consider an indecomposable module $M$. Since $M$ is indecomposable, it is a module of the subalgebra coming from exactly one connected component $Q_M'$ of its quiver whose cyclic order is as previously described. If $\topp(M) \cong S(t_{Q'})$ and $\soc(M) \cong S(r_{Q'})$ we denote $M$ by $M_{s,t}^{Q'}$ where $(s+2)_{|Q'|}=r$ with addition modulo $|Q'|$. If statements of our results do not use the expressions $M_{s,t}$ or the cyclic order of vertices, they hold for arbitrary Nakayama algebras. If they do and the generalisation to disconnected Nakayama algebras involves more than what was described above, we give more details following the proofs. Importantly, all our main results hold for arbitrary (possibly disconnected) Nakayama algebras. We first make an observation concerning the possible TF-orders of a given $\tau$-rigid module.
\begin{lemma}\label{lem:NakayamaTFdecs}
    Let $\Lambda$ be an arbitrary Nakayama algebra. Let $M \in \mods \Lambda$ be a basic $\tau$-rigid module and let $\bigoplus_{i=1}^n P(i)^{a_i} \to M$ be the projective cover of $M$. Then
\[ \text{\# \{TF-orders of $M$\}} = \frac{|M|!}{\prod_{i=1}^n a_i!}. \]
\end{lemma}
\begin{proof}
    Let $X, Y \in \mods \Lambda$ be distinct indecomposable modules. From the uniseriality of $\Lambda$ we have that exactly one of $X \in \Gen Y$ or $Y \in \Gen X$ arises if and only if $\topp(X) \cong \topp(Y)$. The module $M$ admits $a_i$ indecomposable direct summands $M_{j_1}, \dots, M_{j_{a_i}}$ such that $\topp(M_{j-1}) \cong \dots \cong \topp(M_{j_{a_i}}) \cong S(i)$ and these admit a unique TF-order relative to each other. Any indecomposable direct summand $N \in \add M$ with projective cover $P(N) \not \cong P(i)$
    can be inserted in any position to give a TF-order. With this observation, the result is obtained by forming the multinomial coefficient. 
\end{proof}

As an immediate consequence we obtain the following.

\begin{corollary} \label{cor:Nakayamafactorial}
    The number of TF-orders a $\tau$-rigid module divides $|\Lambda|!$.
\end{corollary}

The following example illustrates that this behaviour is special.
\begin{example}
    Consider the algebra $KQ$ with $Q$ the inward orientation of $\A_3$ as drawn below:
    \begin{center} $Q: \quad$\begin{tikzcd}
    1 \arrow[r] & 2 &\arrow[l]  3
\end{tikzcd}\end{center} Then $T = P(1) \oplus I(2) \oplus P(3)$ is a $\tau$-tilting module. The only obstruction to an ordering of $T$ being a TF-order is if $I(2)$ appears before both $P(2)$ and $P(3)$. Thus $T$ has $4$ TF-orders.
\end{example}

We now collect some well-known facts about the module category of a Nakayama algebra.
\newcommand{\crefenum}[2]{%
\namecref{#1}~\hyperref[#2]{\labelcref*{#1}~\ref*{#2}}%
}

\begin{proposition}\label{prop:nakayama-collection-of-facts}
    
    Let $\Lambda$ be a connected Nakayama algebra. Let $M$ be 
    an indecomposable $\Lambda$-module with minimal projective presentation $P_{-1}^M \to P_0^M \to M \to 0$. Let $L$ be a submodule of $M$ with minimal projective presentation $P_{-1}^L \to P_0^L \to L \to 0$. Furthermore, let $M_{s,t}$ denote an indecomposable non-projective module with top $S(t)$ and socle $S(s+2)_n$.
    
    \begin{enumerate}[label=\upshape(\arabic*),ref=(\arabic*)]
        \item\label{list:1}  $M \cong P_0^M/ \rad^{\ell(M)} P_0^M$.
        \item\label{almost_split} For $M$ non-projective, there is an almost split sequence \[0 \to \tau M \to  (\rad P_0^M/\rad^{\ell(M)}P_0^M) \oplus (P_0^M/\rad^{\ell(M)+1}P_0^M) \to M \to 0\]
        where the map $\tau M \to \rad (P_0^M/\rad^{\ell(M)}P_0^M)$ is an epimorphism.
        \item\label{list:2} If $f: M \to N$ is an epimorphism
        of modules, then there exists an epimorphism $\rad^i f: \rad^i M \to \rad^i N$ for all $i \geq 0$.
        \item\label{list:3} $L \cong \rad^{\ell(M)-\ell(L)} M$.
        \item\label{list:4} If $M$ is non-projective, we have $P_{-1}^M \cong P_{-1}^L$.
         \item\label{list:11} If $M$ is non-projective, no proper non-zero submodule of $M$ is projective.
        \item\label{list:5} $\tau M_{s,t} = M_{(s-1)_n, (t-1)_n}$.
        \item\label{list:7} $\tau (\rad^i M_{s,t}) = \rad^i (\tau M_{s,t})$.
        \item\label{list:6} $\ell(M_{s,t}) = (t-s -1)_n$.
        \item\label{list:8} $\Hom(P(j),M_{s,t}) = \begin{cases}
            k & \text{if } j \in \{(s+2)_n,(s+3)_n,\dots,t \},\\ 0 & \text{if } j \in \{(t+1)_n,(t+2)_n,\dots,s, (s+1)_n\}.
        \end{cases}$
        \item\label{list:9}\label{cor:projectiverigid} $\Hom(P(j), \tau M_{s,t})=0$ if and only if $j \in \{ t, (t+1)_n, \dots, s\}$.
        \item\label{list:13} $\Hom(P(j), \tau(\rad^k M_{s,t} / \rad^i M_{s,t}))=0$ for all $j \in \{ t, (t+1)_n, \dots, s\}$, $i \geq 1$ and $k \in \{ 0, \dots, i-1\}$.
       \item\label{list:10}\label{lem:submodrigid}  Let $L',L$ be two non-isomorphic submodules,
       or two non-isomorphic quotient modules of $M_{s,t}$. Then $\Hom(L, \tau L') = 0$.
        \item\label{list:14}If $M$ is $\tau$-rigid then $\Hom((\rad^j M)/\rad^i M, \tau M)=0$ for $i \geq 1$ and $j \in \{0,\dots, i-1\}$.
        \item\label{list:dimhom} Let $N \in \mods \Lambda$ be indecomposable. If $\ell(M) \leq n$ or
        $\ell(N) \leq n$, then $\dim_K \Hom(M,N) \leq 1$.
        
    \end{enumerate}
\end{proposition}
\begin{proof}
    \begin{enumerate}
    \item[] \ref{list:1}, \ref{almost_split}, \ref{list:2} These are (reformulations of) \cite[Thm. V.3.5, Thm. V.4.1, Lem. V.1.1]{assem_skowronski_simson_2006}.
        
\item[\ref{list:3}] By uniseriality every indecomposable module $M$ has a unique maximal proper submodule $\rad M$, and every other proper submodule of $M$ is also a submodule of $\rad M$ and uniserial, thus it follows that $L \cong \rad^t M$ for some $t \geq 0$. The fact that $t=\ell(M)-\ell(L)$ is obvious.
        \item[\ref{list:4}] By \ref{list:1} we have $M \cong P_0^M /\rad^{\ell(M)}P_0^M$. Therefore by \ref{list:2} and \ref{list:3} we get an epimorphism
        \[\rad^{\ell(M)-\ell(L)}P_0^M \to L \cong \rad^{\ell(M)-\ell(L)}M.\] 
        By uniseriality, this implies that there is an epimorphism $P_0^L \to \rad^{\ell(M)-\ell(L)} P_0^M$ and using \ref{list:2} we obtain an epimorphism
    \[ \rad^{\ell(L)} P_0^L \to \rad^{\ell(M)} P_0^M.\]
    Now from (1) it follows that $\rad^{\ell(L)} P_0^L \cong \ker (P_0^L \to L)$ and $\rad^{\ell(M)} P_0^M \cong \ker(P_0^M \to M)$. If $M$ is non-projective then $\rad^{\ell(M)}P_0^M \cong \ker(P_0^M \to M)$ is non-zero. By uniseriality, the existence of the epimorphism $\rad^{\ell(L)} P_0^L \to \rad^{\ell(M)} P_0^M$ then implies $P_{-1}^M \cong P_{-1}^L$ as required.
    \item[\ref{list:11}]This is a direct consequence of 
    (\ref{list:4}).
        \item[\ref{list:5}] This is a reformulation of \cite[Thm. V.4.1]{assem_skowronski_simson_2006}.        
        \item[\ref{list:7}] Note first that \[\rad M_{s,t} = \begin{cases}
            M_{s,t-1} & \text{if $\ell(M) > 1$,} \\
            0 & \text{otherwise.}
         \end{cases}\]
         So by \ref{list:5} we have $\tau \rad M_{s,t} = \rad \tau M_{s,t}$. The result now follows by induction on $i$.
        \item[\ref{list:6}] By uniseriality and \ref{list:4}, the composition factors of $M_{s,t}$ are exactly 
        $$\{S(s+2)_n,S(s+3)_n,\dots,S(t)\},$$ all with multiplicity $1$, since $\ell(M_{s,t}) <n$. So $\ell(M_{s,t}) = (t-(s+2)+1)_n =
        (t-s-1)_n.$
        \item[\ref{list:8}] By \cite[Lem. III.2.11, Cor. III.3.6]{assem_skowronski_simson_2006} we have that $\dim_K \Hom(P(j), X)$
    equals the multiplicity of $\topp P(j) \cong S(j)$ in $X$. So we argue as in the proof of \ref{list:6}.
        \item[\ref{list:9}] This follows directly from \ref{list:8} and \ref{list:5}.  

    \item[\ref{list:13}] Note that any composition factor of $\tau (\rad^k M_{s,t}/\rad^i M_{s,t})$ is also a composition factor of $\tau M_{s,t}$ by \ref{list:7}.
    So if $\Hom(P(j),\tau (\rad^k M_{s,t}/\rad^i M_{s,t})) \neq 0$, then also
    $\Hom(P(j), \tau M_{s,t}) \neq 0$.
    The result then follows from \ref{list:9}.
    \item[\ref{list:10}] Let $L \cong \rad^k M_{s,t}$ for some $k \in \{0, \dots, \ell(M_{s,t})-1\}$, then the composition factors of $L$ are
$$\{S(s+2)_n, \dots, S(t-k)_n) \}.$$
By \ref{list:5}, we have $\soc(\tau L') = S(s+1)_n$.
If $\Hom(L, \tau L') \neq 0$, then $S(s+1)$ would be a composition factor of $L$. But then we would have $ \ell(L) \geq n$, which is a contradiction, and so $\Hom(L, \tau L') = 0$. A dual proof works for two non-isomorphic quotient modules.

 \item[\ref{list:14}] This follows directly from \ref{list:10}, using that ${^\perp(\tau M)}$ is closed under quotient modules since $M$ is $\tau$-rigid, see \cite[Thm. 2.10]{tau}.
\item[\ref{list:dimhom}] If $\ell(N) \leq n$, every composition factor of $N$ appears at most once. So if $\topp(M)$ is a composition factor of $N$, there is a
unique submodule $N'$ of $N$ with $\topp(M) = \topp(N')$.  Assume $f$ is a non-zero map from $M$ to $N$. Then 
$\im f = N'$.
 Since this holds for any such non-zero map,
we have $\dim_K \Hom(M,N) \leq 1$.
The case $\ell(M)\leq n$ is dual.
    \end{enumerate} 
\end{proof}

\newcommand{\crefnakayamalist}[1]{%
\crefenum{prop:nakayama-collection-of-facts}{#1}%
}

We remark that \crefnakayamalist{list:5} and \crefnakayamalist{list:7} hold for arbitrary Nakayama algebras using the adaptation of the cyclic order because the Auslander-Reiten translation is a ``local operation'' on the subalgebra of the corresponding connected component. \crefnakayamalist{list:8}, \crefnakayamalist{list:9}, \crefnakayamalist{list:13} and \crefnakayamalist{list:14} hold using the adapted notation and the inclusion of the case $\Hom(M_{s,t}^{Q'}, N)=0$ whenever $N$ does not correspond to a representation of $Q'$.

\subsection{Bongartz completions and $\tau$-tilting reduction}

The following observation will be crucial.

\begin{lemma}\label{prop:Nakayamareduction}
    Let $\L$ be an arbitrary Nakayama algebra and $(M, P)$ a $\tau$-rigid pair. Then the $\tau$-perpendicular subcategory $J(M,P)$ is Morita equivalent to a Nakayama algebra.
\end{lemma}
\begin{proof}
    We have that $M^\perp \cap {}^\perp \tau M \cap P^\perp \subseteq \mods \L$ is equivalent the module category of some algebra $\Gamma$ by \cite[Thm. 3.8]{jassoreduction} and \cite[Thm. 4.12]{dirrt}. Since every indecomposable module of $\mods \L$ is uniserial, so is every indecomposable module of its subcategory. As a consequence, every module of $\Gamma$ is uniserial and thus $\Gamma$ is a Nakayama algebra. 
\end{proof}

\begin{remark}
    Similar results for $\tau$-tilting reductions have been obtained for the following classes of algebras:
    \begin{itemize}[label={\tiny\raisebox{0.7ex}{\textbullet}}]
        \item hereditary algebras, by \cite[Prop. 1.1]{GeigleLenzing1991}, for which the $\tau$-perpendicular categories correspond precisely with the perpendicular categories considered in \cite{GeigleLenzing1991}.
        \item gentle algebras, by \cite[Thm. 1.1]{Schroer1999} where it is shown that $\End(M)$ is a gentle algebra for any module $M$ without self-extensions over a gentle algebra $\L$. Since $\tau$-rigid modules have no self-extensions by \cite[Thm. 5.8]{auslandersmalo81} and $J(M)$ is Morita equivalent to $ \End(\bcomp(M))/\langle e_M\rangle$ by \cite[Thm. 3.8]{jassoreduction}, the result follows.
        \item preprojective algebras of Dynkin type ADE, by \cite[Thm. A]{Marks2016} since a $\tau$-tilting reduction $\Gamma$ of a preprojective algebra $\L$ of Dynkin type ADE is equivalent to considering a functorially finite wide subcategory of $\mods \L$ by \cite[Thm. 3.8]{jassoreduction}. Hence the composition of the equivalence with the canonical inclusion define a weakly homological embedding $\mods \Gamma \to \mods \L$, see \cite[Prop. 3.2]{Marks2016}, and by \cite[Thm. A]{Marks2016} we have that $B$ is Morita equivalent to a direct product of preprojective algebras of Dynkin types ADE.
    \end{itemize} 
\end{remark}

\cref{prop:Nakayamareduction} allows us to use results on mutation of $\tau$-exceptional pairs over Nakayama algebras to infer properties of mutation of longer $\tau$-exceptional sequences over Nakayama algebras.

Generalising a result of Msapato \cite[Prop. 6.3]{Msapato2022}, we now give explicit descriptions of the Bongartz and co-Bongartz completions of an indecomposable module of a Nakayama algebra. It is clear that if $M$ is projective, then $\bcomp(M) \cong \L$, otherwise we have the following description.

\begin{proposition} \label{prop:Bongartzcompletion}
    Let $\L$ be a connected Nakayama algebra and $M \in \mods \L$ be an indecomposable non-projective $\tau$-rigid module such that $\topp(M)= S(t)$. The Bongartz completion $\bcomp(M)$ is given by the $\tau$-tilting module
    \[\bcomp(M) = M \oplus \bigoplus_{i=1}^{\ell(M)-1} \rad^i M \oplus 
    \bigoplus_{i=0}^{n-\ell(M)-1} P(t+i)_n. \]
\end{proposition}

\begin{proof}
    Let $U$ denote the right-hand side of the above, so that we want to show $\bcomp(M) \cong U$. Note that 
    $(t+n-\ell(M)-1)_n = s$, by \crefnakayamalist{list:6}. Then $U$ is $\tau$-rigid since $\Hom(P(j), \tau U) = 0$ for all $j \in \{ t, \dots, (t+(n-\ell(M)-1))_n\}$ follows from \crefnakayamalist{cor:projectiverigid} and the observation that $\tau(\rad^i M) \cong \rad^i (\tau M)$ by \crefnakayamalist{list:7}. Moreover, $\Hom(L,\tau U)=0$ for any submodule $L \subseteq M$ by \crefnakayamalist{lem:submodrigid}. By counting that $U$ has $n$ non-isomorphic direct summands, we obtain that it is $\tau$-tilting. 

    Assume now that $\bcomp(M) \not \cong U$, in other words, $\Gen U = {}^\perp \tau U \subsetneq {}^\perp \tau M = \Gen  \bcomp(M)$, where the equalities follow from \cite[Thm. 2.12]{tau} and the inclusion is clear since $M$ is a direct summand of $U$. Using \cite[Thm. 2.7]{tau}, let $U'$ be a $\tau$-tilting module such that $\Gen U \subsetneq \Gen U' \subseteq \Gen \bcomp(M_{s,t})$ where the first inclusion is a covering relation in $\ftors \L$. Then by \cref{prop:tau_mutation_air} there exists an indecomposable direct summand $X \in (\add U \cap \add U') \setminus \add M$ which is generated by some direct summand $Y \in \add U'$. Clearly $X$ is not projective, so $X$ is a proper submodule $\rad^k M$ of $M$ for some $k \in \{1, \dots, \ell(M)-1\}$. However, any such module $Y$ would have a non-zero homomorphism to $\tau M$ thus making $U'$ not $\tau$-rigid. A contradiction, hence $\bcomp(M) \cong U$.
\end{proof}

If $\Lambda$ is not connected in \cref{prop:Bongartzcompletion}, the statement may be adapted by changing the cyclic order to be local to $Q_M'$ and adding the direct sum of all projective $\Lambda$-modules which do not correspond with representations of $Q_M'$. As a consequence we can describe the mutation of TF-ordered modules of radical square zero Nakayama algebras using fewer cases.

\begin{corollary}
    Let $\L$ be a radical square zero Nakayama algebra, then there are no left irregular $\tau$-exceptional pairs, in other words, a TF-ordered module $B \oplus C$ never falls into Case TF-4. Furthermore $B \oplus C$ never falls into Case TF-2a.
\end{corollary}
\begin{proof}
    Assume $(f_C (B),C)$ is a left irregular $\tau$-exceptional pair then by definition $C  \in \add (\bcomp(B))$ and $C \not \in \proj \L$. If $B$ is projective, then $\bcomp(B)$ only has projective summands so $C$ would be projective. If $B$ is non-projective, then since $\L$ is a radical square zero Nakayama algebras, $\ell(B) = 1$ and by the description of $\bcomp(B)$ from \cref{prop:Bongartzcompletion}, the only non-projective direct summand of $\bcomp(B)$ is $B$ since it has no proper submodules. Therefore $C \in \add (\bcomp(B))$ implies $C \in \proj \L$. Thus, Case TF-4 may not occur. We also have that for a radical square Nakayama algebra $\ell(B)=2$ implies that $B$ is projective. Therefore there cannot exist a non-projective $B$ with a nontrivial proper quotient $C$. So Case TF-2a cannot arise.
\end{proof}

We now give an explicit description of the co-Bongartz completion of an indecomposable $\tau$-rigid module.
\begin{proposition}\label{lem:cocompletion}
    Let $\L$ be a connected Nakayama algebra and $M \in \mods \L$ be an indecomposable $\tau$-rigid module such that $\topp(M) = S(t)$. Then the co-Bongartz completion $\ccomp(M)$ is given by
    \[ \ccomp(M) = \left(M \oplus \bigoplus_{i = 1}^{\min(n-1,\ell(M)-1)} (M/\rad^i M),\bigoplus_{i = 1}^{n - \ell(M)} P((t +i)_n) \right). \]
\end{proposition}
\begin{proof}
    Denote the pair on the right hand side above by $(T,Q)$.
    Note that $\Gen M = \Gen T$, so by \cref{prop:co-bongartz-summand-characterization}, the pair $(T,Q)$ must be the co-Bongartz completion if it is a $\tau$-tilting pair.

    Assume that $\Hom(T,\tau T) \neq 0$. Then $\Hom(M,\tau T) \neq 0$, as $T \in \Gen M$. By assumption $M$ is $\tau$-rigid, so let $T'$ be an arbitrary indecomposable summand of $T$ not isomorphic to $M$. Thus $T'$ is non-projective and by \crefnakayamalist{list:5} the module $\tau T'$ has simple top given by $S(t - 1)_n$. Since $\ell(\tau T') < n$ it does not have $S(t)$ as a composition factor. Thus $\Hom(M,\tau T') = 0$ and therefore $T$ is $\tau$-rigid. It follows from \crefnakayamalist{list:8} that $\Hom(Q,M) = 0$ and therefore $\Hom(Q,T) = 0$ as $T$ is supported by the same simples as $M$. It only remains to verify that $(T,Q)$ has $n$ non-isomorphic indecomposable direct summands. Observe that if $\ell(M) \geq n$ then $T$ has $n$ indecomposable direct summands and $Q$ has $0$ indecomposable direct summands. Lastly, if $\ell(m) < n$ then $T$ has $\ell(m)$ summands and $Q$ has $n - \ell(m)$ summands, adding to $n$ indecomposable direct summands. This completes the proof.
\end{proof}

\begin{lemma}
    Let $\Lambda = \Lambda_1 \times \Lambda_2$ be a disconnected algebra and $T$ a $\tau$-rigid pair over $\Lambda$. Then $T$ can be decomposed $T = T_1 \oplus T_2$ where $T_i$ is naturally a $\tau$-rigid pair over $\Lambda_i$ for $i = 1,2$ and \[\bcomp_\Lambda(T) = \bcomp_{\Lambda_1}(T_1) \oplus \bcomp_{\Lambda_2}(T_2).\]
\end{lemma}

\begin{proof}
    Since the Bongartz-completion $\bcomp(T) = (M,P)$ is characterized by maximizing $\Gen M$ while having $T$ as a summand, it is enough to observe that when we decompose $M = M_1 \oplus M_2$ for $\Lambda_i$-modules $M_i$, we have $\ind(\Gen M) = \ind(\Gen M_1) \cup \ind (\Gen M_2)$. 
\end{proof}

By iteratively applying the above lemma, we can stitch together Bongartz completions of disconnected Nakayama algebras using \cref{lem:cocompletion}. It is clear that the corresponding technique also works for finding co-Bongartz-completions.

\subsection{Left regular mutation}
We now give concrete descriptions for the mutation of TF-ordered modules of the form $B \oplus C$ for Nakayama algebras. Because the $\tau$-tilting reduction of a Nakayama algebra is again a Nakayama algebra by \cref{prop:Nakayamareduction} and the description from \cref{thm:TFfullbehaviour}, we focus on the case where $B \oplus C$ are the two right-most modules in a TF-order of a $\tau$-rigid module. From \cref{prop:TFadmissiblebehaviour} we have a simple description whenever $B \oplus C$ falls into Case TF-3. The aim of this subsection is to describe the $V$-maps arising in \cref{prop:TFadmissiblebehaviour} in the regular cases, Case TF-1 and Case TF-2 for Nakayama algebras, refining the simplifications given more generally in \cref{lem:gen_mut}. We begin with the following useful observation.
\begin{lemma} \label{lem:case1homs}
    Let $B \oplus C$ be a TF-ordered $\tau$-rigid module. If $C \in \proj \L$ and $B \not \in \proj \L$, then $\Hom(C,B) =0$. 
\end{lemma}
\begin{proof}
    Assume for a contradiction that $\Hom(C,B) \neq 0$. By TF-admissibility we have $B \not \in \Gen C$, so there exists a proper submodule $\rad^k B \in \Gen C$ of $B$ for some $k \in \{1,\dots, \ell(B)-1\}$. Write $\rad^k B \cong C/\rad^{\ell(\rad^k B)} C$. Since $B \not \in \proj \L$, we have $\tau B \neq 0$. By \crefnakayamalist{list:5}, we have that $\soc \tau B = \soc (C/\rad^{\ell(\rad^k B)+1} C)$, and we have $\ell(C/\rad^{\ell(\rad^k B)+1} C) = \ell(C/\rad^{\ell(\rad^k B)} C) +1 \leq \ell(B) = \ell(\tau B)$. By uniseriality, it follows that
    \[ \Hom(C/\rad^{\ell(\rad^k B)+1} C, \tau B) \neq 0,\]
    which implies that $\Hom(C, \tau B)\neq0$. So $B \oplus C$ is not $\tau$-rigid, which is a contradiction. Hence $\Hom(C,B) = 0$.
\end{proof}

We are now able to describe Case TF-1 of mutating TF-ordered modules. 

\begin{proposition} \label{prop:Nakayamacase1}
    Let $B \oplus C$ be a TF-ordered $\tau$-rigid module. If $C \in \proj \L$, then
    \[ \overline{\varphi}(B \oplus C) \cong \begin{cases} 
    C \oplus B & \text{ if } \Hom(C,B) = 0, \\ B \oplus f_C (B) & \text{ otherwise.}
    \end{cases}\]
\end{proposition}
\begin{proof}
    Assume first that  $\Hom(C,B)=0$. Then we have $C[1]
    \in \add (\ccomp(B))$, by \cref{lem:cocompletion}. Moreover, since $B \oplus C$ is $\tau$-rigid we have $\Hom(C, \tau B) = 0$ and thus $C \in \add (\bcomp(B))$ by \cref{prop:Bongartzcompletion}. We therefore have $\overline{\varphi}(B \oplus C) =  V_B(C[1]) \oplus B = C \oplus B$ by \cref{prop:TFadmissiblebehaviour} and \cref{lem:gen_mut}(c). 

    Assume now that $\Hom(C,B) \neq 0$, and hence $f_C(B) \not \cong B$ and
    $t_C(B) \neq 0$. Then, we have $B \in \proj \L$ by \cref{lem:case1homs}. We claim that 
    $f_C(B)$ is $\tau$-rigid. By uniseriality, we have that $t_C(B)$ is the maximal submodule of $B$ with $\topp(t_C(B)) = \topp(C)$. Hence
    $\ell(f_C(B)) = \ell(B) - \ell(t_C(B)) <n$,  and
    therefore $f_C(B)$ is $\tau$-rigid by 
    \cref{prop:Adachilength}. 
    
    By definition of $f_C(B)$ we have $\Hom(C,f_C B)=0$ and thus $C[1] \in \add (\ccomp(f_C(B)))$. Since $f_C(B)$ is a quotient of $B$, \cref{prop:Bongartzcompletion} implies that $B \in \add (\bcomp(f_C(B)))$. Consider the following right exact sequence 
    \[ C \xrightarrow{g} B \xrightarrow{f} f_C(B) \to 0\]
    where $f$ is the projective cover of $f_C(B)$ and $g$ is the composition of the projective cover $g_1: C \twoheadrightarrow t_C(B)$ with the canonical inclusion $g_2: t_C(B) \hookrightarrow B$. Since $\ell(f_C(B)) < n$, we obtain that $\dim_K \Hom(B, f_C B)= 1$, by \crefnakayamalist{list:dimhom} and so the morphism $f$ is a minimal left $\add (f_C B)$-approximation. 
    We next claim that $g$ is a left $\add (\bcomp(f_C B ))$-approximation. For this, we use the description of $\bcomp(f_C B)$ from \cref{prop:Bongartzcompletion}, noting that any direct summand in $\bcomp(f_C B)$ is either a submodule of $f_C B$, or projective. Observe first that $\Hom(C,f_C B)=0$ implies $\Hom(C,L)=0$ for any submodule $L$ of $f_C B$. Now we have
    \[ \topp(C) \cong \topp(t_C(B)) \cong \topp(\rad^{\ell(f_C B)}B) \cong S(t-\ell(f_C B))_n ,\]
    where $B \cong P(t)$.
    Let $P(i) \in \add (\bcomp(f_C B ))$, then by \cref{prop:Bongartzcompletion} we have $$i \in \{ t, \dots, (t+(n-\ell(f_C B)-1))_n\}.$$ Since $(t-\ell(f_C B))_n < t < (t+(n-\ell(f_C B))-1)_n$ we get that any $h \in \Hom(C,P(i))$ factors through $P(t) \cong B$. Thus $g$ is a left $\add (\bcomp(f_C B))$-approximation. We have that $g$ is minimal, since 
    $\coker g = f_C B$ is indecomposable.
    We can now apply the construction in 
    \cite[Def. 3.1, part I(c)]{BuanMarsh2018}
    to obtain that $V_{f_C B}(C[1]) \cong B$.
    We conclude that $\overline{\varphi}(B \oplus C) = V_B(C[1]) \oplus f_C B = B \oplus f_C B$ by \cref{prop:TFadmissiblebehaviour}.
\end{proof}

In a similar we may describe Case TF-2.

\begin{proposition}\label{lem:Nakayamacase2}
    Let $B \oplus C$ be a TF-ordered $\tau$-rigid module. If $C \in \Gen B$, we have $C \cong B / \rad^i B$ for some $i \in \{1, \dots, \ell(B)-1\}$ and
    \[ \overline{\varphi}(B \oplus C) \cong \begin{cases} \rad^i B \oplus B & \text{ if } B \not \in \proj \L, \\ P(\rad^i B) \oplus B & \text{ if } B \in \proj \L, \end{cases} \]
    where $P(\rad^i B) \to \rad^i B$ is the projective cover of $\rad^i B$. 
\end{proposition}
\begin{proof}
   It is clear by \crefnakayamalist{list:1} that $C \cong B/ \rad^i B$ for some $i \in \{1, \dots, \ell(B)-1\}$ in this case. By \cref{prop:TFadmissiblebehaviour} we have $\varphi(B \oplus C) = V_B(C) \oplus B$. Assume first $B \not \in \proj \L$, then $\rad^i B \in \add (\bcomp(B))$ and $C \in \add (\ccomp(B))$, therefore \cref{lem:gen_mut}(b) gives $V_{B}(C) \cong \rad^k B$ as required. 
   
   Assume now that  $B \in \proj \L$. Then $\bcomp(B) \cong \L$ and there exists an exact sequence
    \[ P(\rad^i B) \xrightarrow{g} B \xrightarrow{f} C \to 0,\]
    where $f$ is the projective cover of $C$ and $g$ is the composition of the projective cover $g_1: P(\rad^i B) \to \rad^i B$ and the canonical inclusion $g_2: \rad^i B \hookrightarrow B$. Since $\ell(C) < n$ by \cref{prop:Adachilength}, we have $\dim_K \Hom(B, C)=1$ by \crefnakayamalist{list:dimhom} and hence $f$ is a minimal right $\add B$-approximation. We next claim that $g$ is
    a minimal left $\add B$-approximation. For this, let $h \in \Hom(P(\rad^i B), B)$, then consider the epi-mono factorisation of $h$ given by
    \[ P(\rad^i B) \xrightarrow{h_1} \rad^j B \xrightarrow{h_2} B, \]
    for some $j \in \{1, \dots, \ell(B)-1\}$. Using uniseriality and $\ell(C)< n$, it follows that $\ell(\rad^i B) \geq \ell(\rad^j B)$ and hence there exist an epimorphism $p: \rad^i B \to \rad^j B$. It follows that $h$ factors through $g$ and thus $g$ is a left $\add B$-approximation. It is clear that it is minimal. Using now that $f$ is a minimal right $\add B$-approximation, and that $g$ is a minimal left $\add B$-approximation, the 
    construction in \cite[Def. 3.1 Case 1(b)]{BuanMarsh2018} gives $V_B(C) \cong P(\rad^k B)$ as required.
\end{proof}

As a corollary of \cref{prop:Nakayamacase1} and \cref{lem:Nakayamacase2} we obtain the following explicit description of the $V$-maps in these cases:
\begin{corollary}\label{cor:Vmapregular}
    Let $\L$ be a Nakayama algebra and $B \oplus C$ be a TF-ordered $\tau$-rigid module. Then if $C \in \Gen B$ may be written as $C \cong B/\rad^i B$ we get
    \[
    V_B(C) \cong \begin{cases} \rad^i B & \text{ if $B \not \in \proj \L$,} \\ P(\rad^i B) & \text{ if $B \in \proj \L$}. \end{cases}
    \]
    And if $C$ is projective, then we get
    \[
    V_{f_C(B)}(C[1]) \cong \begin{cases} C & \text{ if $\Hom(C,B) = 0$}, \\
    B & \text{ if $\Hom(C,B) \neq 0$}.
    \end{cases}
    \]
\end{corollary}

\subsection{Left irregular mutation}
Assume for notational simplicity that $\Lambda$ is a connected Nakayama algebra in this section. All results hold for arbitrary Nakayama algebras since they concern the direct sum of two indecomposable modules $B \oplus C$ where $C$ is a submodule of $B$. Therefore $B$, $C$, their submodules and quotient moduels all lie in the same connected component, and the generalisation is straight-forward by simply adjusting the cyclic order for the projective direct summands in \cref{eq:widehatB}.
We complete the description of the mutation of TF-ordered $\tau$-rigid modules for Nakayama algebras by considering the irregular case, i.e. Case TF-4. By \cref{prop:Bongartzcompletion} and the definition of Case TF-4, this implies that $B \oplus C$ satisfies $B \not \in \proj \L$ and $C \cong \rad^i B$ for some $i \in \{1, \dots, \ell(B)-1\}$. The main result of this section is the following.

\begin{theorem} \label{thm:Nakayamacase4}
    Let $B \oplus C$ be a TF-ordered $\tau$-rigid module. If $B \not \in \proj \L$ and $C \cong \rad^i B$ for some $i \in  \{1, \dots, \ell(B)-1\}$, then 
    \[ \overline{\varphi}(B \oplus C) = B \oplus (B/C). \]
\end{theorem}

Recall that by \cref{def:irregularmut} we need to determine the non-split $\Ext$-projective modules of $\Filt (\Gen J(B \oplus C))$ in order
to describe irregular mutation. Throughout this subsection, fix $B \oplus C$ to be a TF-ordered module satisfying the assumptions of \cref{thm:Nakayamacase4}. In particular, assume $\topp(B) = S(t)$ and $C \cong \rad^i B$ for some integer $i \in \{1, \dots, \ell(B)-1\}$. We begin by considering the following related module
\begin{equation}\label{eq:widehatB} \widehat{B} = B \oplus B/C \oplus \bigoplus_{j=1}^{i-1} (\rad^j B)/C \oplus \bigoplus_{j=i+1}^{\ell(B)-1} \rad^j B \oplus \bigoplus_{j=0}^{n-\ell(B)-1} P(t+j)_n.\end{equation}

As a first step of the proof we show that $\widehat{B}$ is a $\tau$-rigid module. Due to the similarities with the construction of $\widehat{B}$ with $\bcomp(B)$ we are able to ignore morphisms between direct summands and Auslander-Reiten translations of direct summands that $\widehat{B}$ shares with $\bcomp(B)$.

\begin{proposition}\label{prop:Bhatrigid}
    The module $\widehat{B}$ is a $\tau$-rigid module.
\end{proposition}
\begin{proof}
    Let $L,L' \in \add\widehat{B}$ be indecomposable, we investigate $\Hom(L, \tau L')$ for all different cases: 
    \begin{itemize}[label={\tiny\raisebox{0.7ex}{\textbullet}}]
        \item If $L \cong P(t+j)_n$ and $L' \cong \rad^k B$ for some $j \in \{0, \dots, n- \ell(B)-1\}$ and $k \in \{0 \} \cup \{i+1, \dots, \ell(B)-1\}$ then $\Hom(L, \tau L')=0$ by \cref{prop:Bongartzcompletion}.
        \item If $L \cong P(t+j)$ and $L' \cong (\rad^k B)/C$ for some $j \in \{0, \dots, n- \ell(B)-1\}$ and $k \in \{ 0, \dots, i-1\}$ then $\Hom(L, \tau L') = 0$ by \crefnakayamalist{list:13}. From the case $j=0$ it follows also that $\Hom(B, \tau L')\neq 0$ since $\topp(B) = S(t)$.
        \item If $L \cong \rad^j B$ and $L' \cong (\rad^k B)/C$ for some $j \in \{ i+1, \dots, \ell(B)-1\}$ and $k \in \{0, \dots, i-1\}$, then we observe that the composition factors of $\tau L'$ consist exactly of the composition factors of $\tau(\rad^k B)$ except the composition factors of the submodule $\tau C$. Since $j>k$, $\topp( L)$ is a composition factor of $\tau C$. Hence $\Hom(L, \tau L')=0$ as required.
        \item If $L \cong \rad^j B$ and $L' \cong \rad^k B$ for some $j,k \in \{ 0\} \cup \{i+1, \dots, \ell(B)-1\}$, then $\Hom(L, \tau L')=0$ by \crefnakayamalist{list:10}.
        \item If $L \cong (\rad^j B)/C$ and $L' \cong (\rad^k B)/C$ for some $j,k \in \{0, \dots, i-1\}$, then $\Hom(L, \tau L')=0$ by \crefnakayamalist{list:10}.
        \item If $L \cong (\rad^j B)/C$ and $L' \cong B$ for some $j \in \{0, \dots, i-1\}$, then $\Hom(L, \tau L') = 0$ by \crefnakayamalist{list:14}.
        \item If $L \cong (\rad^j B)/C$ and $L' \cong \rad^k B$ for some $j \in \{0, \dots, i-1\}$ and $k \in \{i+1, \dots, \ell(B)-1\}$, then it follows directly from the sequence of proper inclusions $L' \hookrightarrow C \hookrightarrow \rad^j B$ that $\topp(\rad^j B)$ is not a composition factor of $\tau L'$ and hence $\Hom(L, \tau L') =0$. 
    \end{itemize}
    Therefore $\widehat{B}$ is a $\tau$-rigid module.
\end{proof}

As a next step we aim to understand the extension closure $\Filt (\Gen J(B \oplus C))$. We begin with the following lemma. 

\begin{lemma}\label{lem:case4extensions}
    Let $j \in \{0, \dots, n-2\}$. Then
    \[ \Filt (\Gen ( P(j) \oplus S(j+1))) = \Gen (P(j) \oplus P(j+1)).\]
    Moreover, if $\Hom(P(n-1), P(0)) \neq 0$, then the same holds for $P(n-1)$ and $S(0)$.
\end{lemma}
\begin{proof}
    Let $j \in \{0, \dots, n-1\}$ and assume that $\Hom(P(j), P(j+1)_n)\neq 0$. Since $P(j) \oplus P(j+1)_n$ is $\tau$-rigid, the subcategory $\Gen (P(j) \oplus P(j+1)_n)$ is a torsion class by \cite[Thm. 5.10]{auslandersmalo81}. Since $S(j+1)_n \in \Gen P(j+1)_n$ we obtain
    \[ \Gen (P(j) \oplus S(j+1)_n) \subseteq \Gen (P(j) \oplus P(j+1)_n)\]
    and therefore
    \[ \Filt (\Gen (P(j) \oplus S(j+1)_n)) \subseteq \Gen (P(j) \oplus P(j+1)_n), \]
    since torsion classes, in particular $\Gen (P(j) \oplus P(j+1)_n)$, are closed under extensions by definition. 
    
    Let us now prove the reverse inclusion. First we remark that $\Filt (\Gen (P(j) \oplus S(j+1)_n))$ is the smallest torsion class containing $\Gen (P(j) \oplus S(j+1)_n)$. We can write $S(j+1)_n \cong P(j+1)_n/\rad P(j+1)_n$ by \crefnakayamalist{list:1}. 
    By \crefnakayamalist{almost_split} there exists an almost split exact sequence
    \[ 0 \to S(j) \to P(j+1)_n / \rad^2 P(j+1)_n \to S(j+1)_n \to 0\]
    since $S(j) \cong \tau S(j+1)_n$. We observe that $S(j) \cong P(j)/\rad P(j) \in \Gen (P(j) \oplus S(j+1)_n)$. By extension closure we get $P(j+1)_n/\rad^2 P(j+1)_n \in \Filt(\Gen (P(j) \oplus S(j+1)_n))$. Now let $k \geq 2$ and assume $P(j+1)_n / \rad^{k} P(j+1)_n \in \Filt (\Gen (P(j) \oplus S(j+1)_n))$. Again, by  \crefnakayamalist{almost_split} there exists an almost split exact sequence
    \begin{align*} 0 \to P(j)/ \rad^k P(j) \to (P(j)/\rad^{k-1}P(j)) \oplus (P(j+1)_n/\rad^{k+1} P(j+1)_n) \\
    \to P(j+1)_n / \rad^k P(j+1)_n \to 0.\end{align*}
    Again, extension closure implies that $P(j+1)_n / \rad^{k+1} P(j+1)_n \in \Filt(\Gen (P(j) \oplus S(j+1)_n))$. For $k \geq \ell(P(j+1)_n)$, we therefore have $P(j+1)_n \cong P(j+1)_n/ \rad^k P(j+1)_n \in \Filt \Gen (P(j) \oplus S(j+1)_n)$. Since now $P(j) \oplus P(j+1)_n \in \Filt(\Gen (P(j) \oplus S(j+1)_n))$, we have shown the reverse inclusion and thus equality. 
\end{proof}

\begin{lemma}\label{lem:case4coverinJ}
    We have $P(t) \in \Filt (\Gen (J(B \oplus C)))$, where
    $\topp(B) = S(t)$.
\end{lemma}
\begin{proof}
    From our setup we have that $C \in \add (\bcomp(B))$, which implies that $\bcomp(B) \cong \bcomp(B \oplus C)$. Denote by $t_{B \oplus C}$ and $f_{B \oplus C}$ the torsion and torsion-free functors of the torsion pair $(\Gen (B \oplus C), (B \oplus C)^\perp)$ respectively. Since $B,C$ and $P(t)$ are uniserial, we either have $t_{B \oplus C}P(t) = t_CP(t)$ or $t_{B \oplus C} P(t) = t_B P(t)$. Consider $g: C \twoheadrightarrow t_C P(t)$. Then $\ker g$ is also submodule of $B$ and hence $t_C P(t)$ is a submodule of $B/\ker g$. Moreover, $B/\ker g$ is a submodule of $P(t)$ as $\ell(P(t))>  \ell(t_BP(t))$. It follows that $t_{B \oplus C}P(t) = t_B P(t)$. By uniseriality there is a unique indecomposable $Y$ making the following sequence exact 
    \[ 0 \to t_B P(t) \to P(t) \to Y \to 0. \]
    Therefore $Y \cong f_{B}P(t) \cong f_{B \oplus C}P(t)$. 

Let us first show that $B \in \Gen f_B P(t)$. If $\Hom(B, P(t))=0$, then $P(t) \in B^\perp$, hence $f_B P(t) = P(t)$ and as $\topp(B) = S(t)$ we have $B \in \Gen P(t)$. If now $\Hom(B, P(t))\neq 0$, then we necessarily have $\ell(f_B P(t)) > \ell(B)$ because their simple tops coincide and otherwise $\Hom(B, f_B P(t)) \neq 0$ by uniseriality. Again, uniseriality implies that $B \in \Gen f_BP(t)$.

    Now from \cref{prop:Bongartzcompletion} we have $P(t) \in \add (\bcomp(B \oplus C))$ and from \cref{defn:Emap} we get $f_B P(t) = f_{B \oplus C} P(t) \in J(B \oplus C)$. From the torsion pair $(\Gen(B \oplus C), (B \oplus C)^\perp)$ we get a short exact sequence
    \[ 0 \to t_B P(t) \to P(t) \to f_B P(t) \to 0.\]
    We have that $t_B P(t) \in \Gen B \subseteq \Gen f_B P(t)$, where the inclusion follows from $B \in \Gen f_B P(t)$. Since the outer terms of the above short exact sequence are in $\Gen f_B P(t) \subseteq \Gen J(B \oplus C)$, we obtain $P(t) \in \Filt(\Gen J(B \oplus C))$ by the extension closure of this torsion class. 
\end{proof}

We are now ready to show the main result of this subsection.

\begin{proof}[Proof of \cref{thm:Nakayamacase4}]
    We begin by showing that 
    \begin{equation}\label{eq:Nakcase4} J(B \oplus C) \subseteq \Gen \widehat{B} \subseteq \Filt(\Gen J (B \oplus C)),\end{equation}
    where $\widehat{B}$ is from \cref{eq:widehatB}. Since $\widehat{B}$ is a $\tau$-rigid module by \cref{prop:Bhatrigid}, the corresponding subcategory $\Gen \widehat{B}$ is a torsion class by \cite[Thm. 5.10]{auslandersmalo81}. Then, if \cref{eq:Nakcase4} holds, we must have $\Gen \widehat{B} = \Filt(\Gen J(B \oplus C))$, since $\Filt(\Gen J(B \oplus C))$ is the smallest torsion class containing $J(B\oplus C)$ by \cite[Lem. 3.1]{MarksStovicek}. It then follows from the bijection \cite[Thm. 2.7]{tau} that the indecomposable direct summands of $\widehat{B}$ are the $\Ext$-projective modules of $\Filt(\Gen J(B \oplus C))$. Since $B$ and $B/C$ are both generated by $P(t)$, which
    is a summand in $\widehat{B}$, they are non-split $\Ext$-projective and by \cite[Prop. 3.13g]{BHM2024} they are the only ones. Thus we get the result by \cref{def:irregularmut}, where $B \cong \eproj_s(\Gen (B \oplus B/C))$ and $Y \cong B/C$.

    First let us show that $J(B \oplus C) \subseteq \Gen \widehat{B}$. \cref{prop:Bongartzcompletion} states that no indecomposable direct summand $X$ of $\bcomplement{B \oplus C}$ is generated by $B \oplus C$. Therefore we have $E_{B \oplus C}(X) = f_{B \oplus C}(X)$ by \cref{defn:Emap}. Applying $f_{B \oplus C}$ component-wise to the indecomposable direct summands of $\bcomplement{B \oplus C}$ we obtain
    \[ \widetilde{B} = f_{B \oplus C}(\bcomp(B \oplus C)/(B \oplus C)) \cong \bigoplus_{j=1}^{i-1} (\rad^j B)/C \oplus \bigoplus_{j=i+1}^{\ell(B)-1} \rad^j B \oplus \bigoplus_{j=0}^{n-\ell(B)-1} f_{B \oplus C}P(t+j)_n \]
    where $C \cong \rad^i B$ for some $1 \leq i \leq \ell(B)-1$. The $\Ext$-projective module $\bcomplement{B \oplus C}$ maps to the projective module $\widetilde{B}$ by \cref{cor:Bong-proj} which generates $J(B \oplus C)$. We therefore have $J(B \oplus C) = \Gen_{J(B \oplus C)}\widetilde{B}$. Since $\widetilde{B} \in \Gen \widehat{B}$ we obtain $J(B \oplus C) \subseteq \Gen \widehat{B}$.
    
    To obtain $\Gen \widehat{B} \subseteq \Filt(\Gen J(B \oplus C))$ it suffices to show that $\widehat{B} \in \Filt(\Gen J(B \oplus C))$ since $\Gen \widehat{B}$ is the smallest torsion class containing $\widehat{B}$ as it is $\tau$-rigid by \cref{prop:Bhatrigid}. Because the non-projective direct summands of $\widehat{B}$ also occur as direct summands of $\widetilde{B}$ we only need to show $P(t+j)_n \in \Filt(\Gen J(B \oplus C))$ for all $j \in \{0,\dots, n- \ell(B)-1 \}$. From \cref{lem:case4coverinJ} we know $P(t) \in \Filt(\Gen J(B \oplus C))$. Since $B$ and $C$ are not projective, we get $f_{B \oplus C} (P(t+j)_n) \neq 0$ and hence $S(t+j)_n \in \Gen f_{B \oplus C}P(t+j)_n \subseteq \Gen J(B \oplus C)$ implies
    \[ P(t+j)_n \in \Filt(\Gen J(B \oplus C)),\]
    for all $j \in \{1, \dots, n- \ell(B)-1 \}$ by \cref{lem:case4extensions}. As a consequence $\Gen \widehat{B} \subseteq \Filt(\Gen J(B \oplus C))$ as required. This completes the proof.
\end{proof}

\section{Transitivity of mutation for Nakayama algebras}
In this section, we use the description of the mutation of TF-ordered modules of Nakayama algebras in the previous section to show that the mutation of complete $\tau$-exceptional sequences is transitive for Nakayama algebras. While this behaviour is exhibited by hereditary algebras, in general only the case of algebras of rank 2 is characterised \cite[Cor. 7.9]{BHM2024}. We recall from \cite{ringel_exceptional} that for hereditary algebras transitivity is achieved by relating every exceptional sequence with one consisting only of simple modules, see \cite[Thm. 2 and 3]{ringel_exceptional}.

In our case, the proof idea is as follows: First, given a $\tau$-tilting module $M$, we show that there exists a sequence of mutations of TF-orders between any two TF-orders of $M$. Second, given two $\tau$-tilting modules $M$ and $M'$, related by a mutation of $\tau$-tilting modules, we show that there exists a mutation of TF-orders between some TF-order of $M$ and some TF-order of $M'$. The connectivity of the exchange graph of $\tau$-tilting pairs $\Hasse(\stt \L)$ then gives the desired result. Throughout this section, let $A$ be an arbitrary Nakayama algebra. 

\begin{lemma}\label{lem:NakTFswapmut}
    Let $B \oplus C$ be a TF-ordered module such that $C \oplus B$ is also TF-ordered. Then there is a sequence of mutations of TF-orders from $B \oplus C$ to $C \oplus B$. 
\end{lemma}
\begin{proof}
    Assume first that $B$ or $C$ is non-projective, and that $\psi(B \oplus C)$ is left regular. Then \cref{prop:TFadmissiblebehaviour} gives us that $\overline{\varphi}(B \oplus C) = C \oplus B$ or $\overline{\varphi}(C \oplus B) = B \oplus C$. If the first equality holds, we are done. If the second equality is true then there must be an $i$ such that $\overline{\varphi}^i(B \oplus C) = C \oplus B$ as wanted, since $\L$ is $\tau$-tilting finite.

    Assume now that $B, C \in \proj \Lambda$ then $B \oplus C$ does not fall into Case TF-2. From \cref{thm:introNakmutation} we have that if $B \oplus C$ falls into Case TF-1a or Case TF-3, then
    $\overline{\varphi}(B \oplus C) = C \oplus B$. We consider the remaining two cases separately.

    In Case TF-1b, we apply \cref{prop:Nakayamacase1} and obtain $\overline{\varphi}(B \oplus C) = B \oplus f_C(B)$. As $B \in \proj \L$, the TF-ordered module $B \oplus f_C(B)$ falls into Case TF-2b. Applying \cref{lem:Nakayamacase2}, we find that $\overline{\varphi}(B \oplus f_C (B)) = P(t_C (B)) \oplus B = C \oplus B$ as wanted.

    Finally, we consider the Case TF-4, where $\psi(B \oplus C)$ is left irregular. We have $C \not \in \proj \L$ and $C \cong \rad^i B$ for some $1 \leq i \leq \ell(B)-1$ by \cref{prop:Bongartzcompletion}. By \cref{thm:Nakayamacase4}, we get that $\overline{\varphi}(B \oplus C) = B \oplus (B/C)$. From the assumption $C \not \in \proj \L$ and $C \in \add (\bcomp(B))$, it follows that $B \not \in \proj \L$. Thus $B \oplus (B/C)$ falls into Case TF-2a and hence $\overline{\varphi}(B \oplus (B/C)) = C \oplus B$ by \cref{lem:Nakayamacase2}.
\end{proof}

We are now able to show the first of the two main steps towards transitivity.
\begin{lemma}\label{lem:TFdecsconnected}
    Let $M = (T,\sigma)$ and $M' = (T,\phi)$ be to TF-orders of a $\tau$-rigid module $T$. Then there is a sequence of mutations $\{\overline{\varphi}_{a_i}\}_{i \in \{1,2,3,\dots,t\}}$ such that $\overline{\varphi}_{a_t} \circ \overline{\varphi}_{a_{t-1}} \circ \dots \circ \overline{\varphi}_{a_1}(M) = M'$.
\end{lemma}
\begin{proof}
    Note first that if $n = 2$ then the claim follows directly from \cref{lem:NakTFswapmut}. 
    Recall that we denote by $\overline{\pi}_i$ the transposition $\begin{pmatrix} i & (i+1)\end{pmatrix}$ acting on ordered $\tau$-rigid pairs. Assume that $\overline{\pi}_i(M)$ is also TF-ordered.

    Then by \cref{prop:tf-reduced-still-tf} both $E_{M_{>i+1}}(M_i) \oplus E_{M_{>i+1}}(M_{i+1})$ and $E_{M_{>i+1}}(M_{i+1}) \oplus E_{M_{>i+1}}(M_{i})$ are TF-ordered in $J(M_{>i+1})$ which by \cref{prop:Nakayamareduction} is the module category of a Nakayama algebra. Thus, by \cref{lem:NakTFswapmut}, we have 
    \[\left({\overline{\varphi}^{J(M_{>i+1})}}\right)^j(E_{M_{>i+1}}(M_i) \oplus E_{M_{>i+1}}(M_{i+1})) = E_{M_{>i+1}}(M_{i+1}) \oplus E_{M_{>i+1}}(M_{i})\] 
    for some $j \geq 1$. By \cref{thm:TFfullbehaviour} we then have \[\left({\overline{\varphi}_i^{J(M_{>i+1})}}\right)^j(E_{M_{>i+1}}(M_{\leq i+1})) = \overline{\pi}_i(E_{M_{>i+1}}(M_{\leq i+1})), \] and lastly by \cref{prop:TFmutations_in_middle} this gives us \[\overline{\varphi}_i^j(M) = E^{-1}_{M_{>i+1}}(\overline{\pi}_i(E_{M_{>i+1}}(M_{\leq i+1}))) \oplus M_{>i+1} = \overline{\pi}_i(M).\] This shows that any transposition of neighboring summands preserving TF-order can be obtained via iterated application of left mutations of the TF-ordered module $M$. By \cref{lem:TFcombinatorialalgo}, we have that $M'$ can be obtained from $M$ using a sequence of such transpositions, meaning it can also be obtained from $M$ using a sequence of left mutations of TF-ordered modules.
\end{proof}

Before investigating the second step in the proof of transitivity, we carefully describe the terms arising in the mutation of $\tau$-tilting modules.

\begin{lemma}\label{lem:minimality}
Let $M \xrightarrow{\alpha = (\alpha_j)_{j=1}^t} \bigoplus_{j=1}^t N_i$
be a minimal left $\add N$-approximation.
Then for any index $i$, the map $\alpha_i$ cannot factor through any combination of the maps $\alpha_j$ for $j \neq i$. This means more precisely that for any fixed index $i$ we can not have maps
$\beta_j \colon N_j \to N_i$ for each $j\neq i$, with the property that $\alpha_i = \sum_{j \neq i} \beta_j \alpha_j$. 
\end{lemma}
\begin{proof}
    Assume that $\alpha_i = \sum_{j \neq i} \beta_j \alpha_j$. Consider the endomorphism $\psi$ of $\bigoplus_{i=1}^t N_i$ which has diagonal entries $\psi_{jj}$ equal to 1 if $j \neq i$ and 0 if $j = i$ and for which $\psi_{ij} = \beta_j$. Since the $i$-th column is zero this is not an isomorphism. However, by construction $\psi \circ \alpha = \alpha$, a contradiction to minimality. 
\end{proof}

\begin{lemma}\label{lem:epiormono}
    Let $\L$ be a Nakayama algebra. Let $M,N \in \mods \L$ be indecomposable modules such that $M \oplus N$ is $\tau$-rigid. If $N$ is not projective, then any nonzero morphism $f \in \Hom(M,N)$ is either an epimorphism or a monomorphism.
\end{lemma}
\begin{proof}
    Assume $N$ is not projective so that $\tau N \not \cong 0$. Let $0 \neq f \in \Hom(M,N)$, and assume $f$ is not an epimorphism. Then $\im f = M/\rad^t M$ is a proper submodule of $N$, that is $\soc(M/\rad^t M) = \soc N$. If $t \neq 0$, we would have that $\soc(M/\rad^{t-1}M) = \soc(\tau N)$ by \crefnakayamalist{list:5} and hence we would get a non-zero map $M \to M/\rad^{t-1}M \to \tau N$. Thus, we have $t=0$, and so $\im f = M$, which means $f$ is a monomorphism.
\end{proof}

We recall the following theorem of Adachi, Iyama, Reiten.
\begin{theorem}\cite[Thm. 2.30]{tau}\label{thm:AIR}
Let $U$ be a sincere $\tau$-rigid module with $|U| = n-1$. Then
there exist indecomposable $\tau$-rigid modules $X,Y$, such that $U \oplus X$ and $U \oplus Y$ are $\tau$-tilting
and there is an exact sequence
 \begin{equation}\label{eq:exchange-sequence}
 X \xrightarrow{f} U' \xrightarrow{g} Y \to 0  
 \end{equation}
where $f$ is a minimal left $\add U$-approximation and $g$ is a right $\add U$-approximation. 
\end{theorem}

\begin{corollary}\label{cor:projgencogen}
    Let $\L$ be a Nakayama algebra and consider the setting of \cref{thm:AIR}. If $X$ is projective, then any indecomposable $U'' \in \add U'$ is either projective or generated by $X$. If $X$ is not projective, then any indecomposable $U'' \in \add U'$ is either projective, a cogenerator of $X$ or generated by $X$.
\end{corollary}
\begin{proof}
    Assume that $X$ is not projective, then this follows directly from \cref{lem:epiormono}, noting that a nonzero morphism from an indecomposable projective module to an indecomposable nonprojective module is not a monomorphism by \crefnakayamalist{list:11}.
    
    Assume that $X$ is projective and let $U'' \in \add U'$ be indecomposable and nonprojective. 
    Let $f' \in \Hom(X, U'')$ be a non-zero corestriction of $f$ to $U''$, and apply \cref{lem:epiormono}. If $f'$ is a monomorphism then $U''$ is a cogenerator of $X$, and if $f$ is an epimorphism then $X$ generates $U''$. 
\end{proof}

\begin{lemma} \label{lem:mutationsseqprops}
Let $\L$ be a Nakayama algebra and consider the setting of 
\cref{thm:AIR}. We then have the following.
    \begin{enumerate}
        \item  $U'$ has at most one indecomposable projective direct summand (and with multiplicity 1).
        \item For any two non-isomorphic indecomposable  direct summands $U_1'', U_2'' \in \add U'$ we have $U_1'' \not \in \Gen U_2''$ and $U_2'' \not \in \Gen U_ 1''$. 
    \end{enumerate}
\end{lemma}
\begin{proof}
    Let $U'' \in \add U'$ be indecomposable and consider the corestriction of $f$ to $U''$, denoted by $f|^{U''}: X \to U''$. The minimality of $f$ implies that $f|^{U''} \neq 0$. 
    
    Non-projective indecomposable direct summands of $U'$ have length smaller than $n$, by \cref{prop:Adachilength}.
    It then follows from \crefnakayamalist{list:dimhom} and \cref{lem:minimality} that $U_1'' \not \cong U_2''$ for all non-projective indecomposable direct summands $U_1'', U_2'' \in \add U'$. 
    \begin{enumerate}
        \item 
        Assume for contradiction that $U_1 \oplus U_2$ is a direct summand of $U'$, with  $U_1$ and $U_2$ both indecomposable
        projectives. For $i=1,2$, let $f_i : X \to U_i$ be the corestrictions of $f : X \to U$, and
        let $\gamma : P_X \to X$ be the projective cover.
        It follows from 
        uniseriality that (up to reordering) $f_2 \gamma = \beta f_1 \gamma$
        for some $\beta \in \Hom(U_1, U_2)$. Since $\gamma$ is an epimorphism, it follows that $f_2 = \beta f_1$. So, by \cref{lem:minimality}, we have a contradiction to the minimality 
        of $f$. Hence $U'$ has at most one indecomposable direct summand, and with multiplicity one. 
        \item Without loss of generality, assume that there are non-isomorphic indecomposable direct summands $U_1'', U_2'' \in \add U'$ such that $U_2'' \in \Gen U_1''$. Then $f$ restricts to a minimal morphism 
        \[X \xrightarrow{f^{'} = \tiny\begin{pmatrix} f_1'' & f_2'' \end{pmatrix}^T} U_1'' \oplus U_2''.\]
        By the minimality of $f'$ we have $f_i'' \neq 0$ for $i=1,2$ and hence it follows from \cref{lem:epiormono} that $f_i''$ is either an epimorphism or a monomorphism or that $U_i''$ is projective. By (1) at most one of $U_1''$ and $U_2''$ can be projective. We distinguish between two cases.

        Assume $U_1''$ is projective. If $f_2''$ is an epimorphism, then $U_2'' \in \Gen X$. Since all indecomposable modules are uniserial we have that $\topp(U_1'') \cong \topp(U_2'') \cong \topp(X)$. As $U_1''$ is projective we obtain that $X \in \Gen U_1'' $. This is a contradiction to the assumption that the mutation at $X$ is a left mutation by \cite[Def.-Thm. 2.28]{tau}. Thus $f_2''$ has to be a monomorphism and thus $X$ a submodule of $U_2''$. As $U_2''$ is $\tau$-rigid and not projective we have that $\ell(U_2'')<n$  by \cref{prop:Adachilength}. Therefore we obtain $\ell(X) \leq \ell(U_2'')< n$. It follows from \crefnakayamalist{list:dimhom} and $f_i'' \neq 0$ that we have $\dim_K(\Hom(X, U_i'')) =1$ for $i=1,2$. Let $L$ be the kernel of the map $f_1'' \colon X \to U_1''$. Note that $L$ may be 0. Since $X$ is a subobject of $U_2''$, it follows that $L$ is also a subobject of $U_2''$. Since $X/L$ is a subobject of $U_1''$ and also of $U_2''/L$, and we have 
        $\ell(U_2''/L)  \leq \ell(U_2'') < \ell(U_1)''$, it follows by uniseriality that $(U_2'')/L$ is a subobject of $U_1''$. Note that the composition 
        \[X \overset{f_2''}{\hookrightarrow} U_2''
        \twoheadrightarrow U_2''/L \hookrightarrow U_1''\]
        is non-zero. Using that $\dim_K \Hom(X, U_1'')=1$
        by \crefnakayamalist{list:dimhom},
        we have that $\alpha h \circ f_2'' = f_1''$ for some $0 \neq \alpha \in K$. So $f_1''$ factors through $f_2''$ and we obtain a contradiction to minimality by \cref{lem:minimality}.

        If $U_1''$ is not projective, then we have $\ell(U_1'') < n$ by \cref{prop:Adachilength} and as $U_2'' \in \Gen U_1''$ and $U_1'' \not \cong U_2''$ we have $\ell(U_2'') < \ell(U_1'')$.
        Assume $f_1''$ is a monomorphism, then $X$ is a proper submodule of $U_1''$ since $X \not \in \add U'$. Consequently, $\topp(U_1'') \cong \topp(U_2'')$ is not a composition factor of $X$. It follows that $f_2''$ is not an epimorphism. Similarly $f_2''$ cannot be a monomorphism as otherwise $X$ is a proper submodule of $U_1''$ and a proper submodule of $U_2''$. But since $\ell(U_1'')<n$ this contradicts the assumption $U_2'' \in \Gen U_1''$. Assume $f_1''$ is an epimorphism,
        since $U_2'' \in \Gen U_1''$, there exists an epimorphism $h': U_1'' \to U_2''$. Because $\ell(U_2'') < \ell(U_1'')<n$, we have that $\topp(X)$ arises as a composition factor only once in $U_1''$ and $U_2''$ we have $\dim_K \Hom(X,U_i'')=1$ for both $i=1,2$. Since $h'$ and $f_1''$ are epimorphisms their composition is non-zero and we have $\beta h' \circ f_1'' = f_2''$ for some $0 \neq \beta \in K$. Thus, $f_2''$ factors through $f_1''$, a contradiction to minimality by \cref{lem:minimality}.
    \end{enumerate}
\end{proof}

The two next lemmas give further properties of  
the middle in the exchange sequence of \cref{thm:AIR}.
\begin{lemma}\label{lem:exchseqmiddleterm}
    Let $\L$ be a Nakayama algebra and consider the setting of \cref{thm:AIR}. Then there is a decomposition $U' \cong U_1 \oplus U_2$ where $U_1$ is indecomposable and $U_2$ is zero or indecomposable and such that
    \begin{itemize}[label={\tiny\raisebox{0.7ex}{\textbullet}}]
        \item $Y \in \Gen U_1$;
        \item $U_2 \in \Gen X$;
        \item $U_2 \in \Cogen Y$;
        \item $U_1 \not \in \Gen X$.
    \end{itemize}
\end{lemma}
\begin{proof}
    Let $X' = \im f$. Since $X$ is uniserial, its quotient module $X'$ is indecomposable. Consider the induced sequence
    \[ 0 \to X' \xrightarrow{f'} U' \xrightarrow{g} Y \to 0.\]
    By the uniseriality of $Y$ it is generated by at least one indecomposable direct summand of $U'$. Thus, let $U_1\in \add U'$ be indecomposable and such that $g_1 = g|_{U_1}$ is an epimorphism. Therefore we have $Y \in \Gen U_1$. Now consider the decomposition $U' \cong U_1 \oplus U_2$ with the morphism $g$ decomposing as 
    \[U_1 \oplus U_2 \xrightarrow{\tiny\begin{pmatrix} g_1 & g_2\end{pmatrix}} Y. \]
    First we show that $U_2$ is indecomposable or zero. Let $L = \ker g_1$ and consider the following commutative diagram with exact rows.
    \begin{equation} \label{eq:middleseq}
    \begin{tikzcd}[ampersand replacement=\&]
    0 \arrow[r]  \& L \arrow[r] \arrow[d,"h"] \& U_1 \arrow[r,"g_1"] \arrow[d, hookrightarrow] \& Y \arrow[r] \arrow[d, equal] \& 0 \\
       0 \arrow[r] \& X' \arrow[r,"f'"]  \& U_1 \oplus U_2 \arrow[r,"g"]  \& Y \arrow[r] \& 0 
    \end{tikzcd}
    \end{equation}
    where the map $h: L \to X'$ exists by the universal property of kernels. The snake lemma implies that $h$ is a monomorphism and that we get an isomorphism $(X')/L \cong U_2$. As $(X')/L$ is a quotient module of a uniserial module it is indecomposable or zero, hence so is $U_2$. Therefore we have $U_2 \in \Gen X$.

    Assume that $U_2$ is non-zero, then we show that $g_2$ is non-zero. Assume for a contradiction that $g_2=0$, then $X' \cong \ker g \cong \ker{\begin{pmatrix} g_1 & 0 \end{pmatrix}} \cong \ker g_1 \oplus U_2$. Since $X'$ is indecomposable and $U_2$ assumed to be non-zero we must have $\ker g_1 =0$, so that $U_1 \cong Y$, which is a contradiction to $Y \not \in \add U'$.

    By \cref{lem:epiormono} the morphism $g_2$ is either an epimorphism or a monomorphism. If it is an epimorphism then the uniseriality of $U_2$ and the non-projectivity of $Y$ imply that $Y \in \Gen U_2$. In particular, since also $Y \in \Gen U_1$, the uniseriality of all modules implies that $U_1 \in \Gen U_2$ or $U_2 \in \Gen U_1$, a contradiction to \cref{lem:mutationsseqprops}(1). It follows that $g_2$ is a monomorphism and therefore $U_2 \in \Cogen Y$.
    
    Finally, assume for a contradiction that $U_1$ is generated by $X$. By assumption $U_1 \not \cong X$ as $X \not \in \add U$, thus $U_1$ is a proper quotient module and hence not projective. Since $U_1$ is $\tau$-rigid we have $\ell(U_1)<n$ by \cref{prop:Adachilength} and thus $\dim_K \Hom(X,U_1) = 1$ by \crefnakayamalist{list:dimhom} and the observation that $f_1: X \to U_1$ is nonzero since $f$ is a minimal. Therefore $\im f_1 \cong U_1$ and hence $f = \begin{pmatrix} f_1 & f_2\end{pmatrix}$ is an epimorphism. It follows that $Y \cong U'/\im f \cong 0$, a contradiction.
\end{proof}

\begin{lemma}\label{lem:UprojthenXproj}
    Let $\L$ be a Nakayama algebra and consider the setting of \cref{thm:AIR} and the description $U' \cong U_1 \oplus U_2$ from \cref{lem:exchseqmiddleterm}. Then we have that if $U_1$ is projective, so is $X$. 
\end{lemma}
\begin{proof}
    Assume $U_1$ is projective. Then we consider the following 2-term complexes in $K^b(\proj \L)$ given by the minimal projective presentations:
    \[ \mathbb{P}_X = (0 \to P_{-1}^X \to P_0^X \to 0), \]
    \[ \mathbb{P}_Y = (0 \to P_{-1}^Y \xrightarrow{m} P_0^Y \to 0), \]
    \[ \mathbb{P}_{U'} = (0 \to P_{-1}^{U_2}  \xrightarrow{\tiny \begin{pmatrix} m' & 0 \end{pmatrix}^t} P_0^{U_2} \oplus U_1 \to 0). \]
    Note that $U_2$ may be zero in which case $P_0^{U_2}, P_{-1}^{U_2}$ and $m'$ are zero. The exchange sequence \cref{eq:exchange-sequence} is given by taking $H^0$ of the sequence $\mathbb{P}_X \to \mathbb{P}_{U'} \xrightarrow{g} \mathbb{P}_Y \to \Sigma \mathbb{P}_X$. In particular, we are interested in the mapping cone of the chain map $g: \mathbb{P}_{U'} \to \mathbb{P}_Y$ which is isomorphic to $\Sigma \mathbb{P}_X$. More precisely, we calculate the mapping cone of the following chain map:
    \[
    \begin{tikzcd}[ampersand replacement =\&]
        \dots \arrow[r] \& 0 \arrow[r] \arrow[d] \& P_{-1}^{U_2} \oplus 0 \arrow[r, "{\tiny\begin{pmatrix} m' & 0 \\ 0 & 0 \end{pmatrix}}"] \arrow[d, "{\tiny \begin{pmatrix} g_3 & 0 \end{pmatrix}}"] \& P_0^{U_2} \oplus U_1 \arrow[r] \arrow[d, "{\tiny \begin{pmatrix} g_1 & g_2\end{pmatrix}}"] \& 0 \arrow[r] \arrow[d] \& \dots \\
        \dots \arrow[r] \& 0 \arrow[r] \& P_{-1}^Y \arrow[r, "m"] \& P_0^Y \arrow[r] \& 0 \arrow[r] \& \dots
    \end{tikzcd}
    \]
    Note again, that if $U_2$ is zero then the map $g_3$ is zero. Since $Y \in \Gen U_1$ by \cref{lem:exchseqmiddleterm} and $g$ is a minimal approximation it follows that $g_2: U_1 \to P_0^Y$ is an isomorphism. Similarly, as $U_2 \in \Cogen Y$ and $g$ is a minimal approximation it follows that $g_3: P_{-1}^{U_2} \to P_{-1}^Y$ is an isomorphism (whenever $U_2$ is nonzero) by \crefnakayamalist{list:4}. We obtain the following description of the mapping cone and thus of $\Sigma \mathbb{P}_X$:
    \[
    \begin{tikzcd}[ampersand replacement =\&, column sep = 85, row sep = 65]
        \Sigma \mathbb{P}_X \cong \arrow[d, "h"] \& P_{-1}^{U_2} \arrow[r, "{\begin{pmatrix} -m' & 0 & g_3\end{pmatrix}^t}"] \arrow[d,"1"]  \& P_0^{U_2} \oplus U_1 \oplus P_{-1}^Y \arrow[r, "{\begin{pmatrix} g_1 & g_2 & m \end{pmatrix}}"] \arrow[d,"{\begin{pmatrix} 1 & 0 & m'g_3^{-1} \\ g_2^{-1} g_1 & 1 & g_2^{–1} m \\ 0 & 0 & 1 \end{pmatrix}}"]  \& P_0^Y \arrow[d,"1"] \\
        \Sigma \mathbb{P}_X \cong  \& P_{-1}^{U_2} \arrow[r, "{\begin{pmatrix} 0 & 0 & g_3\end{pmatrix}^t}"]  \& P_0^{U_2} \oplus U_1 \oplus P_{-1}^Y \arrow[r, "{\begin{pmatrix} 0 & g_2 & 0 \end{pmatrix}}"] \& P_0^Y
    \end{tikzcd}
    \]
    where the chain map $h$ is an isomorphism of chain complexes. We have thus split up the complex $\Sigma \mathbb{P}_X$ into a direct sum of indecomposable summands. Since any 2-term chain complex whose differential is an isomorphism is isomorphic to the zero object in $K^b(\proj \L)$, we obtain that some direct summands of $\mathbb{P}_X$ above are zero. In particular, if $U_2$ is zero, then $P_0^{U_2}, P_{-1}^{U_2}$ and the maps $m', g_3$ are zero and we obtain an isomorphism $\Sigma \mathbb{P}_X \cong P_{-1}^Y$ and $X$ is projective in this case. If $U_2$ is non-zero then $\mathbb{P}_X \cong P_0^{U_2}$ and so $X$ is projective. 
\end{proof}
\color{black}

Furthermore we also obtain the converse to \cref{lem:UprojthenXproj}.

\begin{lemma}\label{lem:XprojthenUproj}
    Let $\L$ be a Nakayama algebra and consider the setting of \cref{thm:AIR} and the description $U' \cong U_1 \oplus U_2$ from \cref{lem:exchseqmiddleterm}. We have that if $X$ is projective, so is $U_1$. Moreover, if $X$ is not projective then $X \cong \rad^i U_1$ for some $i \in \{1, \dots, \ell(U_1)-1\}$. 
\end{lemma}
\begin{proof}
    Assume $X$ is projective, then $U_1$ is either projective or generated by $X$ by \cref{cor:projgencogen}. Since it cannot be generated by $X$ by \cref{lem:exchseqmiddleterm} we have that $U_1$ is projective. Now assume $X$ is not projective, then $U_1$ is not projective by the contrapositive of \cref{lem:UprojthenXproj}. It follows from \cref{cor:projgencogen} that $U_1$ must be a cogenerator of $X$ since it cannot be generated by $X$ by \cref{lem:exchseqmiddleterm}. Hence $X \cong \rad^i U'$ for some $i \in \{1, \dots, \ell(U')-1\}$ as required. 
\end{proof}

We can now continue with the second step in the proof of transitivity. The following lemma provides the setting for relating the mutation of $\tau$-tilting modules with that of TF-ordered $\tau$-tilting modules.

\begin{lemma}\label{lem:mutationsetting}
    Let $\L$ be a Nakayama algebra and consider $U \oplus X$ and $U \oplus Y$ as in the setting of \cref{thm:AIR}. Let $U_1 \in \add U$ be the indecomposable direct summand which generates $Y$ as in \cref{lem:exchseqmiddleterm}. There exists a map $\sigma: \{1,\dots,n\} \to \ind(\add(U \oplus Y))$ satisfying the following two properties:
    \begin{enumerate}
        \item  $(U \oplus Y, \sigma)$ is TF-ordered and such that $U_1$ and $Y$ are adjacent, that is $\sigma(\sigma^{-1}(Y)-1) = U_1$. 
        \item $(U \oplus X, \sigma_X^Y)$ is TF-ordered, where 
        \[\sigma_X^Y(i) = \begin{cases} X & \text{ if } i = \sigma^{-1}(Y), \\ \sigma(i) & \text{ otherwise.} \end{cases} \]
    \end{enumerate}
    
\end{lemma}
\begin{proof}
    First we show that $U_1$ is the indecomposable module of minimal length in $\add U$ which generates $Y$. From \cref{lem:exchseqmiddleterm} we know that $U_1 \not \cong 0$ is such that $Y \in \Gen U_1$. Let $Z \in \add U$ be an indecomposable direct summand not isomorphic to $U_1$ and such that $Y \in \Gen Z$. Assume $\ell(Z) < \ell(U_1)$, i.e. $Z$ is a proper quotient of $U_1$ and thus not projective. We have that $\ell(Y) < \ell(Z) < \ell(U_1)$ from which it follows that $\Hom(X,Z) \neq 0$. 
    
    Next we claim that it follows that $\Hom(X,Z) \neq 0$. This is clear if $U_2 = 0$, and if $U_2 \neq 0$, then the composition $X \overset{f\mid^{U_2}}{\rightarrow} U_2 
    \overset{g\mid_{U_2}}{\rightarrow}Y$ is non-zero, since the first map is an epimorphism and the second map is a monomorphism by \cref{lem:exchseqmiddleterm}. So we have that both $\Hom(X,Y) \neq 0$ and $\Hom(X,U_1) \neq 0$. It then 
    follows by uniseriality that also $\Hom(X,Z) \neq 0$.
    
    Hence by \cref{lem:epiormono} we distinguish between two cases. If $Z \in \Gen X$ then it is generated by $X$ and $U_1$ so by the uniseritality of the indecomposable modules we have either $U_1 \in \Gen X$, a contradiction to \cref{lem:exchseqmiddleterm} or $X  \in \Gen U_1$, a contradiction to the assumption that the mutation at $X$ is a left mutation by \cref{prop:tau_mutation_air}. Otherwise, $X \in \Cogen Z$. Since $Z$ is non-projective, it follows from \crefnakayamalist{list:11}
    that $X$ is non-projective. Hence, by \cref{lem:XprojthenUproj}, also $U_1$ is non-projective.
    However then $\ell(X)< \ell(Z)< \ell(U_1) < n$ by \cref{prop:Adachilength}, which is a contradiction to the existence of monomorphisms $X \hookrightarrow U_1$ and $X \hookrightarrow Z$. Hence $U_1$ is the summand of minimal length in $U$ which generates $Y$.

     We proceed as follows. We first construct 
        a TF-ordering of $U \oplus Y$, satisfying (1).
        Then we replace $Y$ with $X$ in this ordering, 
        and modify it to produce a TF-ordering of
        $U \oplus X$. Then we replace $X$ with $Y$ in
        this ordering, to give us a new TF-ordering of
        $U \oplus Y$ which also satisfies (1),
        and in addition has the property that (2) is satisfied.
        By \cite[Prop. 3.2]{mendozatreffinger_stratifyingsystems} the module $U \oplus Y$ admits a TF-order $(U \oplus Y, \theta)$. We must have $\theta^{-1}(U_1) < \theta^{-1}(Y)$ since $Y \in \Gen U_1$. Since $U_1$ is of minimal length in $\add U$ generating $Y$, we obtain that
    \[ (U \oplus Y, \theta') = (U \oplus Y, \pi_{\theta^{-1}(Y)-2} \circ \dots \circ \pi_{\theta^{-1}(U_1)+1} \circ \pi_{\theta^{-1}(U_1)} \circ \theta) \]
    is TF-ordered. By construction $U_1$ and $Y$ are adjacent.
    If $(U \oplus X, (\theta')_X^Y)$ is TF-ordered, we are done. Otherwise, there exists an indecomposable direct summand $U' \in \Gen X$ such that $(\theta')^{-1}(U') < (\theta')^{-1}(Y)$. Let $U''$ be of minimal length among such indecomposable direct summands of $U$. Then
    \[ (U \oplus X, \theta'') = (U \oplus X, \pi_{\theta^{-1}(Y)} \circ \dots \circ \pi_{\theta^{-1}(U'')} \circ (\theta')_X^Y) \]
    satisfies $(\theta'')^{-1}(U'') > (\theta'')^{-1}(X)$ as required. We repeat this process for all such direct summands of $U$ to obtain a TF-ordered module $(U \oplus X, \theta''')$, with the property that $U_1$ is adjacent to $X$.
    
    We claim that $(U \oplus Y, (\theta''')_Y^X)$ is also TF-ordered. We have that $U_1$ generates $Y$, so if $X$ also generated $Y$, we would, by uniseriality, either have that $U_1 \in \Gen X$ or that  $X \in \Gen U_1$. By \cref{lem:exchseqmiddleterm} we have that $U_1 \not\in \Gen X$, and $X \not \in \Gen U_1 \subset \Gen U$, since $X$ is the Bongartz complement of $U$. So $X$ does not generate $Y$. The summands that were moved to the right of $X$ when defining $(U \oplus X, \theta''')$, all were factor modules of $X$, so none of these generates $Y$ either. 
    Hence, by construction we also have that $(U \oplus Y, (\theta''')_Y^X)$ is TF-ordered. This concludes the proof.
 \end{proof}
 
If $\L$ is a connected Nakayama algebra, recall that we assume its vertices to be labelled in (linear or cyclic) order from $0$ to $n-1$, which induces a cyclic order on the simple $\L$-modules $S(i)$ using addition modulo $n$ on their indices. 

\begin{lemma}\label{lem:indexorder}
        Let $\L$ be a basic connected Nakayama algebra and consider the setting of \cref{thm:AIR} with the description $U' \cong U_1 \oplus U_2$ from \cref{lem:exchseqmiddleterm}. We have $\topp(U_1) \leq \soc(\tau^2 Y) \leq \topp(X)$ in the cyclic order.
    \end{lemma}
    \begin{proof}
         Assume that $\topp(U_1) \cong S(q)$, $\topp(X) \cong S(t)$ and $\soc(\tau^2 Y) \cong S(r)$ for $q,r,t \in \{0,\dots, n-1\}$. Since $Y \in \Gen U_1$ by \cref{lem:exchseqmiddleterm} we have $\topp(Y) \cong \topp(U_1)$, which means that $Y \cong M_{r,q}$ by \cref{prop:nakayama-collection-of-facts} or equivalently $\soc(Y) \cong S(r+2)_n$. Since $\ell(Y)<n$ we have $\ell(Y) = (q-r-1)_n = (q-r-2)_n+1$. Consider the exchange sequence $X \xrightarrow{f} U_1 \oplus U_2 \to Y \to 0$ and let $X' \cong \im f$. Let us show that $\soc(X')\cong \soc(U_1)$. From \cref{eq:middleseq} we know that $U_1$ and $X'$ have a common submodule. Uniseriality implies that either $X' \in \Cogen U_1$ or $U_1 \in \Cogen X'$. In both cases we obtain $\soc(X')\cong \soc(U_1)$ by uniseriality.
         Because $X'\in \Gen X$ we have $\topp(X') \cong \topp(X)$. From this we obtain
        \begin{align*}
            (q-r-1)_n & = \ell(Y) \\
            & = \ell(U_1) + \ell(U_2) - \ell(X') \\
            & \geq \ell(U_1) - \ell(X')  \\
            & = (q-s+1)_n - (t-s+1)_n \\
            & = (q-t)_n
        \end{align*}
        where we assumed that $\soc(X') \cong \soc(U_1) \cong S(s)$. It follows that $q \leq r+1 \leq t$ in the cyclic order modulo $n$ and thus $\topp(U_1) \leq \soc(\tau Y) \leq \topp(X)$ in the induced cyclic order of simples. Moreover we have $(r+1)_n \neq q$ as otherwise $\ell(Y) =n$, whereas $\ell(Y) < n$ by \cref{prop:Adachilength} because $Y$ is $\tau$-rigid
        and non-projective. Therefore we obtain $q \leq r \leq t$ as required, and the result follows.
    \end{proof}

Recall that if $\L$ is not connected, we may consider a cyclic order on the simple $\L$-modules which lie in the same connected component of the underlying quiver. In this way \cref{lem:indexorder} holds for arbitrary Nakayama algebras when the cyclic order is changed to one for the connected component that $Y$ and by \cref{lem:mutationsseqprops} also $U_1,U_2$ and $X$ correspond to.

\begin{lemma}\label{lem:unchangedTFdec}
    Let $\L$ be a Nakayama algebra and let $T = (U \oplus X, \sigma_X^Y)$ and $T' = (U\oplus Y, \sigma)$ be as in \cref{lem:mutationsetting}. Recall that this means that $T_m = T_m'$ for all $m \neq i$. Then we have 
    \[ E_{T_{i-1} \oplus X}(T_j) \cong E_{T_{i-1} \oplus Y}(T_j)\]
    for all $j \in \{1, \dots, i-2\}$.
\end{lemma}
\begin{proof}
    Assume that $\topp(T_j) = S(p)$. Since $T$ and $T'$ are TF-orders by \cref{lem:mutationsetting}, we have $T_j \not \in \Gen(T_{i-1} \oplus X)$ and $T_j \not \in \Gen(T_{i-1} \oplus Y)$. Thus, by \cref{defn:Emap} we have
    \[ E_{T_{i-1} \oplus X}(T_j) = f_{T_{i-1} \oplus X}(T_j) \quad \text{and} \quad E_{T_{i-1} \oplus Y}(T_j) = f_{T_{i-1} \oplus Y}(T_j).\]
    Moreover, since $Y \in \Gen T_{i-1}$ we have $f_{T_{i-1} \oplus Y} = f_{T_{i-1}}$. If $\Hom(X, T_j) = 0$ then the desired result is obvious. We distinguish between two cases.

    First, assume $T_j$ is not projective, and let $\alpha \in \Hom(X, T_j)$ be a nonzero morphism. From \cref{lem:exchseqmiddleterm} we have that 
    \[f = \begin{pmatrix} f_1 & f_2 \end{pmatrix}^t: X \to U_1 \oplus U_2\]
    is an $\add U$-approximation, where $U_1 \cong T_{i-1}$. Since $T_j \in \add U$ there exists a homomorphism 
    \[h = \begin{pmatrix} h_1 & h_2 \end{pmatrix}: U_1 \oplus U_2 \to T_j\]
    such that $h \circ f = \alpha$. If $\Hom(U_2,T_j)=0$ then $h_2=0$ and we are done, since any map from $X$ to $T_j$ factors through $U_1$ i.e. $f_{X \oplus T_{i-1}}(T_j) = f_{T_{i-1}}(T_j)$. 
    
    Therefore assume that $\Hom(U_2, T_j)\neq 0$ so that in particular $U_2 \not \cong 0$. This implies $U_2 \cong T_k$ for $k > i$ in the TF-order of $T$ of \cref{lem:mutationsetting}, since $U_2 \in \Gen X$ by \cref{lem:exchseqmiddleterm}. From \cref{lem:epiormono} we obtain that $h_2 : U_2 \to T_j$ is a monomorphism, since $T_j$ is not projective and $T_j\not \in \Gen U_2$ as $j < i <k$ in the TF-order. Therefore $\soc(U_2) \cong \soc(T_j)$.

    By \cref{lem:epiormono} we obtain that $\alpha: X \to T_j$ has to be a monomorphism since $T_j$ is not projective and $T_j \not \in \Gen X$ since $X \cong T_i$ and $j<i$ in the TF-order. Moreover, this implies that $X$ cannot be projective by \crefnakayamalist{list:11}. As $X$ is not projective and since $\alpha$ is a monomorphism we obtain $\soc(X) \cong \soc(T_j)$ and by the previous paragraph $\soc(X) \cong \soc(U_2)$. This is a contradiction to the fact that $U_2 \in \Gen X$ and $\ell(U_2)< \ell(X) < n$ by \cref{prop:Adachilength}. Hence $\Hom(U_2, T_j)=0$.

    Assume now that $T_j$ is projective and let $\alpha \in \Hom(X,T_j)$ be a nonzero morphism. We distinguish between two cases. If $X$ is not projective then neither is $U_1$ and we have $X \cong \rad^l U_1$ for some $l \in \{1, \dots, \ell(U_1)-1\}$ by \cref{lem:XprojthenUproj}. Therefore $\ell(X) < \ell(U_1) < n$ by \cref{prop:Adachilength} and by \crefnakayamalist{list:dimhom} we have 
    \[ \dim_K \Hom(X, T_j) = 1 \geq \dim_K \Hom(U_1, T_j).\]
    From $\alpha$ we construct a nonzero morphism $U_1 \to T_j$ as follows: Since $X \cong \rad^l U_1$ it follows that $\ker \alpha$ is also a submodule of $U_1$, and therefore $U_1/\ker \alpha$ is a submodule of $T_j$ since $X/\ker \alpha \cong \im \alpha$. We then have that the following composition, which we denote by $\beta$, is nonzero
    \[ X \overset{f_1}{\hookrightarrow} U_1 \twoheadrightarrow U_1/\ker \alpha \hookrightarrow T_j \]
    since $\ker \alpha \not \cong X$ because $\alpha$ is nonzero. Since $\dim_K \Hom(X,T_j)=1$ we must have that $\alpha = a\beta$ for some $0 \neq a \in K$. So we have that $\alpha$ factors through $f_1: X \to U_1$ and hence $f_{T_{i-1} \oplus X}(T_j) = f_{T_{i-1}}(T_j)$ as required.

    Assume now that $X$ is projective, then $U_1 \cong T_{i-1}$ is also projective by \cref{lem:XprojthenUproj}. If $f_{T_{i-1} \oplus X}(T_j) \neq f_{T_{i-1}}(T_j)$ then we must have that $f_{T_{i-1} \oplus X}(T_j) = f_{X}(T_j)$ by the uniseriality of $T_j$ and moreover $t_{T_{i-1}}(T_j) \cong \rad^l t_X(T_j)$ for some $l \in \{1, \dots, \ell(t_X(T_j))\}$. In particular, this implies that we have $\topp(U_1) < \topp(X) < \topp(T_j)$ in the (strict) cyclic order. On the other hand we obtain from \cref{lem:indexorder} the relation $\topp(U_1) \leq \soc(\tau^2 Y) \leq \topp(X)$ in the induced non-strict version of the cyclic order. As $U \oplus Y$ is $\tau$-rigid we have in particular that $\Hom(T_j, \tau Y)=0$. By \crefnakayamalist{list:9} we have $\topp( Y) \leq \topp(T_j) \leq \soc(\tau^2 Y)$ in the induced non-strict cyclic order. Since $Y \in \Gen U_1$ by \cref{lem:exchseqmiddleterm}, have $\topp(Y) = \topp(U_1)$ by uniseriality and therefore $\topp( U_1) \leq \topp(T_j) \leq \soc(\tau^2 Y)$. Using the fact that $a \leq b \leq c$ and $a \leq c \leq d$ in a cyclic order implies $a \leq b \leq d$, we have that $\topp(U_1) \leq \topp(T_j) \leq \soc(\tau^2 Y)$ and $\topp(U_1) \leq \soc(\tau^2 Y) \leq \topp(X)$ in the cyclic orders implies $\topp(U_1) \leq \topp(T_j) \leq \topp(X)$, a contradiction to $\topp(U_1) < \topp(X) < \topp(T_j)$ in the cyclic order. Thus we obtain the desired result.
\end{proof}

\begin{proposition}\label{prop:mutationmutation}
    Let $T = (U \oplus X, \sigma_X^Y)$ and $T' = (U\oplus Y, \sigma)$ be as in \cref{lem:mutationsetting}. Then $\overline{\varphi}_{i-1}(T) = T'$ where $i = \sigma^{-1}(Y) = {\sigma^Y_X}^{-1}(X)$.
\end{proposition}
\begin{proof}
    Assume first that $i = n$. By \cref{lem:exchseqmiddleterm} the exchange sequence \cref{eq:exchange-sequence} is given by
    \[ X \xrightarrow{f} U_1 \oplus U_2 \to Y \to 0\]
    where $f$ is a minimal left $(\add U)$-approximation and $U_1 \cong T_{i-1}$ generates $Y$. No indecomposable direct summand of $U$ is generated by $X$ since $X$ is the right-most summand of the TF-ordered module $T$,  therefore $U_2 = 0$. Because $f: X \to U_1$ appearing in the exchange sequence \cref{eq:exchange-sequence} is a left $(\add U)$-approximation, $\im f$ is the maximal submodule of $U_1$ generated by $X$. However this implies that $\im f= t_X(U_1)$ and hence $Y\cong U_1/\im X \cong f_X(U_1)$ and thus $f_X(T_{i-1}) = f_X(U_1) \cong Y$.
    
    Since $X \not \in \Gen T_{i-1}$ by \cref{prop:tau_mutation_air} the TF-ordered module $T_{i-1} \oplus X$ does not fall into Case TF-2. Moreover, if $T_{i-1}$ is projective, then $\overline{\varphi}(T_{i-1} \oplus X)= T_{i-1} \oplus Y$ by Case TF-1b of \cref{thm:introNakmutation} and using that $f_X(T_{i-1})\cong Y$. If $T_{i-1}$ is not projective, then $X \cong \rad^k(T_{i-1})$ for some $k \in \{1, \dots, \ell(T_{i-1})-1\}$ by \cref{lem:XprojthenUproj}. In this case $\overline{\varphi}(T_{i-1} \oplus X) \cong T_{i-1} \oplus Y$ by Case TF-4 of \cref{thm:introNakmutation} and using that $f_X(T_{i-1})\cong Y$. It follows from \cref{lem:unchangedTFdec} and \cref{thm:TFfullbehaviour} that $\overline{\varphi}_{n-1}(T) \cong T'$ as required. 

    Assume now that $i < n$. Note that $T_{>i} = T'_{>i}$, so $E_{T_{>i}}(T_{\leq i})$ and $E_{T_{>i}}(T'_{\leq i})$ are two TF-ordered $\tau$-tilting modules in the category $J(T_{>i})$, which by \cref{prop:Nakayamareduction} is the module category of a Nakayama algebra. Further, \cref{thm:E-associative} gives us the equalities 
    \begin{equation}\label{eq:Egenpreserve}
        E^{J(T_{i-1})}_{E_{T_{i-1}}(T_{>i})}(E_{T_{i-1}}(Y)) = E_{T_{>i} \oplus T_{i-1}}(Y) = E^{J(T_{>i})}_{E_{T_{>i}}(T_{i-1})}(E_{T_{>i}}(Y)).
    \end{equation}
    By \cref{defn:Emap} we have that $E_{T_{i-1}}(Y)$ is a shifted projective and $E_{T_{>i}}(Y)$ is a module.  Hence \cref{eq:Egenpreserve} shows that the module $E_{T_{>i}}(Y)$ becomes a shifted projective after applying $E_{E_{T_{>i}}(T_{i-1})}^{J(T_{>i})}$ and therefore $E_{T_{>i}}(Y) \in \Gen(E_{T_{>i}}(T_{i-1}))$ by \cref{defn:Emap}, so we are in the setting of \cref{lem:mutationsetting}. Thus by the $i = n$ case dealt with above, 
    \[\overline{\varphi}_{i-1}^{J(T_{>i})}(T_{\leq i}) = E_{T>i}(T'_{\leq i}).\]
    Finally, by \cref{prop:TFmutations_in_middle} we have \[\overline{\varphi}_{i-1}(T) = E^{-1}_{T_{>i}}(E_{T>i}(T'_{\leq i})) \oplus T_{>i} = T'\] as wanted.
\end{proof}

We are now able to prove the main result of this section.
\begin{theorem}\label{thm:transitivemut}
    Mutation of TF-ordered $\tau$-tilting modules, and thus of complete $\tau$-exceptional sequences, is transitive for Nakayama algebras.
\end{theorem}
\begin{proof}
  Let $(M,\sigma)$ and $(N,\sigma_N)$ be two TF-ordered $\tau$-tilting modules. Since Nakayama algebras are representation finite and thus $\tau$-tilting finite, the mutation of $\tau$-tilting modules is transitive. Let $M - \mu_{a_1} M - \dots - \mu_{a_r} M \cong N$ be a sequence of mutations of $\tau$-tilting modules. It follows that there exists a corresponding sequence of mutations of TF-ordered $\tau$-tilting modules 
    \[ (M, \sigma) - (M, \sigma') - (\mu_{a_1} M, \sigma_1) - (\mu_{a_1} M, \sigma_1 ') - \dots - (\mu_{a_r} M, \sigma_r) - (\mu_{a_r} N, \sigma_N) \]
    where a mutation between two TF-orders of the same $\tau$-tilting module exists by \cref{lem:TFdecsconnected} and a mutation between two different $\tau$-tilting modules related by a single mutation exists by \cref{prop:mutationmutation}. Therefore mutation of TF-ordered $\tau$-tilting modules and and hence of complete $\tau$-exceptional sequences is transitive.
\end{proof}

\section{Examples and illustrations} \label{sec:examples}

In this section we give examples of the theory developed in this paper. Moreover, to illustrate the examples of $\tau$-rigid modules of Nakayama algebras we will use the geometric model for $\tau$-rigid developed in \cite{Adachi2016}. We begin with an example illustrating the regular mutation of TF-orders for general finite-dimensional algebras as described in \cref{thm:TFfullbehaviour}.

\begin{example}\label{exmp:nonNak1}
    Let $\Lambda =KQ/\langle ba,dc\rangle$ with $Q$ as below. Note that $\Lambda$ is a gentle algebra.
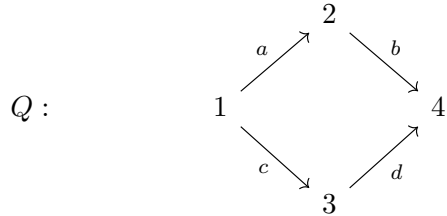
\[\begin{tikzcd}
	&&& 2 \\
	{Q:} && 1 && 4 \\
	&&& 3
	\arrow["b", from=1-4, to=2-5]
	\arrow["a", from=2-3, to=1-4]
	\arrow["c"', from=2-3, to=3-4]
	\arrow["d"', from=3-4, to=2-5]
\end{tikzcd}\] Then $T = P(1) \oplus P(2) \oplus P(3) \oplus P(4)$ is a TF-ordered $\tau$-tilting module. We will compute $\overline{\varphi}_3(T)$. Note that $\Psi(P(3) \oplus P(4))$ is left regular as $P(4)$ is projective. \cref{prop:TFadmissiblebehaviour} thus gives \[\overline{\varphi}(P(3) \oplus P(4)) = V_{E_{P(4)}(P(3))}(P(4)[1]) \oplus E_{P(4)}(P(3)).\] As $E_{P(4)}(P(3)) = S(3)$, we have $V_{E_{P(4)}(P(3))}(P(4)[1]) = V_{S(3)}(P(4)[1])$. Since $\mathbb{P}_{S(3)}$ is on the form $P(4) \to P(3)$, we can verify that $V_{S(3)}(P(4)[1]) = P(3)$, so \[\overline{\varphi}(P(3) \oplus P(4)) = P(3) \oplus S(3).\]

    Letting $M = P(3) \oplus P(4)$ and $N = P(3) \oplus S(3)$, \cref{thm:TFfullbehaviour} gives that \[\overline{\varphi}_3(T) = E^{-1}_N(E_M(P(1))) \oplus E^{-1}_N(E_M(P(2))) \oplus P(3) \oplus S(3).\]

    It may be verified directly that $E^{-1}_N(E_M(P(1))) = P(1)$ and that $E^{-1}_N(E_M(P(2))) = I(4)$. We can conclude that \[\overline{\varphi}_3(P(1) \oplus P(2) \oplus P(3) \oplus P(4)) = P(1) \oplus I(4) \oplus P(3) \oplus S(3).\]
\end{example}

\begin{example}\label{exmp:radqszeropreproj}
    Let $\L\cong R\A_3$ be the radical square zero quotient of the preprojective algebra of type $\A_3$ whose vertices are labelled by $\{1,2,3\}$ with 2 being the middle vertex. Consider the TF-ordered $\tau$-tilting module $P(3) \oplus P(1) \oplus I(2)$. Since $I(2) \not \in \Gen P(1)$ and $I(2)$ is not projective, \cref{prop:TFadmissiblebehaviour} immediately gives 
    \[ \overline{\varphi}_2(P(3) \oplus P(1) \oplus I(2)) = P(3) \oplus I(2) \oplus P(1).\]
    Now let us consider the $\tau$-tilting reduction with respect to $P(1)$ since $P(3), I(2) \not \in \Gen(P(1))$ we have that $E_{P(1)}(P(3))=f_{P(1)}(P(3)) = P(3)$ and $E_{P(1)}(I(2)) = f_{P(1)}(I(2)) = S(3)$ by \cref{defn:Emap}. Consider the mutation of TF-ordered modules in $J(P(1))$
    \[ \overline{\varphi}^{J(P(1))}(P(3) \oplus S(3)) = V_{P(3)}^{J(P(1))}(S(3)) \oplus P(3) = I(3) \oplus P(3)\]
    where the first equality follows from \cref{prop:TFadmissiblebehaviour} and the second equality follows from the fact that $\bcomp^{J(P(3))}(P(3)) = I(3) \oplus P(3)$ in $J(P(3))$. We obtain that
    \[ \overline{\varphi}_1( P(3) \oplus I(2) \oplus P(1)) = P(2) \oplus P(3) \oplus P(1),\]
    since $f_{P(1)}^{-1}(I(3)) = P(2)$. Now since $P(1)$ is projective, from \cref{prop:TFadmissiblebehaviour} we have
    \[\overline{\varphi}(P(3) \oplus P(1))= V_{E_{P(1)}(P(3))} (P(1)[1]) \oplus E_{P(1)}(P(3)).\]
    Since $P(3) \not \in \Gen P(1)$ we have $E_{P(1)}(P(3)) = f_{P(1)}(P(3)) = P(3)$ and since $P(2) \in \add(\bcomp(P(3)))$ and $P(2)[1] \in \add(\ccomp(P(3)))$ we have 
    \[\overline{\varphi}_2(P(2) \oplus P(3) \oplus P(1))= P(2) \oplus P(1) \oplus P(3)\]
    by \cref{lem:gen_mut}c). Thus we have demonstrated the three cases of \cref{prop:TFadmissiblebehaviour}. 
\end{example}

We now introduce the geometric model for $\tau$-rigid modules of \cite{Adachi2016}. For a given (basic, connected) Nakayama algebra of rank $n$, consider a disk $\Dcal$, as a connected oriented 2-dimensional Riemann surface with boundary, with puncture $\bullet$ and boundary marked point labelled $0$ to $n-1$ in counter-clockwise orientation. 

An \textit{inner arc} $\langle i,j \rangle$ in $\Dcal$ is a path from the marked point $i$ to the marked point $j$ which is homotopic to the boundary path $i,(i+1)_n, \dots, (i+t)_n=j$, where $t\in[2,n]$. For such an inner arc we call $j$ the \textit{terminal point} and $\ell\langle i,j \rangle = t$ the \textit{length} of the arc. A \textit{projective arc} $\langle \bullet, j\rangle$ is a path from the puncture to the marked point $j$. Now, a \textit{triangulation} of $\Dcal$ is a maximal set of pairwise nonhomotopic pairwise non-intersecting inner arcs and projective arcs. Let $\Tcal(n;\ell_1, \dots, \ell_n)$ denote the set of triangulations of $\Dcal$ such that the length of every inner arc with terminal point $j$ is at most $\ell_j$ for all $j \in \{0, \dots, n-1\}$.

Recall that throughout, we have written non-projective $\tau$-rigid modules as $M_{s,t}$ where $s$ is such that $\soc M_{s,t} \cong S(s+2)_n$ and $t$ is such that $\topp M_{s,t} \cong S(t)$. This is because of the following bijection between such modules and inner arcs.

\begin{theorem}\cite[Thm. 2.16]{Adachi2016}\label{thm:diskmutation}
    There is a bijection $\taut(\L) \to \Tcal(n; \ell(P(0)), \dots, \ell(P(n-1)))$ given on indecomposable direct summands as $M_{s,t} \mapsto \langle s, t \rangle$ and $P(i) \mapsto \langle \bullet, i \rangle$.
\end{theorem}

Under the assumption that $\ell(P(i))\geq n$ for all $i \in \{0, \dots, n-1\}$, the above can be extended to all $\tau$-tilting pairs by decorating $\Dcal$ with a sign $\sigma \in \{+,-\}$. In this special case, if a $\tau$-tilting pair has a nonzero shifted projective direct summand, then it does not have any nonzero non-shifted projective direct summands by \cite[Thm. 2.6(2)]{Adachi2016} and the bijection may be extended by defining $P(i)[1] \mapsto (\langle \bullet, (i-1)_n\rangle, -)$. In other words, if the decoration of the disk is $-$, then the projective arcs correspond to a shifted projective module with index shifted by 1. 

\begin{example}
    Let $\L$ be a Nakayma algebra and let $M \cong M_{s,t}$ be indecomposable non-projective, then the Bongartz completion $\bcomp(M)$ was described in \cref{prop:Bongartzcompletion} and its corresponding triangulation can be seen on the left in \cref{fig:completions}. The co-Bongartz completion $\ccomp(M)$ was described in \cref{lem:cocompletion}. If $\ell(P(i)) \geq n$ for all $0\leq i \leq n-1$, then it corresponds with the signed triangulation on the right in \cref{fig:completions}. In this setting the two triangulations of \cref{fig:completions} also visualise the map $V_M: \ccomplement{M} \to \bcomplement{M}$ using \cref{cor:Vmapregular}. More precisely, we have $V_M(M/\rad^i M) \cong \rad^i M$ for all $i \in \{1, \dots, \ell(M)\}$. Moreover, \cref{cor:Vmapregular} implies that if $\Hom(P(j),M) =0$ and $\Hom(P(j),\tau M)=0$, then $V_M(P(j)[1])=P(j)$. In particular, by \crefnakayamalist{list:8} and \crefnakayamalist{list:9} this is the case for $j \in \{ (t+1)_n, \dots, s\}$. Now, there is one direct summand remaining, and since $V_M$ must be a bijection we obtain $V_M(P(s+1)_n[1]) \cong P(t)$.
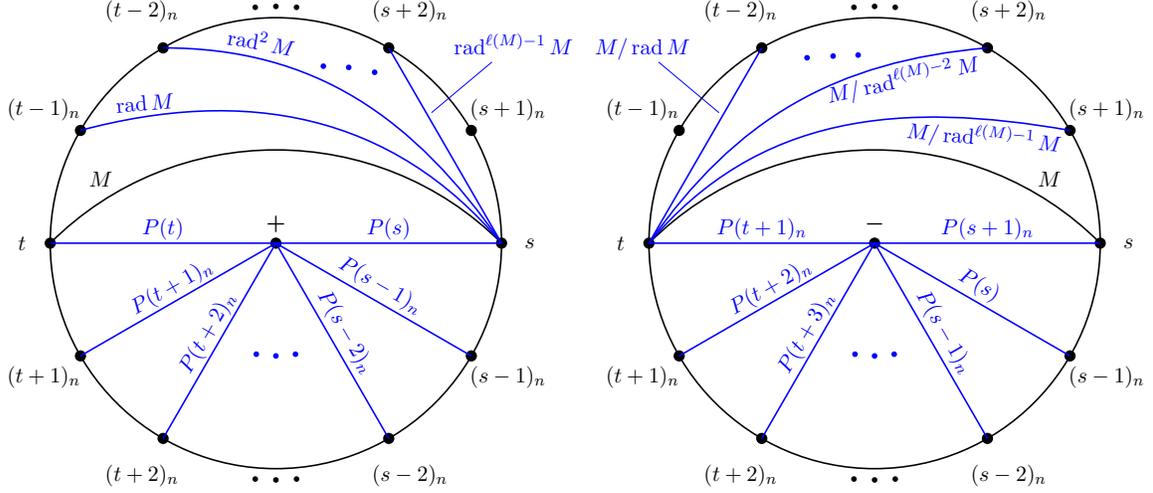
\begin{figure}[ht!]
        \centering
    \begin{subfigure}[b!]{0.45\textwidth}
    \[
    \scalebox{0.75}{
    \begin{tikzpicture}
        \draw[fill=none, thick](0,0) circle (4);
        \node[black] at (90:0.35) [] {$\boldsymbol{+}$};
        \node[black] at (4.5,0) [] {$s$};
        \node[black] at (-4.5,0) [] {$t$};
        \node[black] at (30:4.75) [] {$(s+1)_n$};
        \node[black] at (60:4.75) [] {$(s+2)_n$};
        \node[black] at (120:4.75) [] {$(t-2)_n$};
        \node[black] at (150:4.75) [] {$(t-1)_n$};
        \node[black] at (210:4.75) [] {$(t+1)_n$};
        \node[black] at (240:4.75) [] {$(t+2)_n$};
        \node[black] at (300:4.75) [] {$(s-2)_n$};
        \node[black] at (330:4.75) [] {$(s-1)_n$};
        \node[black] at (160:3.3) [] {$M$};
        \node[blue] at (133:3.4) [rotate=5] {$\rad M$};
        \node[blue] at (94:3.6) [rotate=-10] {$\rad^2 M$};
        \node[blue] at (40:5.5) [] {$\rad^{\ell(M)-1} M$};
        \node[blue] at (-2,0.25) [] {$P(t)$};
        \node[blue] at (2,0.25) [] {$P(s)$};
        \node[blue] at (203:2) [rotate=30] {$P(t+1)_n$};
        \node[blue] at (233:2) [rotate=60] {$P(t+2)_n$};
        \node[blue] at (308:2) [rotate=-60] {$P(s-2)_n$};
        \node[blue] at (338:2) [rotate=-30] {$P(s-1)_n$};
        \draw[-, blue] (40:5) -- (40:3.6);
        \node at (4,0)[circle,fill,inner sep=2pt]{};
        \node at (-4,0)[circle,fill,inner sep=2pt]{};
        \node at (0,0)[circle,fill,inner sep=2pt]{};
        \node at (30:4) [circle,fill,inner sep=2pt]{};
        \node at (60:4) [circle,fill,inner sep=2pt]{};
        \node at (120:4) [circle,fill,inner sep=2pt]{};
        \node at (150:4) [circle,fill,inner sep=2pt]{};
        \node at (210:4) [circle,fill,inner sep=2pt]{};
        \node at (240:4) [circle,fill,inner sep=2pt]{};
        \node at (300:4) [circle,fill,inner sep=2pt]{};
        \node at (330:4) [circle,fill,inner sep=2pt]{};
        \node at (265:4.2) [circle,fill,inner sep=1pt]{};
        \node at (270:4.2) [circle,fill,inner sep=1pt]{};
        \node at (275:4.2) [circle,fill,inner sep=1pt]{};
        \node at (85:4.2) [circle,fill,inner sep=1pt]{};
        \node at (90:4.2) [circle,fill,inner sep=1pt]{};
        \node at (95:4.2) [circle,fill,inner sep=1pt]{};
        \node[blue] at (260:2) [circle,fill,inner sep=1pt]{};
        \node[blue] at (270:2) [circle,fill,inner sep=1pt]{};
        \node[blue] at (280:2) [circle,fill,inner sep=1pt]{};
        \node[blue] at (60:3.5) [circle,fill,inner sep=1pt]{};
        \node[blue] at (67.5:3.375) [circle,fill,inner sep=1pt]{};
        \node[blue] at (75:3.25) [circle,fill,inner sep=1pt]{};
        \draw[-,black, thick] (4,0) to[out=135,in=45] (-4,0);
        \draw[-, blue, thick] (0:0) -- (180:4);
        \draw[-, blue, thick] (0:0) -- (210:4);
        \draw[-, blue, thick] (0:0) -- (240:4);
        \draw[-, blue, thick] (0:0) -- (300:4);
        \draw[-, blue, thick] (0:0) -- (330:4);
        \draw[-, blue, thick] (0:0) -- (0:4);
        \draw[-, blue, thick] (0:4) -- (60:4);
        \draw[-, blue, thick] (0:4) to[out=125,in=0] (120:4);
        \draw[-, blue, thick] (0:4) to[out=130,in=15] (150:4);
    \end{tikzpicture}
    }
    \]
    \end{subfigure}%
    \quad
    \begin{subfigure}[b!]{0.45\textwidth}
    \[
    \scalebox{0.75}{
    \begin{tikzpicture}
        \draw[fill=none, thick](0,0) circle (4);
        \node[black] at (90:0.35) [] {$\boldsymbol{-}$};
        \node[black] at (4.5,0) [] {$s$};
        \node[black] at (-4.5,0) [] {$t$};
        \node[black] at (30:4.75) [] {$(s+1)_n$};
        \node[black] at (60:4.75) [] {$(s+2)_n$};
        \node[black] at (120:4.75) [] {$(t-2)_n$};
        \node[black] at (150:4.75) [] {$(t-1)_n$};
        \node[black] at (210:4.75) [] {$(t+1)_n$};
        \node[black] at (240:4.75) [] {$(t+2)_n$};
        \node[black] at (300:4.75) [] {$(s-2)_n$};
        \node[black] at (330:4.75) [] {$(s-1)_n$};
        \node[black] at (20:3.3) [] {$M$};
        \node[blue] at (140:5.37) [] {$M/\rad M$};
        \node[blue] at (44:2.7) [rotate=-5] {$M/\rad^{\ell(M)-1} M$};
        \node[blue] at (80:2.95) [rotate=12] {$M/\rad^{\ell(M)-2} M$};
        \node[blue] at (-2,0.25) [] {$P(t+1)_n$};
        \node[blue] at (2,0.25) [] {$P(s+1)_n$};
        \node[blue] at (203:2) [rotate=30] {$P(t+2)_n$};
        \node[blue] at (233:2) [rotate=60] {$P(t+3)_n$};
        \node[blue] at (308:2) [rotate=-60] {$P(s-1)_n$};
        \node[blue] at (338:2) [rotate=-30] {$P(s)$};
        \draw[-, blue] (140:5) -- (140:3.6);
        \node at (4,0)[circle,fill,inner sep=2pt]{};
        \node at (-4,0)[circle,fill,inner sep=2pt]{};
        \node at (0,0)[circle,fill,inner sep=2pt]{};
        \node at (30:4) [circle,fill,inner sep=2pt]{};
        \node at (60:4) [circle,fill,inner sep=2pt]{};
        \node at (120:4) [circle,fill,inner sep=2pt]{};
        \node at (150:4) [circle,fill,inner sep=2pt]{};
        \node at (210:4) [circle,fill,inner sep=2pt]{};
        \node at (240:4) [circle,fill,inner sep=2pt]{};
        \node at (300:4) [circle,fill,inner sep=2pt]{};
        \node at (330:4) [circle,fill,inner sep=2pt]{};
        \node at (265:4.2) [circle,fill,inner sep=1pt]{};
        \node at (270:4.2) [circle,fill,inner sep=1pt]{};
        \node at (275:4.2) [circle,fill,inner sep=1pt]{};
        \node at (85:4.2) [circle,fill,inner sep=1pt]{};
        \node at (90:4.2) [circle,fill,inner sep=1pt]{};
        \node at (95:4.2) [circle,fill,inner sep=1pt]{};
        \node[blue] at (260:2) [circle,fill,inner sep=1pt]{};
        \node[blue] at (270:2) [circle,fill,inner sep=1pt]{};
        \node[blue] at (280:2) [circle,fill,inner sep=1pt]{};
        \node[blue] at (95:3.35) [circle,fill,inner sep=1pt]{};
        \node[blue] at (102.5:3.4) [circle,fill,inner sep=1pt]{};
        \node[blue] at (110:3.5) [circle,fill,inner sep=1pt]{};
        \draw[-,black, thick] (4,0) to[out=135,in=45] (-4,0);
        \draw[-, blue, thick] (0:0) -- (180:4);
        \draw[-, blue, thick] (0:0) -- (210:4);
        \draw[-, blue, thick] (0:0) -- (240:4);
        \draw[-, blue, thick] (0:0) -- (300:4);
        \draw[-, blue, thick] (0:0) -- (330:4);
        \draw[-, blue, thick] (0:0) -- (0:4);
        \draw[-, blue, thick] (120:4) -- (180:4);
        \draw[-, blue, thick] (60:4) to[out=185,in=55] (180:4);
        \draw[-, blue, thick] (30:4) to[out=170,in=50] (180:4);
    \end{tikzpicture}
    }
    \]
    \end{subfigure}%
    \caption{Bongartz completion (left) and co-Bongartz completion (right) of an indecomposable $\tau$-rigid $M \cong M_{s,t}$ on Adachi's disk model}
    \label{fig:completions}
    \end{figure}
\end{example}

In fact, we can also use the disk model to visualise the quiver of the $\tau$-tilting reduction of an indecomposable module of a Nakayama algebra. This description improves \cref{prop:Nakayamareduction} and generalises a description given in the proof of \cite[Thm. 6.5]{Msapato2022} for a special class of self-injective Nakayama algebras.

\begin{proposition}\label{prop:reductionquiver}
    Let $\L$ be a basic and connected  Nakayama algebra and $M \cong M_{s,t}$ be indecomposable. Then the underlying quiver of the $\tau$-tilting reduction $\Lambda_M$, with $\mods \Lambda_M$ equivalent to $J(M)$, can be read off the disk model. In particular, the quiver of $\Lambda_M$ is of the form $\overrightarrow{\A}_{\ell(M)-1} \times \overrightarrow{\Delta}_{(s-t)_n}$ with up to three arrows removed from $\overrightarrow{\Delta}_{(s-t)_n}$, as illustrated on the right disk in \cref{fig:seconddisks}. In particular, it is again a Nakayama algebra.
\end{proposition}
\begin{proof}
    The underlying quiver of $\End(\bcomp(M))$ has vertices indexed by the indecomposable direct summands of $\bcomp(M)$ as given in \cref{prop:Bongartzcompletion}. There exists an arrow between two vertices if and only if there is an irreducible morphism between the corresponding direct summands of $\bcomp(M)$. The quiver of the $\tau$-tilting reduction $\Lambda_M$ is obtained from the quiver of $\End(\bcomp(M))$ by deleting the vertex indexed by $M$. From  \cref{prop:Bongartzcompletion} we immediately get a sequence of irreducible morphisms
    \[ \rad^{\ell(M)-1} M \to \dots \to \rad M \to M.\]
    This induces the subquiver $\overrightarrow{\A}_{\ell(M)-1}$ of $\Lambda_M$. By \cref{prop:nakayama-collection-of-facts} we have $\Hom_A(P(i),\rad^j M)=0$ for all $P(i) \in \add \bcomp(B)$ and $k\geq 1$. Conversely, any morphism from $\rad^j M$ to $P(i)$ factors through $M$. Hence there are no arrows connecting the subquiver $\overrightarrow{\A}_{\ell(M)-1}$ of $\Lambda_M$ and the subquiver $\overrightarrow{\Delta}_{(s-t)_n}$ of $\Lambda_M$ induced by the sequence of irreducible morphisms
    \[ P(t) \xrightarrow{a} P(t+1)_n \xrightarrow{} \dots \to P(k) \xrightarrow{b} P(k+1)_n \to \dots \to P(s-1) \to P(s) \xrightarrow{c} P(t)\]
    where the morphism $b$ does not exist if $P(k) = P(n-1)$ and $\L$ is a linear Nakayama algebra and the morphism $c$ exists if and only if $\Hom(P(s),P(t))\neq 0$. If $\ell(M) = \ell(P(t+1)_n)-1$ then the morphism $a: P(t) \to P(t+1)_n$ factors through $M$ and hence does not give rise to an  arrow in the quiver of $\Lambda_M$. In conclusion we get the desired description and the embedding of the quiver of $\Lambda_M$ into the disk with triangulation given by $\bcomp(M)$ is given in the right diagram of \cref{fig:seconddisks}.
\end{proof}

\begin{figure}[ht!]
        \centering
        \begin{subfigure}[b!]{0.5\textwidth}
    \[
   \scalebox{0.75}{
    \begin{tikzpicture}
        \draw[fill=none, thick](0,0) circle (4);
        \node[black] at (4.8,0) [] {$(i+1)_n$};
        \node[black] at (30:4.75) [] {$(i+2)_n$};
        \node[black] at (90:4.5) [] {$(j-k-1)_n$};
        \node[black] at (120:4.75) [] {$(j-k)_n$};
        \node[black] at (150:5.1) [] {$(j-k+1)_n$};
        \node[black] at (210:4.75) [] {$j$};
        \node[black] at (240:4.75) [] {$(j+1)_n$};
        \node[black] at (300:4.75) [] {$(i-1)_n$};
        \node[black] at (330:4.75) [] {$i$};
        
        \node[white] at (285:4.75) [] {$(y+1)_n$};
        \node[black] at (90:1) [] {$B$};
        \node[blue] at (16:5.5) [] {$\rad^{\ell(B)-1} B$};
        \node[orange] at (140:5.7) [] {$(\rad^{k-1}B)/C$};
        \node[black] at (203:2) [rotate=30] {$P(j)$};
        \node[black] at (233:2) [rotate=60] {$P(j+1)_n$};
        \node[black] at (308:2) [rotate=-60] {$P(i-1)_n$};
        \node[black] at (338:2) [rotate=-30] {$P(i)$};
        \node[blue] at (70:3) [rotate=-60] {$\rad^{k+1} B$};
        \node[orange] at (126:2.5) [rotate=60] {$(\rad^0 B)/ C$};
        \draw[-, blue] (6:3.6) -- (15:4.8);
        \draw[-, red] (140:3.8) -- (140:5.3);
        \node at (0:4)[circle,fill,inner sep=2pt]{};
        \node at (0,0)[circle,fill,inner sep=2pt]{};
        \node at (30:4) [circle,fill,inner sep=2pt]{};
        \node at (120:4) [circle,fill,inner sep=2pt]{};
        \node at (150:4) [circle,fill,inner sep=2pt]{};
        \node at (210:4) [circle,fill,inner sep=2pt]{};
        \node at (240:4) [circle,fill,inner sep=2pt]{};
        \node at (300:4) [circle,fill,inner sep=2pt]{};
        \node at (330:4) [circle,fill,inner sep=2pt]{};
        \node at (90:4) [circle,fill,inner sep=2pt]{};
        \node at (265:4.2) [circle,fill,inner sep=1pt]{};
        \node at (270:4.2) [circle,fill,inner sep=1pt]{};
        \node at (275:4.2) [circle,fill,inner sep=1pt]{};
        \node[black] at (260:2) [circle,fill,inner sep=1pt]{};
        \node[black] at (270:2) [circle,fill,inner sep=1pt]{};
        \node[black] at (280:2) [circle,fill,inner sep=1pt]{};
        \node at (60:4.2) [circle,fill,inner sep=1pt]{};
        \node at (65:4.2) [circle,fill,inner sep=1pt]{};
        \node at (55:4.2) [circle,fill,inner sep=1pt]{};
        \node at (180:4.2) [circle,fill,inner sep=1pt]{};
        \node at (175:4.2) [circle,fill,inner sep=1pt]{};
        \node at (185:4.2) [circle,fill,inner sep=1pt]{};
        \node[orange] at ([shift=(90:1cm)]180:3) [circle,fill,inner sep=1pt]{};
        \node[orange] at ([shift=(90:1cm)]190:3) [circle,fill,inner sep=1pt]{};
        \node[orange] at ([shift=(90:1cm)]170:3) [circle,fill,inner sep=1pt]{};
        \node[blue] at ([shift=(300:1.5cm)]45:3) [circle,fill,inner sep=1pt]{};
        \node[blue] at ([shift=(300:1.5cm)]55:3) [circle,fill,inner sep=1pt]{};
        \node[blue] at ([shift=(300:1.5cm)]65:3) [circle,fill,inner sep=1pt]{};
        \draw[-,black, thick] (-30:4) to[out=120,in=0] (90:0.75) to[out=180,in=60] (210:4);
        \draw[-, black, thick] (0:0) -- (210:4);
        \draw[-, black, thick] (0:0) -- (240:4);
        \draw[-, black, thick] (0:0) -- (300:4);
        \draw[-, black, thick] (0:0) -- (330:4);
        \draw[-, blue, thick] (-30:4) -- (30:4);
        \draw[-, blue, thick] (-30:4) -- (90:4);
        \draw[-, orange, thick] (90:4) -- (150:4);
        \draw[-, orange, thick] (90:4) -- (210:4);
    \end{tikzpicture}
    }
    \]
    \end{subfigure}%
    \begin{subfigure}[b!]{0.45\textwidth}
    \[
    \scalebox{0.75}{
    \begin{tikzpicture}
        \draw[fill=none, thick](0,0) circle (4);
        \node[black] at (4.5,0) [] {$s$};
        \node[black] at (-4.5,0) [] {$t$};
        \node[black] at (30:4.75) [] {$(s+1)_n$};
        \node[black] at (60:4.75) [] {$(s+2)_n$};
        \node[black] at (150:4.75) [] {$(t-1)_n$};
        \node[black] at (210:4.75) [] {$(t+1)_n$};
        \node[black] at (255:4.75) [] {$y$};
        \node[black] at (285:4.75) [] {$(y+1)_n$};
        \node[black] at (330:4.75) [] {$(s-1)_n$};
        
        \node[white] at (150:5.1) [] {$(j-k+1)_n$};
        \node at (4,0)[circle,fill,inner sep=2pt]{};
        \node at (-4,0)[circle,fill,inner sep=2pt]{};
        \node at (0,0)[circle,fill,inner sep=2pt]{};
        \node at (30:4) [circle,fill,inner sep=2pt]{};
        \node at (60:4) [circle,fill,inner sep=2pt]{};
        \node at (150:4) [circle,fill,inner sep=2pt]{};
        \node at (210:4) [circle,fill,inner sep=2pt]{};
        \node at (255:4) [circle,fill,inner sep=2pt]{};
        \node at (285:4) [circle,fill,inner sep=2pt]{};
        \node at (330:4) [circle,fill,inner sep=2pt]{};
        \node at (232.5:4.2) [circle,fill,inner sep=1pt]{};
        \node at (227.5:4.2) [circle,fill,inner sep=1pt]{};
        \node at (237.5:4.2) [circle,fill,inner sep=1pt]{};
        \node at (312.5:4.2) [circle,fill,inner sep=1pt]{};
        \node at (302.5:4.2) [circle,fill,inner sep=1pt]{};
        \node at (307.5:4.2) [circle,fill,inner sep=1pt]{};
        \node[orange] at (230:3) [circle,fill,inner sep=1pt]{};
        \node[orange] at (232.5:3) [circle,fill,inner sep=1pt]{};
        \node[orange] at (235:3) [circle,fill,inner sep=1pt]{};
        \node[orange] at (305:3) [circle,fill,inner sep=1pt]{};
        \node[orange] at (307.5:3) [circle,fill,inner sep=1pt]{};
        \node[orange] at (310:3) [circle,fill,inner sep=1pt]{};
        \node[orange] at ([shift=(285:1.8cm)]97.5:5) [circle,fill,inner sep=1pt]{};
        \node[orange] at ([shift=(285:1.8cm)]94.5:5) [circle,fill,inner sep=1pt]{};
        \node[orange] at ([shift=(285:1.8cm)]100.5:5) [circle,fill,inner sep=1pt]{};
        \draw[-,black, thick] (4,0) to[out=135,in=45] (-4,0);
        \draw[-, blue, thick] (0:0) -- (180:4);
        \draw[-, blue, thick] (0:0) -- (210:4);
        \draw[-, blue, thick] (0:0) -- (255:4);
        \draw[-, blue, thick] (0:0) -- (285:4);
        \draw[-, blue, thick] (0:0) -- (330:4);
        \draw[-, blue, thick] (0:0) -- (0:4);
        \draw[-, blue, thick] (0:4) -- (60:4);
        \draw[-, blue, thick] (0:4) to[out=130,in=15] (150:4);
        \draw[->, orange, thick, dashed] (182:3) arc (182:208:3);
        \draw[->, orange, thick] (212:3) arc (212:227:3);
        \draw[->, orange, thick] (238:3) arc (238:253:3);
        \draw[->, orange, thick, dashed] (257:3) arc (257:283:3);
        \draw[->, orange, thick] (287:3) arc (287:302:3);
        \draw[->, orange, thick] (313:3) arc (313:328:3);
        \draw[->, orange, thick] (332:3) arc (332:358:3);
        \draw[->, orange, thick, dashed] ([shift=(270:2.3cm)]40:3.75) arc (40:140:3.75);
        \draw[->, orange, thick] ([shift=(285:1.8cm)]70:5) arc (70:92.5:5);
        \draw[->, orange, thick] ([shift=(285:1.8cm)]102.5:5) arc (102.5:125:5);
        \node[orange] at (180:3) [circle,fill,inner sep=1.5pt]{};
        \node[orange] at (210:3) [circle,fill,inner sep=1.5pt]{};
        \node[orange] at (255:3) [circle,fill,inner sep=1.5pt]{};
        \node[orange] at (285:3) [circle,fill,inner sep=1.5pt]{};
        \node[orange] at (330:3) [circle,fill,inner sep=1.5pt]{};
        \node[orange] at (0:3) [circle,fill,inner sep=1.5pt]{};
        \node[orange] at ([shift=(285:1.8cm)]128.5:5) [circle,fill,inner sep=1.5pt]{};
        \node[orange] at ([shift=(285:1.8cm)]68:5) [circle,fill,inner sep=1.5pt]{};
        \node[orange] at (195:3.3) [] {$a$};
        \node[orange] at (90:1.2) [] {$c$};
        \node[orange] at (270:3.3) [] {$b$};
    \end{tikzpicture}
    }
    \]
    \end{subfigure}%
    \caption{The module $\widehat{B}$ of \cref{thm:Nakayamacase4} (left) and the quiver (in orange) of $\End(\bcomp(M))/\langle M \rangle$ embedded into Adachi's disk model (right) }
    \label{fig:seconddisks}
    \end{figure}
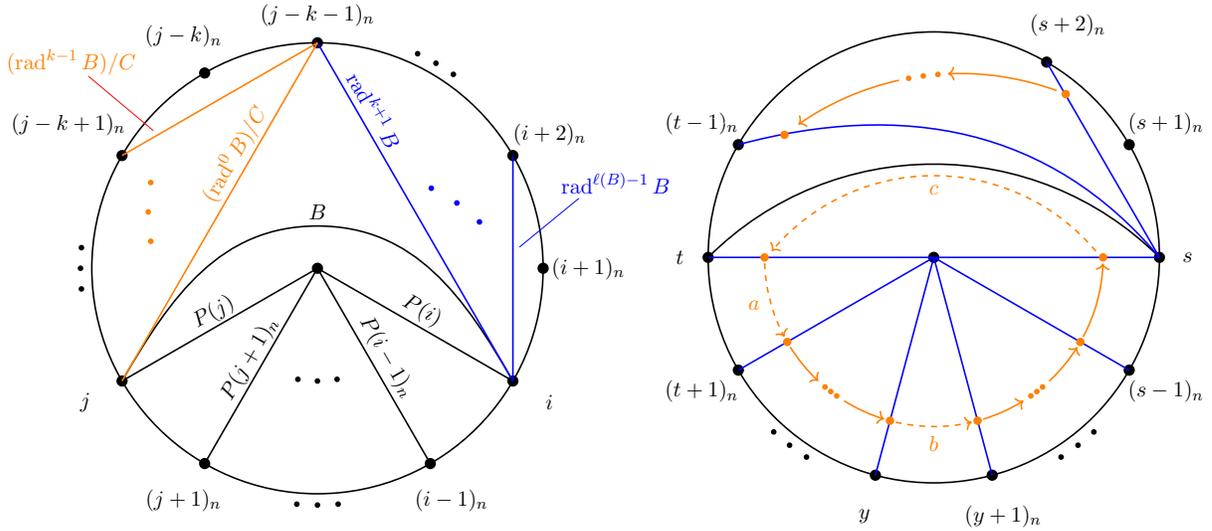

The following example demonstrates the absence of the three arrows $a,b,c$ in \cref{prop:reductionquiver}. 

\begin{example} 
    Let $\L \cong K(5 \xrightarrow{a} 4 \xrightarrow{b} 3 \xrightarrow{c} 2 \xrightarrow{d} 1) / \langle dc, cba \rangle$. Let $M \cong {}_{3}^{4}$, then one can check easily that $J(M) \cong A_1 \times A_1 \times A_1 \times A_1 \cong K^4$ is the direct product of four Nakayama algebras of rank 1. In this case none of the arrows $a,b,c$ in the right disk of \cref{fig:seconddisks} exist.
\end{example}

Similar to \cref{exmp:radqszeropreproj}, let us now describe the different cases of the mutation of a TF-ordered module $B \oplus C$ described in \cref{thm:introNakmutation}. When illustrating TF-ordered modules on the disk model, we indicate the order of the modules with small indices next to the arcs and provide a visual distinction using different colours or different line styles.
\begin{example}
    Let $\L \cong K\overrightarrow{\Delta}_4/R^3$ be the radical cube zero Nakayama algebra whose underlying quiver is cyclic with 4 vertices. Then \cref{fig:phiorbit1} and \cref{fig:phiorbit2} show two $\overline{\varphi}$-orbits of TF-ordered module $B \oplus C$ of $\L$ which cover all six cases of \cref{thm:introNakmutation}. 

    \begin{figure}[ht!]
        \centering
        \scalebox{0.95}{
        \begin{tikzpicture}
        \draw[fill=none, thick](0,0) circle (2);
        \draw[fill=none, thick](-5.5,0) circle (2);
        \draw[fill=none, thick](5.5,0) circle (2);
        \draw[->, black, thick] (-3.3,0) -- (-2.2,0);
        \node[black] at (0,-2.3) [] {\small$TF-1a$};
        \draw[->, black, thick] (2.2,0) -- (3.3,0);
        \node[black] at (5.5,-2.3) [] {\small$TF-1b$};
        \draw[-, black, thick] (7.7,0) -- (8.25,0);
        \draw[->, black, thick] (-8.25,0) -- (-7.7,0);
        \node[black] at (-5.5,-2.3) [] {\small$TF-2b$};
        \node[black] at (2.75,0.2) [] {\small$\overline{\varphi}$};
        \node[black] at (-2.75,0.2) [] {\small$\overline{\varphi}$};
        \node[black] at (7.975,0.2) [] {\small$\overline{\varphi}$};
        \draw[white] (-5.5,0)[]{} -- ++ (45:2.1) node[right,black]{0};
        \draw[white] (-5.5,0)[]{} -- ++ (45:2) node[circle,fill,inner sep=1.5pt, black]{};
        \draw[white] (-5.5,0)[]{} -- ++ (135:2.1) node[left,black]{1};
        \draw[white] (-5.5,0)[]{} -- ++ (135:2) node[circle,fill,inner sep=1.5pt, black]{};
        \draw[white] (-5.5,0)[]{} -- ++ (225:2.1) node[left,black]{2};
        \draw[white] (-5.5,0)[]{} -- ++ (225:2) node[circle,fill,inner sep=1.5pt, black]{};
        \draw[white] (-5.5,0)[]{} -- ++ (315:2.1) node[right,black]{3};
        \draw[white] (-5.5,0)[]{} -- ++ (315:2) node[circle,fill,inner sep=1.5pt, black]{};
        \draw[white] (0,0)[]{} -- ++ (45:2.1) node[right,black]{0};
        \draw[white] (0,0)[]{} -- ++ (45:2) node[circle,fill,inner sep=1.5pt, black]{};
        \draw[white] (0,0)[]{} -- ++ (135:2.1) node[left,black]{1};
        \draw[white] (0,0)[]{} -- ++ (135:2) node[circle,fill,inner sep=1.5pt, black]{};
        \draw[white] (0,0)[]{} -- ++ (225:2.1) node[left,black]{2};
        \draw[white] (0,0)[]{} -- ++ (225:2) node[circle,fill,inner sep=1.5pt, black]{};
        \draw[white] (0,0)[]{} -- ++ (315:2.1) node[right,black]{3};
        \draw[white] (0,0)[]{} -- ++ (315:2) node[circle,fill,inner sep=1.5pt, black]{};
        \draw[white] (5.5,0)[]{} -- ++ (45:2.1) node[right,black]{0};
        \draw[white] (5.5,0)[]{} -- ++ (45:2) node[circle,fill,inner sep=1.5pt, black]{};
        \draw[white] (5.5,0)[]{} -- ++ (135:2.1) node[left,black]{1};
        \draw[white] (5.5,0)[]{} -- ++ (135:2) node[circle,fill,inner sep=1.5pt, black]{};
        \draw[white] (5.5,0)[]{} -- ++ (225:2.1) node[left,black]{2};
        \draw[white] (5.5,0)[]{} -- ++ (225:2) node[circle,fill,inner sep=1.5pt, black]{};
        \draw[white] (5.5,0)[]{} -- ++ (315:2.1) node[right,black]{3};
        \draw[white] (5.5,0)[]{} -- ++ (315:2) node[circle,fill,inner sep=1.5pt, black]{};
        \node at (-5.5,0)[circle,fill,inner sep=1.5pt]{};
        \node at (0,0)[circle,fill,inner sep=1.5pt]{};
        \node at (5.5,0)[circle,fill,inner sep=1.5pt]{};
        \draw[-, blue, thick] (5.5,0) [draw] ++ (225:2) to[out=0,in=270] ++ (45:4);
        \node at (6.5,-0.5)[rotate=50,blue]{$S(0)$};
        \draw[-, orange, thick] (5.5,0) -- ++ (45:2) node[pos =0.5, rotate=45,yshift=2.5mm]{$P(0)$};
        \draw[-, blue, thick] (0,0) -- ++ (45:2) node[pos =0.5, rotate=45,yshift=2.5mm]{$P(0)$};
        \draw[-, orange, thick] (0,0) -- ++ (315:2) node[pos =0.5, rotate=315,yshift=2.5mm]{$P(3)$};
        \draw[-, orange, thick] (-5.5,0) -- ++ (45:2) node[pos =0.5, rotate=45,yshift=2.5mm]{$P(0)$};
        \draw[-, blue, thick] (-5.5,0) -- ++ (315:2) node[pos =0.5, rotate=315,yshift=2.5mm]{$P(3)$};
        \node at (0.1,-0.325)[black]{\tiny{1}};
        \node at (0.1,0.325)[black]{\tiny{2}};
        \node at (-5.4,-0.325)[black]{\tiny{2}};
        \node at (-5.4,0.325)[black]{\tiny{1}};
        \node at (5.6,0.325)[black]{\tiny{1}};
        \node at (4.35,-1.2)[black]{\tiny{2}};
        \end{tikzpicture}}
        \caption{The $\overline{\varphi}$-orbit of a TF-ordered $\tau$-rigid module $B \oplus C$ illustrating three cases of \cref{thm:introNakmutation}. The left direct summand corresponds to the orange arc and the right direct summand corresponds to the blue arc.}
        \label{fig:phiorbit1}
    \end{figure}
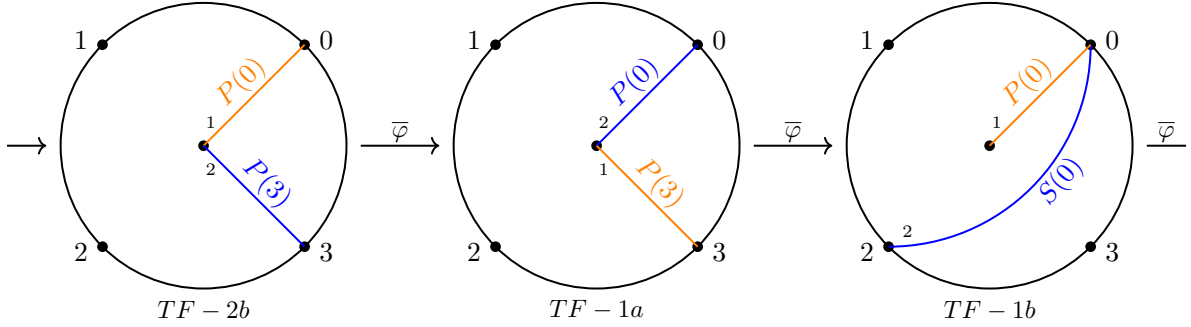
    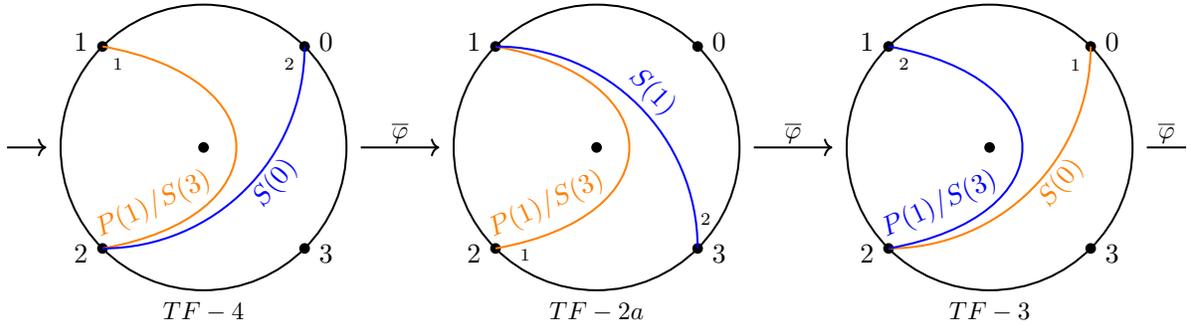
\begin{figure}[ht!]
        \centering
        \scalebox{0.95}{
        \begin{tikzpicture}
        \draw[fill=none, thick](0,0) circle (2);
        \draw[fill=none, thick](-5.5,0) circle (2);
        \draw[fill=none, thick](5.5,0) circle (2);
        \draw[->, black, thick] (-3.3,0) -- (-2.2,0);
        \node[black] at (0,-2.3) [] {\small$TF-2a$};
        \draw[->, black, thick] (2.2,0) -- (3.3,0);
        \node[black] at (5.5,-2.3) [] {\small$TF-3$};
        \draw[-, black, thick] (7.7,0) -- (8.25,0);
        \draw[->, black, thick] (-8.25,0) -- (-7.7,0);
        \node[black] at (-5.5,-2.3) [] {\small$TF-4$};
        \node[black] at (2.75,0.2) [] {\small$\overline{\varphi}$};
        \node[black] at (-2.75,0.2) [] {\small$\overline{\varphi}$};
        \node[black] at (7.975,0.2) [] {\small$\overline{\varphi}$};
        \draw[white] (-5.5,0)[]{} -- ++ (45:2.1) node[right,black]{0};
        \draw[white] (-5.5,0)[]{} -- ++ (45:2) node[circle,fill,inner sep=1.5pt, black]{};
        \draw[white] (-5.5,0)[]{} -- ++ (135:2.1) node[left,black]{1};
        \draw[white] (-5.5,0)[]{} -- ++ (135:2) node[circle,fill,inner sep=1.5pt, black]{};
        \draw[white] (-5.5,0)[]{} -- ++ (225:2.1) node[left,black]{2};
        \draw[white] (-5.5,0)[]{} -- ++ (225:2) node[circle,fill,inner sep=1.5pt, black]{};
        \draw[white] (-5.5,0)[]{} -- ++ (315:2.1) node[right,black]{3};
        \draw[white] (-5.5,0)[]{} -- ++ (315:2) node[circle,fill,inner sep=1.5pt, black]{};
        \draw[white] (0,0)[]{} -- ++ (45:2.1) node[right,black]{0};
        \draw[white] (0,0)[]{} -- ++ (45:2) node[circle,fill,inner sep=1.5pt, black]{};
        \draw[white] (0,0)[]{} -- ++ (135:2.1) node[left,black]{1};
        \draw[white] (0,0)[]{} -- ++ (135:2) node[circle,fill,inner sep=1.5pt, black]{};
        \draw[white] (0,0)[]{} -- ++ (225:2.1) node[left,black]{2};
        \draw[white] (0,0)[]{} -- ++ (225:2) node[circle,fill,inner sep=1.5pt, black]{};
        \draw[white] (0,0)[]{} -- ++ (315:2.1) node[right,black]{3};
        \draw[white] (0,0)[]{} -- ++ (315:2) node[circle,fill,inner sep=1.5pt, black]{};
        \draw[white] (5.5,0)[]{} -- ++ (45:2.1) node[right,black]{0};
        \draw[white] (5.5,0)[]{} -- ++ (45:2) node[circle,fill,inner sep=1.5pt, black]{};
        \draw[white] (5.5,0)[]{} -- ++ (135:2.1) node[left,black]{1};
        \draw[white] (5.5,0)[]{} -- ++ (135:2) node[circle,fill,inner sep=1.5pt, black]{};
        \draw[white] (5.5,0)[]{} -- ++ (225:2.1) node[left,black]{2};
        \draw[white] (5.5,0)[]{} -- ++ (225:2) node[circle,fill,inner sep=1.5pt, black]{};
        \draw[white] (5.5,0)[]{} -- ++ (315:2.1) node[right,black]{3};
        \draw[white] (5.5,0)[]{} -- ++ (315:2) node[circle,fill,inner sep=1.5pt, black]{};
        \node at (-5.5,0)[circle,fill,inner sep=1.5pt]{};
        \node at (0,0)[circle,fill,inner sep=1.5pt]{};
        \node at (5.5,0)[circle,fill,inner sep=1.5pt]{};
        \draw[-, orange, thick] (5.5,0) [draw] ++ (225:2) to[out=0,in=270] ++ (45:4);
        \node at (6.5,-0.5)[rotate=50,orange]{$S(0)$};
        \draw[-, blue, thick] (5.5,0) [draw] ++ (225:2) to[out=10,in=-10,looseness=2.3] ++ (90:2.82842712474);
        \node at (4.8,-0.8)[rotate=25,blue]{$P(1)/S(3)$};
        \draw[-, orange, thick] (0.0,0) [draw] ++ (225:2) to[out=10,in=-10,looseness=2.3] ++ (90:2.82842712474);
        \draw[-, blue, thick] (315:2) to[out=90,in=0] (135:2) node[pos=0.5, rotate=-45, yshift=11mm]{$S(1)$};
        \node at (-0.7,-0.8)[rotate=25,orange]{$P(1)/S(3)$};
        \draw[-, orange, thick] (-5.5,0) [draw] ++ (225:2) to[out=10,in=-10,looseness=2.3] ++ (90:2.82842712474);
        \node at (-6.2,-0.8)[rotate=25,orange]{$P(1)/S(3)$};
        \draw[-, blue, thick] (-5.5,0) [draw] ++ (225:2) to[out=0,in=270] ++ (45:4);
        \node at (-4.5,-0.5)[rotate=50,blue]{$S(0)$};
        \node at (-1,-1.5)[black]{\tiny{1}};
        \node at (1.525,-1)[black]{\tiny{2}};
        \node at (4.3,1.175)[black]{\tiny{2}};
        \node at (6.7,1.15)[black]{\tiny{1}};
        \node at (-6.7,1.175)[black]{\tiny{1}};
        \node at (-4.3,1.175)[black]{\tiny{2}};
        \end{tikzpicture}}
        \caption{The $\overline{\varphi}$-orbit of a TF-ordered $\tau$-rigid module $B \oplus C$ illustrating the remaining three cases of \cref{thm:introNakmutation}. The left direct summand corresponds to the orange arc and the right direct summand corresponds to the blue arc.}
        \label{fig:phiorbit2}
    \end{figure}
\end{example}

For an algebra $\L$ of rank 2, the $\varphi$-orbit admits a nice description via the poset of $\tau$-tilting pairs, as given in \cite[Thm. 7.8]{BHM2024}. We recall that if $\L$ is of rank 2, then the Hasse quiver of the poset $\stt(\L)$ has a largest element $(\L,0)$ and a least element $(0,\L)$ which are connected by two disjoint, possibily infinite, chains. The result \cite[Thm. 7.8]{BHM2024} states that Hasse quiver of the $\varphi$-orbit is obtained from $\Hasse(\stt \L)$ by removing the top and bottom vertices of $\Hasse(\stt \L)$ and introducing two new ``diagonal'' arrows, from the bottom element of each chain to the top element of the other, such that the quiver becomes a cycle.

An intuitive generalisation of this to algebras of higher rank would be to expect that the length of any $\varphi$-orbit is bounded by the sum of the sides of $\Hasse(\Wcal)$ for any functorially finite wide subcategory $\Wcal \subseteq \mods \L$ of rank 2. However, as we will exhibit in the following example, this is not the case even for a Nakayama algebra in rank 3. 

\begin{example}
    Let $\L\cong KQ/I$ be the Nakayama whose underlying quiver is
    \[ \begin{tikzcd}
        Q: 2 \arrow[r, "a", swap] & 1 \arrow[r, "b", swap] & 0 \arrow[ll, bend right, "c", swap]
    \end{tikzcd}\]
    and where $I = \langle cb, ac \rangle$. Let $\Wcal$ be a wide subcategory of $\L$ of rank 2, and let $\Hasse(\stt \Wcal)$ be the Hasse diagram of the poset of $\tau_{\Wcal}$-tilting pairs. It is easily checked via direct calculation that the sum of lengths of the distinct chains connecting the top and bottom element of $\Hasse(\stt \Wcal)$ is at most three. However, the TF-ordered $\tau$-rigid module $P(2) \oplus P(0)$ has a $\overline{\varphi}$-orbit of length 4, which is displayed in \cref{fig:phiorbitlength4}. This demonstrates that a straight-forward generalisation of \cite[Thm. 7.8]{BHM2024} as described above fails.
    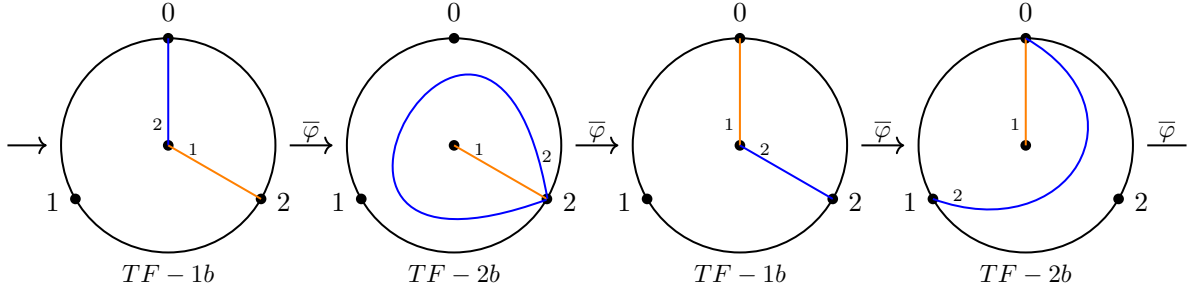
\begin{figure}[ht!]
        \centering
        \[
        \scalebox{0.95}{
        \begin{tikzpicture}
        \clip (-8.5,-2) rectangle (8.5, 2);
        \draw[fill=none, thick](6,0) circle (1.5);
        \draw[fill=none, thick](2,0) circle (1.5);
        \draw[fill=none, thick](-2,0) circle (1.5);
        \draw[fill=none, thick](-6,0) circle (1.5);
        \draw[->, black, thick] (-4.3,0) -- (-3.7,0);
        \draw[->, black, thick] (-0.3,0) -- (0.3,0);
        \draw[->, black, thick] (3.7,0) -- (4.3,0);
        \draw[-, black, thick] (7.7,0) -- (8.25,0);
        \draw[->, black, thick] (-8.25,0) -- (-7.7,0);
        \node[black] at (-6,-1.8) [] {\small$TF-1b$};
        \node[black] at (-2,-1.8) [] {\small$TF-2b$};
        \node[black] at (2,-1.8) [] {\small$TF-1b$};
        \node[black] at (6,-1.8) [] {\small$TF-2b$};
        \node[black] at (0,0.2) [] {\small$\overline{\varphi}$};
        \node[black] at (-4,0.2) [] {\small$\overline{\varphi}$};
        \node[black] at (4,0.2) [] {\small$\overline{\varphi}$};
        \node[black] at (7.975,0.2) [] {\small$\overline{\varphi}$};
        \draw[white] (-6,0)[]{} -- ++ (90:1.6) node[above,black]{0};
        \draw[white] (-6,0)[]{} -- ++ (90:1.5) node[circle,fill,inner sep=1.5pt, black]{};
        \draw[white] (-6,0)[]{} -- ++ (210:1.6) node[left,black]{1};
        \draw[white] (-6,0)[]{} -- ++ (210:1.5) node[circle,fill,inner sep=1.5pt, black]{};
        \draw[white] (-6,0)[]{} -- ++ (330:1.6) node[right,black]{2};
        \draw[white] (-6,0)[]{} -- ++ (330:1.5) node[circle,fill,inner sep=1.5pt, black]{};
        \draw[white] (-2,0)[]{} -- ++ (90:1.6) node[above,black]{0};
        \draw[white] (-2,0)[]{} -- ++ (90:1.5) node[circle,fill,inner sep=1.5pt, black]{};
        \draw[white] (-2,0)[]{} -- ++ (210:1.6) node[left,black]{1};
        \draw[white] (-2,0)[]{} -- ++ (210:1.5) node[circle,fill,inner sep=1.5pt, black]{};
        \draw[white] (-2,0)[]{} -- ++ (330:1.6) node[right,black]{2};
        \draw[white] (-2,0)[]{} -- ++ (330:1.5) node[circle,fill,inner sep=1.5pt, black]{};
        \draw[white] (2,0)[]{} -- ++ (90:1.6) node[above,black]{0};
        \draw[white] (2,0)[]{} -- ++ (90:1.5) node[circle,fill,inner sep=1.5pt, black]{};
        \draw[white] (2,0)[]{} -- ++ (210:1.6) node[left,black]{1};
        \draw[white] (2,0)[]{} -- ++ (210:1.5) node[circle,fill,inner sep=1.5pt, black]{};
        \draw[white] (2,0)[]{} -- ++ (330:1.6) node[right,black]{2};
        \draw[white] (2,0)[]{} -- ++ (330:1.5) node[circle,fill,inner sep=1.5pt, black]{};
        \draw[white] (6,0)[]{} -- ++ (90:1.6) node[above,black]{0};
        \draw[white] (6,0)[]{} -- ++ (90:1.5) node[circle,fill,inner sep=1.5pt, black]{};
        \draw[white] (6,0)[]{} -- ++ (210:1.6) node[left,black]{1};
        \draw[white] (6,0)[]{} -- ++ (210:1.5) node[circle,fill,inner sep=1.5pt, black]{};
        \draw[white] (6,0)[]{} -- ++ (330:1.6) node[right,black]{2};
        \draw[white] (6,0)[]{} -- ++ (330:1.5) node[circle,fill,inner sep=1.5pt, black]{};
        \node at (-6,0)[circle,fill,inner sep=1.5pt]{};
        \node at (-2,0)[circle,fill,inner sep=1.5pt]{};
        \node at (2,0)[circle,fill,inner sep=1.5pt]{};
        \node at (6,0)[circle,fill,inner sep=1.5pt]{};
        \draw[-, orange, thick] (6,0) --++ (90:1.5);
        \draw[-, blue, thick] (6,0) [draw] ++ (90:1.5) to[out=330,in=340, looseness=2] ++ (240:2.598076211353316);
        \draw[-, orange, thick] (2,0) --++ (90:1.5);
        \draw[-, blue, thick] (2,0) --++ (330:1.5);
        \draw[-, orange, thick] (-2,0) --++ (330:1.5);
        \draw[-, blue, thick] (-2,0) [draw] ++ (329.5:1.5) to[out=200,in=100, looseness=600] ++ (45:0.02);
        \draw[-, blue, thick] (-6,0) --++ (90:1.5);
        \draw[-, orange, thick] (-6,0) --++ (330:1.5);
        \node at (1.85,0.3)[black]{\tiny{1}};
        \node at (2.35,-0.05)[black]{\tiny{2}};
        \node at (-1.65,-0.05)[black]{\tiny{1}};
        \node at (-0.7,-0.15)[black]{\tiny{2}};
        \node at (5.85,0.3)[black]{\tiny{1}};
        \node at (5.05,-0.7)[black]{\tiny{2}};
        \node at (-5.65,-0.05)[black]{\tiny{1}};
        \node at (-6.15,0.3)[black]{\tiny{2}};
        \end{tikzpicture}
        } \]
        \caption{A $\varphi$-orbit of length 4, which cannot come from applying \cite[Thm. 7.8]{BHM2024} to any functorially finite wide subcategory of rank 2. The left direct summand corresponds to the orange arc and the right direct summands corresponds to the blue arc.}
        \label{fig:phiorbitlength4}
    \end{figure}
\end{example}

Parallel to \cite[Exmp. 8.3]{BHM2024}, where it was shown for the radical square zero quotient of the linear Nakayama algebra $K\overrightarrow{\A}_3$, the mutation of $\tau$-exceptional sequences does not satisfy the Braid group relations, we show the same for the radical square zero quotient of the cyclic Nakayama algebra with underlying quiver $\overrightarrow{\Delta}_3$.

\begin{example}\label{exmp:nonbraid}
    Let $\L \cong K\overrightarrow{\Delta}_3/R^2$ be the radical square zero cyclic Nakayama algebra on three vertices. We illustrate in  \cref{fig:Braidrelations} that the Braid relations are not satisfied by the mutation of TF-ordered $\tau$-tilting modules.
    \begin{figure}[ht!]
        \centering
        \[
    \scalebox{0.95}{
    \begin{tikzpicture}
        \draw[fill=none, thick](0,0) circle (2);
        \draw[fill=none, thick](-6,0) circle (2);
        \draw[fill=none, thick](6,0) circle (2);
        \draw[->, black, thick] (-2.2,0) -- (-3.8,0);
        \draw[->, black, thick] (2.2,0) -- (3.8,0);
        \node[black] at (2.9,0.2) [] {\small$\overline{\varphi}_2\overline{\varphi}_1\overline{\varphi}_2$};
        \node[black] at (-2.9,0.2) [] {\small$\overline{\varphi}_1\overline{\varphi}_2\overline{\varphi}_1$};
        \draw[white] (-6,0)[]{} -- ++ (90:2.1) node[above,black]{0};
        \draw[white] (-6,0)[]{} -- ++ (90:2) node[circle,fill,inner sep=1.5pt, black]{};
        \draw[white] (-6,0)[]{} -- ++ (210:2.1) node[left,black]{1};
        \draw[white] (-6,0)[]{} -- ++ (210:2) node[circle,fill,inner sep=1.5pt, black]{};
        \draw[white] (-6,0)[]{} -- ++ (330:2.1) node[right,black]{2};
        \draw[white] (-6,0)[]{} -- ++ (330:2) node[circle,fill,inner sep=1.5pt, black]{};
        \draw[white] (0,0)[]{} -- ++ (90:2.1) node[above,black]{0};
        \draw[white] (0,0)[]{} -- ++ (90:2) node[circle,fill,inner sep=1.5pt, black]{};
        \draw[white] (0,0)[]{} -- ++ (210:2.1) node[left,black]{1};
        \draw[white] (0,0)[]{} -- ++ (210:2) node[circle,fill,inner sep=1.5pt, black]{};
        \draw[white] (0,0)[]{} -- ++ (330:2.1) node[right,black]{2};
        \draw[white] (0,0)[]{} -- ++ (330:2) node[circle,fill,inner sep=1.5pt, black]{};
        \draw[white] (6,0)[]{} -- ++ (90:2.1) node[above,black]{0};
        \draw[white] (6,0)[]{} -- ++ (90:2) node[circle,fill,inner sep=1.5pt, black]{};
        \draw[white] (6,0)[]{} -- ++ (210:2.1) node[left,black]{1};
        \draw[white] (6,0)[]{} -- ++ (210:2) node[circle,fill,inner sep=1.5pt, black]{};
        \draw[white] (6,0)[]{} -- ++ (330:2.1) node[right,black]{2};
        \draw[white] (6,0)[]{} -- ++ (330:2) node[circle,fill,inner sep=1.5pt, black]{};
        \node at (-6,0)[circle,fill,inner sep=1.5pt]{};
        \node at (0,0)[circle,fill,inner sep=1.5pt]{};
        \node at (6,0)[circle,fill,inner sep=1.5pt]{};
        \draw[-, blue, thick] (6,0) -- ++ (90:2) node[pos=0.5, rotate=270, yshift=3mm]{$P(0)$};
        \draw[-, blue, thick, dashed] (6,0) -- ++ (330:2) node[pos=0.5, rotate=330, yshift=3mm]{$P(2)$};
        \draw[-, blue, thick, dotted] (6,0) [draw] ++ (90:2) to[out=225,in=195, looseness=2] ++ (300:3.4641016151377545);
        \draw[-, blue, thick, dashed] (0,0) -- ++ (210:2) node[pos=0.5, rotate=30, yshift=-3mm]{$P(1)$};
        \draw[-, blue, thick] (0,0) -- ++ (330:2) node[pos=0.5, rotate=330, yshift=-3mm]{$P(2)$};
        \draw[-,blue, thick, dotted] (330:2) to[out=105,in=75, looseness=2] (210:2) node[pos=0.5, yshift=12mm]{$S(1)$};
        \draw[-, blue, thick] (-6,0) -- ++ (90:2) node[pos=0.5, rotate=270, yshift=3mm]{$P(0)$};
        \draw[-, blue, thick,dashed] (-6,0) -- ++ (210:2) node[pos=0.5, rotate=30, yshift=3mm]{$P(1)$};
        \draw[-, blue, thick, dotted] (-6,0) -- ++ (330:2) node[pos=0.5, rotate=330, yshift=3mm]{$P(2)$};
        \node at (5.3,-1)[rotate=325,blue]{$S(2)$};
        \node at (5.85,0.3)[black]{\tiny{1}};
        \node at (6.35,-0.025)[black]{\tiny{2}};
        \node at (5.475,1.7)[black]{\tiny{3}};
        \node at (-6.15,0.3)[black]{\tiny{1}};
        \node at (-6.2,-0.325)[black]{\tiny{2}};
        \node at (-5.65,-0.025)[black]{\tiny{3}};
        \node at (0.35,-0.025)[black]{\tiny{1}};
        \node at (-0.2,-0.325)[black]{\tiny{2}};
        \node at (1.7,-0.325)[black]{\tiny{3}};
        \end{tikzpicture}
        } \]
        \caption{An example of a cyclic Nakayama algebra for which the mutation of TF-ordered $\tau$-tilting modules does not satisfy the Braid relations. The solid arc corresponds to the left-most direct summand, the dashed one to the middle one and the dotted one to the right-most one. }
        \label{fig:Braidrelations}
    \end{figure}
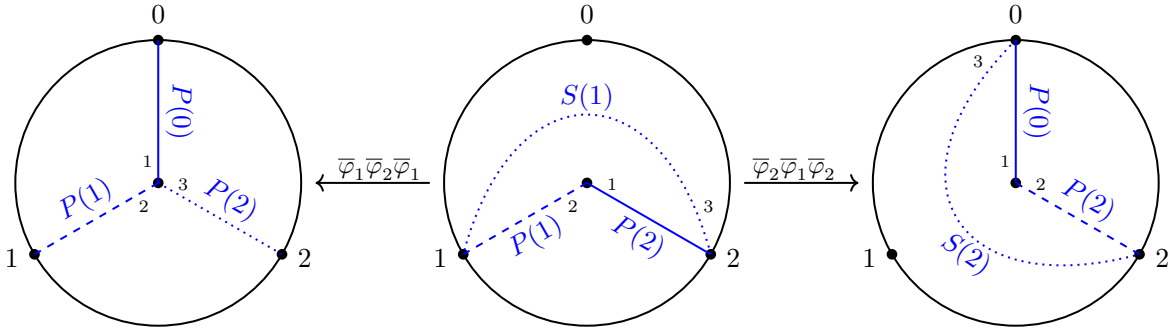
\end{example}

\printbibliography

\end{document}